\documentclass[cmp]{svjour}
\usepackage{graphics}
\usepackage{epsfig}
\usepackage{epstopdf}
\usepackage{subfigure}
\usepackage{caption}
\usepackage{longtable,booktabs}
\usepackage{amssymb,amsfonts}
\usepackage{refstyle}
\usepackage{xcolor}

\journalname{Communications in Mathematical Physics}
\newcommand{\mod}{\textrm{ mod }}

\begin{document}
\title{{The Decomposition Formula and Stationary Measures for Stochastic Lotka-Volterra Systems with Applications to Turbulent Convection}}
\titlerunning{Decomposition Formula and Stationary Measures for SLVS}
\author{Lifeng Chen\inst{1} \and Zhao Dong\inst{2}\and Jifa Jiang\inst{3}\fnmsep\inst{*}\and Lei Niu\inst{4}\and Jianliang Zhai\inst{5}}
%
\institute{Mathematics and Science College, Shanghai Normal University, Shanghai 200234, PRC. \\ \email{1000360929@smail.shnu.edu.cn}
\and Academy of Mathematics and Systems Science, Chinese Academy of Sciences, Beijing 100190, PRC. \email{dzhao@amt.ac.cn}
\and Mathematics and Science College, Shanghai Normal University, Shanghai 200234, PRC. \\ \email{jiangjf@shnu.edu.cn}. Corresponding author
\and Mathematics and Science College, Shanghai Normal University, Shanghai 200234, PRC. \\ \email{1000342830@smail.shnu.edu.cn}
\and Wu Wen-Tsun Key Laboratory of Mathematics, USTC, Chinese Academy of Sciences, Hefei Anhui 230026, PRC. \email{zhaijl@ustc.edu.cn}
}
\authorrunning{L. Chen\and Z. Dong\and J. Jiang\and L. Niu\and J. Zhai}

\date{}
%
%
\maketitle
\begin{abstract}
Motivated by the work of Busse et al. \cite{Busse1980science} on turbulent convection in a rotating layer, we exploit the long-run behavior for stochastic Lotka-Volterra (LV) systems both in pull-back trajectory and in stationary measure. It is proved stochastic decomposition formula describing the relation between solutions of stochastic and deterministic LV systems and stochastic Logistic equation. By virtue of this formula, it is verified that every pull-back omega limit set is an omega limit set for deterministic LV  systems multiplied by the random equilibrium of the stochastic Logistic equation. This formula is used to derive the existence of a stationary measure, its support and ergodicity. We prove the tightness for the set of stationary measures and the invariance for their weak limits as the noise intensity vanishes, whose supports are contained in the Birkhoff center.

The developed theory is successfully utilized to completely classify three dimensional competitive stochastic LV systems into $37$ classes. Time average probability measures weakly converge to an ergodic stationary measure on the attracting domain of an omega limit set in all classes except class 27 c). Among them there are two classes possessing a continuum of random closed orbits and ergodic stationary measures supported in cone surfaces, which weakly converge to the Haar measures of periodic orbits  as the noise intensity vanishes.  In the exceptional class, almost every pull-back trajectory  cyclically oscillates around the boundary of the stochastic carrying simplex characterized by three unstable stationary solutions. The limit for the time average probability measures is neither unique nor ergodic. These are subject to turbulent characteristics.
\end{abstract}
\vspace{11mm}
\setcounter{tocdepth}{3}
\tableofcontents

\section{Introduction}
 Turbulent convection  in a fluid layer heated from below  and  rotating about a vertical axis was studied by Busse et al. \cite{BusseExample,Busse1980science,Busse1980nonlinear}. They proved that turbulence occurs when both the Rayleigh number and the Taylor number exceed their critical values. This kind of turbulence is understood in terms of a manifold of stationary solutions, each of which is unstable relative to some other solution in the manifold so that the system evolves in time by realizing cyclically the different solutions of the manifold. This cyclically fluctuating solution was achieved  by May and Leonard \cite{May1} in the context of population biology,  which was confirmed by experiment in \cite{Busse1980science,Busse1980nonlinear}.

 Using the depth $d$ of the layer, the temperature difference between the upper and lower boundaries divided by the Rayleigh number $\frac{(T_2-T_1)}{R}$, and the thermal diffusion time $\frac{d^2}{\kappa}$ as scales for length, temperature and time, respectively, the convection model described as above is formulated by the Navier-Stokes equations for the velocity vector $\bold{v}$ and the heat equation for the deviation $\theta$ of the temperature from the static state:
 \begin{equation}
    \label{sys:01}
    \left\{
        \begin{array}{l}
            P^{-1}(\frac{\partial}{\partial t}+\mathbf{v}\cdot\nabla)\mathbf{v}+\frac{\sqrt{T}}{2}\mathbf{\lambda}\times \mathbf{v}=-\nabla\pi+\mathbf{\lambda}\theta+\nabla^2\mathbf{v}, \\
            \nabla\cdot \mathbf{v}=0, \\
            (\frac{\partial}{\partial t}+\mathbf{v}\cdot \nabla)\theta=R\mathbf{\lambda}\cdot \mathbf{v}+\nabla^2\theta,
        \end{array}
    \right.
\end{equation}
 where $\mathbf{\lambda}=(0,0,1)^{\tau}$. The physical state of the layer is represented in terms of three dimensionless parameters: the Rayleigh, Taylor, and Prandtl numbers
 $$ R=\frac{g\gamma(T_2-T_1)d^3}{\kappa\nu},\ \ T=\frac{4\Omega^2d^4}{\nu^2},\ \ {\rm and}\ \ P=\frac{\nu}{\kappa}.$$
 Here $\gamma,\ g,\ \kappa$ and $\nu$ are the thermal expansion coefficient, the gravitational acceleration constant, the thermal diffusivity and the kinematic viscosity, respectively. $\Omega$ is the angular velocity rotating about the vertical axis through the center of the layer. A stress-free condition is applied at the boundaries.

 The vertical component of the velocity field in the limit of small amplitudes can be approximately expressed by
\begin{equation}\label{vz}
\bold{v}_z =f(z,\alpha_c)\sum_{j=-n}^nc_j(t){\rm exp}(i\bold{k}_j\cdot \bold{r}),
\end{equation}
where $z$ (normalized between $-\frac{1}{2}$ and $\frac{1}{2}$) is the vertical component of  the position vector $\bold{r}$, $\alpha_c$ is the critical wave number predicted by linear analysis, $c_j(t)$ is the time-dependent amplitude and $\bold{k}_j$ is the horizontal wave vector. The equation for $f(z,\alpha_c)$ together with the boundary conditions at $z=\pm \frac{1}{2}$ represents an eigenvalue problem for $R(\alpha_c)$. At a finite value $\alpha_c$ this function reaches its minimum $R_c$ at which the onset of convection occurs. The horizontal wave vectors are subject to the conditions $|\bold{k}_j|=\alpha_c$ and $\bold{k}_{-j}=-\bold{k}_j$. The time-dependent amplitudes $c_j(t)$ are subject to the conditions $c_j(t)=-c_j^*(t)$, where $c_j^*(t)$ denotes the complex conjugate of $c_j(t)$ (see \cite{BusseExample,Busse1980nonlinear} for more details).  Then it follows from \cite{BusseExample} that $c_i$ satisfies the equations
\begin{equation}\label{amplitude}
M\frac{dc_i}{dt}=c_i\{(R-R_c)K-\frac{1}{2}\sum_{j=-n}^nT_{ij}\mid c_j\mid^2\}, \ \ i=1,2,...,n
\end{equation}
where the matrix elements $T_{ij}$ obey the symmetry relationships
\begin{equation}\label{symmetry}
T_{ij}=T_{i-j}=T_{-ij}.
\end{equation}
When the Rayleigh number $R$ exceeds the critical value $R_c$ depending on the Taylor number $T$, the static state becomes unstable and convective motions set in.   Restricting the analysis to the case $n=3$, setting $y_i=|c_i|^2$ for $i=1,2,3$ and making suitable renormalization,  Busse et al. \cite{BusseExample,Busse1980science,Busse1980nonlinear} transformed (\ref{amplitude}) into the standard {\it symmetric May-Leonard system} \cite{May1}:
\begin{equation}\label{sys1}
    \begin{array}{l}
        \displaystyle\frac{d y_1}{dt}=y_1(1- y_1-\alpha y_{2}-\beta y_{3}),\\
        \noalign{\medskip}
        \displaystyle\frac{d y_2}{dt}=y_2(1-\beta y_1-y_{2}-\alpha y_{3}), \ 0<\alpha<1<\beta,\\
        \noalign{\medskip}
        \displaystyle\frac{d y_3}{dt}=y_3(1-\alpha y_1-\beta y_{2}-y_{3}),
    \end{array}
\end{equation}
where $y_{i}\geq 0,\ i=1,2,3$.
If the Taylor number $T$ exceeds the critical value $T_c$, then $\alpha +\beta>2$ and $\alpha<1$.  A well-known result of  May-Leonard \cite{May1} applies here to arrive at that system (\ref{sys1}) exhibits nonperiodic
oscillations of bounded amplitude which have an ever increasing
cycle time.  Hence, Busse et al. \cite{BusseExample,Busse1980science,Busse1980nonlinear} concluded that turbulent convection in a rotating layer is approximately described by a manifold of three stationary solutions, all of which are unstable with respect to each other. Such a manifold is spanned by the three axial equilibria and approached rapidly from arbitrary initial conditions. We interestingly point out that this manifold is just the {\it carrying simplex} which would be founded by Hirsch \cite{H88} nearly ten years later.

However,  Busse and Heikes  \cite[p.174]{Busse1980science} pointed out:``small-amplitude disturbances are present at all times and not just at the initial moment, $\cdots$. The existence of a noise level prevents the amplitudes $y_i$ from decaying to arbitrary small levels. At the same time it introduces a random element into the time dependence of the system; $\cdots$.'' In the second paragraph of introduction, Heikes and Busse \cite{Busse1980nonlinear} clearly showed that this randomness occurs for Rayleigh number $R$ close to the critical value for the onset of convection, $R_c$. These statements mean that  the deviation of the Rayleigh number and its critical value is an average size in some sense. Affected by randomness,  Heikes and Busse \cite[p.31]{Busse1980nonlinear} expected that the transition from one set of rolls (stationary solutions) to the next becomes nearly periodic, with a transition time which fluctuates statistically about a mean value.

This stimulates us to exploit that whether stochastic version of cyclically fluctuating solution exists when $R - R_c$ is perturbed by a white noise
   $$(R - R_c)+\sigma B_t(\omega).$$
   Let $y_i=|c_i|^2$ for $i=1,2,...,n$ and rescale the time by $Mt$. Then the perturbed equations read
   \begin{equation}\label{slv}
   {\rm (E_{\sigma})}:dy_i = y_i(r+\displaystyle\sum^n_{j=1}a_{ij}y_j)dt + \sigma y_idB_t,\
 i=1,2,...,n
 \end{equation}
  on the positive orthant, where $r=(R - R_c)K$,  $a_{ij}=-T_{ij}$ and $\sigma$ are parameters, $B_t$ is a Brownian motion. ${\rm (E_{\sigma})}$ is called It\^{o} stochastic differential equations.
Our main purpose is to investigate an approach to prove the existence of cyclically fluctuating solutions in path for three dimensional system (\ref{slv}) provided that the Rayleigh number and the Taylor number exceed their critical values, and that these cyclically fluctuating solutions concentrate around each axial equilibrium (it is a saddle), which theoretically shows that turbulence still occurs under stochastic disturbances.  The system ${\rm (E_{\sigma})}$  is  well-known Lotka-Volterra equations with the growth rate $r$ and a white noise perturbation. Lotka-Volterra equations have been playing an important role in  population dynamics and game dynamics (see \cite{May,hofbauer1988}).

 Following the pioneer work of Khasminskii \cite{KHAS,KHAS1} for locally Lipschitz coefficients, the
existence and uniqueness of regular stationary measures for stochastic ordinary differential equations in $\mathbf{R}^n$ have been extensively studied. However, the feasible domain for ${\rm (E_{\sigma})}$ is the positive orthant of $\mathbf{R}^n$. Recently, in a general domain of $\mathbf{R}^n$, Huang, Ji, Liu and Yi \cite{HJLY1,HJLY2,HJLY3,HJLY4,HJLY5} have systematically investigated stationary measures for stochastic ordinary differential equations via Fokker-Plank equations. In \cite{HJLY1}, they provide several useful measure estimates of complement of a compact subset of regular stationary measures, such estimates can be used to obtain tightness of a family of stationary measures. Their key technique tool is the level set method, in particular, the integral identity they prove. This tool is used to give several new existence, respectively, non-existence results for stationary measures by applying Lyapunov-like, respectively, anti-Lyapunov-like functions in \cite{HJLY2,HJLY3,HJLY4}. Their results only need weaker regularity conditions and are applicable to both non-degenerate and degenerate equations. All they consider are regular stationary measures. Bogachev, R\"{o}ckner and Stannat \cite{BRS} and Shaposhnikov \cite{SS} provided examples to admit multiple or even infinite stationary measures, and more results stated in the paper of literature reviews \cite{BKR}.
For the attractors for random dynamical systems, one can refer the works of \cite{A,CDF,CF,CH} and the references therein.
Huang, Ji, Liu and Yi \cite{HJLY5} show that limiting measures as diffusion matrices vanish are invariant for the flow generated by drift vector field, and are supported in the global attractor of the flow if  the flow generated by drift vector field is dissipative.  Due to the limitation on the size of this article, the list of references contains less than a quarter of the bibliography we collected.

The purpose of this paper is generally to exploit the long-run behavior for stochastic Lotka-Volterra  systems ${\rm (E_{\sigma})}$  both in limit of pull-back trajectories and in stationary measures.  Motivated by \cite{JN}, we will first investigate the relation between solutions of ${\rm (E_{\sigma})}$ and those of  ${\rm (E_0)}$ and stochastic Logistic equation
\begin{equation}\label{stoLog}
dg=g(r -rg)dt +\sigma gdB_t
\end{equation}
and prove the Stochastic Decomposition Formula
 in the sense that
\begin{equation}\label{decom}
 \Phi(t,\omega,y)=g(t,\omega,g_0)\Psi(\int_0^tg(s,\omega,g_0)ds,\frac{y}{g_0}),
\end{equation}
where $\Phi(t,\omega,y)$, $ \Psi(t,y)$ are the solutions of
 ${\rm (E_{\sigma})}$ and ${\rm (E_0)}$ through the initial point $y$, respectively, and $g(t,\omega,g_0)$
 is the solution of  (\ref{stoLog}) through the initial point $g_0>0$.

 The stochastic decomposition formula (\ref{decom}) will play an important role to achieve our goal. By virtue of this formula, it is verified that every pull-back omega limit set for ${\rm (E_{\sigma})}$ is an omega limit set of ${\rm (E_{0})}$  multiplied by the random equilibrium of the scalar stochastic Logistic equation (\ref{stoLog}). We also investigate the weak convergence for the transition probability function of solution process as the time tends to infinity. Using the stochastic decomposition formula (\ref{decom}), the Khasminskii theorem  \cite[p.65]{KHAS} and  the Portmanteau theorem, it is shown that a bounded orbit for ${\rm (E_{0})}$  deduces the existence of a stationary measure for  ${\rm (E_{\sigma})}$ supported in a lower dimensional cone consisting of all rays connecting the origin and all points in the omega limit of this orbit. This means that any stationary measure is  not regular and the system possesses a continuum of stationary measures or multiple isolated stationary measures.  Furthermore, we provide the necessary and sufficient conditions for Markov semigroup to have a unique and ergodic stationary measure on some lower dimensional manifolds. If ${\rm (E_{0})}$  is dissipative, then we prove that the set of stationary measures with small noise intensity is tight, and that their limiting measures in weak topology are invariant with respect to the flow of  ${\rm (E_{0})}$  as the noise intensity $\sigma$ tends to zero, whose supports are contained in the Birkhoff center of ${\rm (E_{0})}$. This means that on the global attractor of  ${\rm (E_{0})}$ any limiting measure takes the complement of the Birkhoff center measure zero. In section 6,  we provide a complete classification for three dimensional competitive ${\rm (E_{\sigma})}$ with identical intrinsic growth rate both in pull-back trajectory and in stationary motion. There are exactly 37 scenarios in terms of competitive coefficients. Among them,  each pull-back trajectory in 34 classes is asymptotically stationary, but possibly different stationary solution for different trajectory in same class. All stationary measures are given for every system in these 34 classes, all limiting measures are the convex combination of its Dirac measures at equilibria. Two of the remain classes possess a family of stochastic closed orbits, and a continuum of ergodic stationary measures supported in  cone surfaces decided by periodic orbits for ${\rm (E_{0})}$, respectively, which weakly converge to the Haar measures of periodic orbits  as the noise intensity tends to zero.  All stationary measures are decided by these ergodic stationary measures via ergodic decomposition theorem.  Among all above 36 classes, for any given solution of ${\rm (E_{\sigma})}$ through a point $y$, its time average probability measure for transition probability function weakly converges to a stationary measure, which is ergodic in the attracting domain of the omega limit set for $y$. The final class, the most interesting and complicated one, violates this law. The time average probability measure for transition probability function of a solution not passing through the ray connecting the origin and the positive equilibrium of ${\rm (E_{0})}$ does not weakly converge, but has infinite limit measures which are not ergodic and support in three positive axes. As the noise intensity tends to zero, these stationary measures weakly converge to a convex combination of Dirac measures on three unstable axis equilibria.  We will reveal that the essential reason for both peculiar characteristics is that solutions concentrate around three equilibria very long time (approximately infinite) with probability nearly one. Besides, almost every pull-back trajectory  cyclically oscillates around the boundary of the stochastic carrying simplex which is characterized by three unstable stationary solutions.  All these are subject to turbulent characteristics. This rigorously proves that a stochastic version for so called statistical limit cycle exists and that the turbulence in a fluid layer heated from below and rotating about a vertical axis is robust under stochastic disturbances.

Here and throughout of this article, we will use notation
$\mathbf{R}^n_+:=\{y\in \mathbf{R}^n: y=(y_1,y_2,...,y_n),\ y_i\geq 0,\ i=1,2,...,n\} $
to denote the positive orthant, and its interior denotes ${\rm Int}\mathbf{R}_+^n$.
\section{Stochastic Decomposition Formula}
  A continuous random process $B=\{B_t
, \mathcal{F}_t:t\in \mathbf{R}\}$, defined on some probability
space $(\Omega,\mathcal{F},\mathbb{P})$, is called a two-sided time
{\it Brownian motion}, if $B_0=0$ a.s., for any $t\neq s$, $B_t-B_s$
is normally distributed with mean zero and variance $|t-s|$,
and $B$ has independent increments. In a canonical way, we may
assume $\Omega=C_0(\mathbf{R},\mathbf{R})$ endowed with the
compact-open topology, and $\mathcal{F}$ is its Borel
$\sigma-$algebra, $\mathbb{P}$ is Wiener measure on
$(\Omega,\mathcal{F})$, and $B_t(\omega)$ can be regarded as
coordinate process $\omega(t)$. We define the shift by
$$\theta_t:\Omega \rightarrow \Omega,\
\theta_t\omega(s):=\omega(s+t)-\omega(t),\ s,t\in \mathbf{R}.$$ Then
$\theta_t$ is a homeomorphism for each $t$ and
$(t,\omega)\rightarrow \theta_t\omega$ is continuous, hence
measurable. Thus the Brownian motion generates an ergodic metric
dynamical system $(\Omega,\mathcal{F},\mathbb{P},\{\theta_t:t\in
\mathbf{R}\})$ (see the Appendix A.3 in Arnold \cite{A} for details).

\begin{theorem}[Stochastic Decomposition Formula for It\^{o} Type]\label{isdf}
Let $\Phi(t,\omega,y)$ and $\Psi(t,y)$  be the solutions of
$(E_\sigma)$ and $(E_0)$, respectively.
 Then
\begin{equation}\label{sdfi}
   \Phi(t,\omega,y)=  g(t,\omega,g_0)\Psi(\int_0^tg(s,\omega,g_0)ds, \frac{y}{g_0}), \ y\in\mathbf{R}^n_+, \ g_0> 0,
\end{equation}
where $g(t,\omega,g_0)$ is a positive solution of the Logistic
equation
\begin{equation}\label{slogistic}
dg =  g(r- rg)dt + \sigma gdB_t,\ g(0,\omega,g_0)=g_0.
\end{equation}
\end{theorem}

\begin{proof}
 Define
$$G(x,\omega,u):=x\Psi(\int_0^ug(s,\omega,g_0)ds,\frac{y}{g_0})$$
and let
$\tilde{\Phi}(t,\omega,y)$ denote the right hand of (\ref{sdfi}).
Since $g(t,\omega,g_0)$ is a solution of (\ref{slogistic}), it is
adapted to the filtration $\{\mathcal{F}_t\}$.  By the definition of
Riemann integral, the integral for $g(t,\omega,g_0)$ with variable upper limit is still
adapted to the filtration $\{\mathcal{F}_t\}$. Therefore,
$G(x,\omega,u)$ is adapted with respect to $(\omega,u)$, continuous
with respect to $(x,u)$, continuously differentiable with respect to
$u$, and linear with respect to $x$, which implies that all
conditions for the extension of It\^{o}'s Formula hold (see
Exercise 3.12 in \cite[p.152]{Yor}).

Applying the extension of It\^{o}'s Formula, we obtain that for each $i = 1,2,..., n$,
$$
  \begin{array}{rl}
     d\tilde{\Phi}_i = &\Psi_i(\int_0^tg(s,\omega,g_0)ds,\frac{y}{g_0})[g(r-r g)dt + \sigma
gdB_t]\\
&+g^2(t,\omega,g_0)\Psi_i(\int_0^tg(s,\omega,g_0)ds,\frac{y}{g_0})\times\\
&[r+\displaystyle\sum^n_{j=1}a_{ij}\Psi_j(\int_0^tg(s,\omega,g_0)ds,\frac{y}{g_0})]dt\\
=&\tilde{\Phi}_i(r+ \displaystyle\sum^n_{j=1}a_{ij}\tilde{\Phi}_j)dt + \sigma
\tilde{\Phi}_i dB_t.
   \end{array}
$$
This completes the proof.

\end{proof}

We can also present Stochastic Decomposition Formula for Stratonovich stochastic differential equations:
\begin{equation}\label{sslv}
   dy_i = y_i(r+\displaystyle\sum^n_{j=1}a_{ij}y_j)dt + \sigma y_i\circ dB_t,\
 i=1,2,...,n,
 \end{equation}
which reveals the connection between solutions of (\ref{sslv}) and those of
\begin{equation}\label{dlv}
  \frac{dy_i}{dt} = y_i(r+\displaystyle\sum^n_{j=1}a_{ij}y_j),\
  y_i\geq 0, \
 i=1,2,...,n,
 \end{equation}
 and  Stratonovich Logistic equation:
 \begin{equation}\label{ssl}
dg =  g(r- rg)dt + \sigma g\circ dB_t.
\end{equation}
Here $\circ$ means Stratonovich integral.

\begin{theorem}[Stochastic Decomposition Formula for Stratonovich  Type]\label{ssdf}
Assume that $\Phi(t,\omega,y)$ and $\Psi(t,y)$  be the solutions of
(\ref{sslv}) and (\ref{dlv}), respectively, and $g(t,\omega,g_0)$ is a positive solution of the Logistic
equation of (\ref{ssl}). Then
 \begin{equation}\label{sdfs}
   \Phi(t,\omega,y)=  g(t,\omega,g_0)\Psi(\int_0^tg(s,\omega,g_0)ds,\frac{y}{g_0}), \ y\in\mathbf{R}^n_+, \ g_0> 0.
\end{equation}
\end{theorem}
\begin{proof}
Applying Theorem 2.4.2 in \cite[p.72]{CH}, we know that the stochastic Stratonovich type LV system (\ref{sslv}) is
equivalent to the stochastic It\^{o} type LV
system:
\begin{equation}\label{Itotype}
    dy_i = y_i(r+\frac{\sigma^2}{2}+\displaystyle\sum^n_{j=1}a_{ij}y_j)dt + \sigma y_idB_t,\
 i=1,2,...,n.
\end{equation}
Similarly, the stochastic Stratonovich type Logistic equation
(\ref{ssl}) is equivalent to the stochastic It\^{o}
type Logistic equation:
\begin{equation}\label{slogisticI}
dg =  g(r+\frac{\sigma^2}{2}- rg)dt + \sigma g dB_t.
\end{equation}
Thus, the conclusion can be obtained in the same manner as that of Theorem \ref{isdf}.
\end{proof}
\section{The Long-Run Behavior for (\ref{sslv})}

The stochastic Logistic equation (\ref{ssl}) can be solved, its solutions may be
represented in the form
\begin{equation}\label{sl}
g(t,\omega,x)=\frac{x{\rm exp}\{rt+\sigma B_t(\omega)\}}{1+rx\int_0^t{\rm exp}\{rs+\sigma
B_s(\omega)\}ds},\ x\geq0,
\end{equation}
whose {\it random equilibrium (or stationary solution)} is
\begin{equation}\label{lre}
u(\omega)= (r\int_{-\infty}^0{\rm exp}\{rs+\sigma
B_s(\omega)\}ds)^{-1},
\end{equation}
that is,
$$g(t,\omega,u(\omega))=u(\theta_t\omega),\ {\rm for\ all} \ t\in \mathbf{R}\ {\rm and\ all}\ \omega\in \Omega.$$
Similarly, we can define a random equilibrium for (\ref{sslv}). It is easy to see that $u(\theta_t\omega)$ is  a unique stationary solution,  whose probability density function satisfies the Fokker-Plank equation
$$\frac{\partial p}{\partial t}=-\frac{\partial}{\partial x}[x(r+\frac{\sigma^2}{2}-rx)p]+\frac{\sigma^2}{2}\frac{\partial^2}{\partial x^2}[x^2p]$$
and can be solved as follows
 \begin{equation}\label{densityL}
p^\sigma(x)=\frac{(\frac{2r}{\sigma^2})^{\frac{2r}{\sigma^2}}}{\Gamma(\frac{2r}{\sigma^2})}x^{\frac{2r}{\sigma^2}-1}{\rm
exp}\{-\frac{2r}{\sigma^2}x\}, \ \sigma \neq 0,\ x\geq0,
 \end{equation}
 where $\Gamma(\cdot)$ is $\Gamma-$function, see \cite{Gard} for
 details. Using the properties for $\Gamma-$function, we can calculate that
$$\mathbb{E}u=\int_0^{+\infty}xp^\sigma(x)dx=1.$$
The Birkhoff-Khintchin ergodic theorem (see, e.g., Arnold  \cite[Appendix]{A} ) implies that
\begin{equation}\label{ergodic}
\lim_{\mid t\mid\rightarrow \infty}\frac{1}{t}\int_0^tu(\theta_s\omega)ds=1
\end{equation}
on a $\theta$-invariant set $\Omega^*\in \mathcal{F}$ of full measure.

Employing the same method as in the proof of Theorems \ref{isdf} and \ref{ssdf}, we can derive the following stochastic decomposition formula.

\begin{corollary}\label{usdf}
Suppose that $\Phi(t,\omega,u(\omega)y)$ and $\Psi(t,y)$  be the solutions of
(\ref{sslv}) and (\ref{dlv}) through $u(\omega)y$ and $y$, respectively. Then
 \begin{equation}\label{udfs}
   \Phi(t,\omega,u(\omega)y)=  u(\theta_t\omega)\Psi(\int_0^tu(\theta_s\omega)ds,y),\ y\in\mathbf{R}_{+}^{n}.
\end{equation}
\end{corollary}

The pull-back trajectory for (\ref{ssl}) ((\ref{sslv})) emanating from $x$ ($y$) is defined by $g(t,
\theta_{-t}\omega,x)$ ($\Phi(t,\theta_{-t}\omega,y)$), which describes the internal evolution of the environment and $\theta_{-t}\omega$ represents the state of the environment at time $-t$ which transforms into the ``real" state $(\omega)$ at the time of observation (time $0$, after a time $t$ has elapsed).  Chueshov \cite[p.202]{CH} proved the following.
\begin{lemma}\label{lcon}
Every positive pull-back trajectory for stochastic Logistic equation (\ref{ssl}) is convergent to the random equilibrium $u(\omega)$.
\end{lemma}

\begin{lemma}\label{lem8.3}
For stochastic Logistic equation (\ref{ssl}),  we have
\begin{equation}\label{sys44}
\lim_{ t \rightarrow \infty}\frac{1}{t}\int_0^tg(s,
\theta_{-t}\omega,x)ds=1\ {\rm for\ all}\ x>0\ {\rm and}\ \omega \in
\Omega,\ {\rm and}
\end{equation}
\begin{equation}\label{sys44.1}
\lim_{ \mid t\mid \rightarrow \infty}\frac{1}{t}\int_0^tg(s,
\omega,x)ds=1\ {\rm for\ all}\ x>0\ {\rm and}\ \omega \in \Omega.
\end{equation}
\end{lemma}

\begin{proof}
Denote by $z(\omega)$ the random variable in $\mathbf{R}$ such that
$z(t,\omega):= z(\theta_t\omega)$ is {\it Stationary
Ornstein-Uhlenbeck Process} in\ $\mathbf{R}$ which solves the
OU-equation
$$dz = -\mu zdt + dB_t,\ \ \mu > 0.$$
Let us first introduce a conjugate transformation:
$$h(t,\omega,x)= g(t,\omega,x){\rm exp}\{-z(\theta_t\omega)\}.$$
Applying It\^{o} formula to the function $h(t,\omega)=
g(t,\omega,x){\rm exp}\{-z(\theta_t\omega)\}$, we transform the
stochastic equation (\ref{ssl}) into
\begin{equation}\label{sys45}
\frac{dh}{dt}=(1+\mu z(\theta_t\omega))h-h^2{\rm
exp}\{z(\theta_t\omega)\}.
\end{equation}
Since (\ref{sys45}) is Bernoulli's, it can be transformed into a
linear equation with negative Lyapunov exponent, which implies that
(\ref{sys45}) has a positive random equilibrium such that it is
exponentially stable.

For any fixed $t,x>0,\ \omega \in\Omega$, $h(s,\theta_{-t}\omega,x)$
is a solution of (\ref{sys45}) on $ 0\leq s\leq t$. Putting
$h(s,\theta_{-t}\omega,x)$ into the equation (\ref{sys45}), dividing
by $h(s,\theta_{-t}\omega,x)$, integrating from $0$ to $t$, and then
dividing by $t$, we have
$$\frac{1}{t}\ln \frac{h(t,\theta_{-t}\omega,x)}{x}=1+\frac{\mu}{t}\int_{-t}^0z(\theta_s\omega)ds-\frac{1}{t}\int_0^t g(s,\theta_{-t}\omega,x)ds. $$
Letting $t\rightarrow \infty$, we get that
$$\lim_{t\rightarrow \infty}\frac{1}{t}\int_0^t g(s,\theta_{-t}\omega,x)ds=1.$$
Here we have used $h(t,\theta_{-t}\omega,x)$ converges to a positive
equilibrium as $t\rightarrow \infty$, and $\displaystyle\lim_{t\rightarrow
\infty}\frac{1}{t}\int_0^t z(\theta_{s}\omega)ds=0$ by the Birkhoff
ergodic theorem. (\ref{sys44}) is true.

By using (\ref{sl}) and the Strong Law of Large Number for Brownian motion (see \cite[p.104]{KSH}), we have
\begin{displaymath}
    \begin{array}{rl}
    \displaystyle\lim_{t\rightarrow \infty}\frac{1}{t}\int_0^tg(s,\omega,x)ds =&\displaystyle\lim_{t\rightarrow \infty}\frac{1}{rt}\ln \{1+rx\int_0^t{\rm exp}(rs+\sigma B_s(\omega))\}ds\\
     = & \displaystyle\lim_{t\rightarrow \infty}\frac{1}{rt}\ln \int_0^t{\rm exp}(rs+\sigma B_s(\omega))ds\\
     = & 1.
\end{array}
\end{displaymath}
(\ref{sys44.1}) is proved.
 \end{proof}

For any $y\in \mathbf{R}^n_+$, let
$L(y):=\{\lambda y: \lambda \geq 0 \}$
denote the ray joining the origin and $y$. Denote by $F$ the vector field decided by the right side of  (\ref{dlv}) with $F_i(y):= y_i(r+\displaystyle\sum^n_{j=1}a_{ij}y_j)$ for $i=1,2,...,n$. $P$ is called an equilibrium if $F(P)=O$, an equilibrium $P=(p_{1},\cdots,p_{n})$ is said to be positive if $ p_i>0$ for $i=1,2,...,n$. Set by $\omega_F(z)$ the $\omega-$limit set for the trajectory $\Psi(t,z)$, which is the solution of (\ref{dlv}) passing through $z$.
Then we define
$$\mathcal{A}(\omega_F(z)):=\{y\in \mathbf{R}^n_+: \lim_{t\to \infty}\mathrm{dist}(\Psi(t,y),\omega_F(z))=0\}$$
to be the attracting domain for $\omega_F(z)$. Let $\mathcal{E}$ and $\Gamma$ denote the equilibria set and a closed orbit for  (\ref{dlv}), respectively. Then $\mathcal{A}(P)$ ($P\in \mathcal{E}$) and $\mathcal{A}(\Gamma)$ are the attracting domains for the equilibrium $P$ and the closed orbit $\Gamma$, respectively. A subset $S\subset \mathbf{R}^n_+$ is called positively invariant
(invariant) set for  (\ref{dlv}) if $\Psi(t,S)\subset (=)S$ for each $t\geq 0$.

\begin{theorem}[Cone Invariance]\label{cinav}
Let  $S\subset \mathbf{R}^n_+$ be positively invariant set for  (\ref{dlv}). Then
$$\Lambda(S):=\{\lambda y: {\rm for \ any}\ \lambda\geq 0\ {\rm and}\  y\in S\}$$
is invariant in the sense that
$$\Phi(t,\omega,\lambda y)\in \Lambda(S)\ {\rm whenever}\ \lambda\geq 0,\ \  y\in S,\ t>0,\ {\rm and}\ \omega \in \Omega.$$
Similar result holds for pull-back trajectory.
\end{theorem}
\begin{proof}
Take $ \lambda\geq 0,\  y\in S, \ {\rm and}\ t>0$. Then it follows from the Stochastic Decomposition Formula (\ref{sdfs}) that
$$\Phi(t,\omega,\lambda y)=  g(t,\omega,\lambda)\Psi(\int_0^tg(s,\omega,\lambda)ds,y).$$
Thus, the conclusion is implied by the invariance of $S$ and the definition of $\Lambda(S)$.
\end{proof}

The pull-back omega-limit set $\Gamma_y(\omega)$ of the pull-back the trajectory $\Phi(t,\theta_{-t}\omega,y)$ is defined by
$$\Gamma_y(\omega):=\bigcap_{t>0}\overline{\bigcup_{\tau\geq t}\Phi(\tau,\theta_{-\tau}\omega,y)}.$$
Now we state the main result of this section, which says that the pull-back omega-limit set of a  trajectory for (\ref{sslv}) is the omega-limit set of the trajectory (\ref{dlv}) through the same initial point multiplied by the random equilibrium for (\ref{ssl}).
\begin{theorem}\label{long-run}
Suppose that $\Psi(t,y)$ is a bounded solution for (\ref{dlv}). Then the pull-back omega-limit set $\Gamma_y(\omega)$ of the trajectory $\Phi(t,\theta_{-t}\omega,y)$ emanating from $y$ is $u(\omega)\omega_F(y)$, whose attracting domain is $\mathcal{A}(\omega_F(y))$.
\end{theorem}
\begin{proof}
For a given $y\in \mathbf{R}_+^n,\ z\in \omega_F(y)$, the stochastic decomposition formula (\ref{sdfs}) implies that
\begin{displaymath}
    \begin{array}{rl}
    & \Phi(t,\theta_{-t}\omega,y)-u(\omega)z\\
    \noalign{\medskip}
    =&g(t,\theta_{-t}\omega,1)\Psi(\int_0^tg(s,\theta_{-t}\omega,1)ds,y)-u(\omega)z \\
\noalign{\medskip}
    =&\big(g(t,\theta_{-t}\omega,1)-u(\omega)\big)\Psi(\int_0^tg(s,\theta_{-t}\omega,1)ds,y)\\
     & \ \ +u(\omega)\big(\Psi(\int_0^tg(s,\theta_{-t}\omega,1)ds,y)-z\big),
    \end{array}
\end{displaymath}
which deduces that
\begin{displaymath}
    \begin{array}{rl}
    & {\rm dist}\big(\Phi(t,\theta_{-t}\omega,y), u(\omega)\omega_F(y)\big)\\
    \noalign{\medskip}
    \leq& \mid g(t,\theta_{-t}\omega,1)-u(\omega)\mid
     \|\Psi(\int_0^tg(s,\theta_{-t}\omega,1)ds,y)\|\\
     & \ \ + u(\omega){\rm
    dist}\big(\Psi(\int_0^tg(s,\theta_{-t}\omega,1)ds,y),\omega_F(y)\big)\\
    \noalign{\medskip}
    \rightarrow& \ 0 \qquad {\rm as}\ t\rightarrow \infty,
    \end{array}
\end{displaymath}
by Lemma \ref{lcon}, (\ref{sys44}) and the boundedness of the trajectory $\Psi(\cdot,y)$. This proves that
$$\Gamma_y(\omega)\subset u(\omega)\omega_F(y)\ {\rm for\ any}\ \omega\in \Omega.$$

 Suppose that $z^*\in \omega_F(y)$. From (\ref{sys44}), there exists a sequence of $\{t_n\}$ tending to infinity such that
 $$\lim_{n\rightarrow \infty}\Psi\big(\int_0^{t_n}g(s,\theta_{-t_n}\omega,1)ds,y\big)=z^*.$$
 By the stochastic decomposition formula (\ref{sdfs}), we have
 \begin{equation}\label{sdl1}
  \Phi(t_n,\theta_{-t_n}\omega,y)=  g(t_n,\theta_{-t_n}\omega,1)\Psi(\int_0^{t_n}g(s,\theta_{-t_n}\omega,1)ds,y).
  \end{equation}
  Letting $n\rightarrow \infty$ in (\ref{sdl1}) and using Lemma \ref{lcon}, we get that $u(\omega)z^*\in \Gamma_y(\omega).$ In other words, $\Gamma_y(\omega)=u(\omega)\omega_F(y).$

   Let $p$ be in the attracting domain of  $\Gamma_y(\omega)$, that is,
   $$\lim_{t\rightarrow \infty}{\rm
dist}\big(\Phi(t,\theta_{-t}\omega,p), u(\omega)\omega_F(y)\big)=0.$$
This means that the trajectory $\Psi(\cdot,p)$ must be bounded. Applying the result proved above, we have
$\Gamma_p(\omega)=u(\omega)\omega_F(p).$ Because the pull-back trajectory $\Phi(t,\theta_{-t}\omega,p)$ is attracted by $ u(\omega)\omega_F(y)$, $ u(\omega)\omega_F(p)\subset u(\omega)\omega_F(y)$, in other words, $ \omega_F(p)\subset \omega_F(y)$. The proof is complete.
\end{proof}

\section{Stationary Measures,  Weak Convergence and Ergodicity}
Throughout the rest of the paper, we will assume without further mention that the domain of $\Psi(\cdot,\cdot)$
on $[0,\infty)\times\mathbf{R}_+^n$.

 In this section, we will investigate stationary measures and their supports,  weak
 convergence for the transition probability function, and ergodicity of the solution of (\ref{sslv}).

The transition probability function is defined by
\begin{equation}\label{tp}
 P(t,y,A):=\mathbb{P}(\Phi(t,\omega,y)\in A)
\end{equation}
which generates a Markov semigroup $\{P_t,t\geq 0\}$ with
$$P_tf(y):=\int_{\mathbf{R}^{n}_{+}}f(z)P(t,y,dz), \quad y\in \mathbf{R}^{n}_{+},  $$
where $f\in\mathcal{B}_b(\mathbf{R}^{n}_{+})$. Here $\mathcal{B}_b(\mathbf{R}^{n}_{+})$ denotes the set of
bounded measurable functions on $\mathbf{R}^{n}_{+}$.
\begin{theorem}[The Existence of Stationary Solution by Equilibrium]\label{ess}
Suppose that $P \in \mathcal{E}$ is a positive equilibrium for (\ref{dlv}). Then the system (\ref{sslv}) always has
 stationary solution $U(\omega):=u(\omega)P$, whose support is the ray $L(P)$ and its distribution
 function is
 \begin{equation}\label{df}
 F^\sigma_P(y)=\int_0^{{\rm
 min}\{\frac{y_1}{p_1},\frac{y_2}{p_2},...,\frac{y_n}{p_n}\}}p^\sigma(s)ds.
 \end{equation}
 Let $\mu^\sigma_P$ denote the probability measure decided by the distribution
function $F^\sigma_P$. Then for any $A\in \mathcal{B}(\mathbf{R}_+^n)$,
\begin{equation}\label{sm}
\mu^\sigma_P(A)= \mathbb{P}(U\in A)
\end{equation}
defines a stationary measure of the Markov semigroup $\{P_t,t\geq 0\}$.
\end{theorem}
\begin{proof}
The stochastic decomposition formula (\ref{udfs}) implies that the system (\ref{sslv}) always has
 stationary solution $U(\theta_t\omega):=u(\theta_t\omega)P$, whose support is obviously the ray $L(P)$. Let $F^\sigma_P$ denote the distribution function of $U(\theta_t\omega)$. Then, by the $\theta$-invariance property with respect to $\mathbb{P}$,
\begin{displaymath}
    \begin{array}{rl}
    F^\sigma_P(y)&=\mathbb{P}\{\omega: u(\omega)P\leq y\}\\
\noalign{\medskip}
    &=\mathbb{P}\{\omega: u(\omega)\leq {\rm
 min}\{\frac{y_1}{p_1},\frac{y_2}{p_2},...,\frac{y_n}{p_n}\}\}\\
 \noalign{\medskip}
 &=\int_0^{{\rm
 min}\{\frac{y_1}{p_1},\frac{y_2}{p_2},...,\frac{y_n}{p_n}\}}p^\sigma(s)ds.
    \end{array}
\end{displaymath}
The expressions (\ref{df}) and (\ref{sm}) are immediate.

In order to prove $\mu^\sigma_P(\cdot)$ to be stationary for $\{P_t,t\geq 0\}$, we need to show that
\begin{equation}\label{sta}
P_t\mu^\sigma_P(A)=\int_{\mathbf{R}_+^n}P(t,y,A)\mu^\sigma_P(dy) =\mu^\sigma_P(A)
\end{equation}
for  any $t\geq 0$ and $A \in \mathcal{B}(\mathbf{R}_+^n)$.

According to Arnold \cite[p.107]{A}, the future $\mathcal{F}_{+}$ and the past $\mathcal{F}_{-}$ $\sigma$-algebras
for generated by $(\theta,\Phi)$ are defined by
$$\mathcal{F}_{+}=\sigma\{B_t(\omega):t\geq 0\}$$
and
$$\mathcal{F}_{-}=\sigma\{B_t(\omega):t\leq 0\},$$
respectively. It is easy to see that $\mathcal{F}_{+}$ and $\mathcal{F}_{-}$
are independent and
$$
  u(\omega)=\left(r\int_{-\infty}^{0}\exp\{rs+\sigma B_{s}(\omega)\}ds\right)^{-1}
$$
is $\mathcal{F}_{-}$ measurable. This implies
that $U(\omega)$  is $\mathcal{F}_{-}$ measurable. Then by the celebrated lemma
\cite[Lemma 2.1.5 p.54]{A}
with $h(y,\omega)=I_{A}(\Phi(t,\omega,y))$,
$\mathcal{C}=\mathcal{F}_{-}$ and $\xi=U$, it yields that a.s.
\begin{displaymath}
    \begin{array}{rl}
    &\mathbb{E}[I_{A}(U(\theta_{t}\omega))|\mathcal{F}_{-}]\\
    =&\mathbb{E}[I_{A}\big(\Phi(t,\omega,U(\omega))\big)|\mathcal{F}_{-}]\\
    =&\mathbb{E}[I_{A}\big(\Phi(t,\cdot,y)\big)]|_{y=U(\omega)}
    \end{array}
\end{displaymath}
where we have used the fact that $I_{A}(\Phi(t,\omega,y))$ is $\mathcal{F}_{+}$ measurable
for each $y\in\mathbf{R}^{n}_{+}$. Therefore,
\begin{displaymath}
    \begin{array}{rl}
    \mu^{\sigma}_{P}(A)&=\mathbb{E}[I_{A}(U(\omega))]\\
    &=\mathbb{E}[I_{A}(U(\theta_{t}\omega))]\\
    &=\mathbb{E}[\mathbb{E}[I_{A}(U(\theta_{t}\omega))|\mathcal{F}_{-}]]\\
    &=\mathbb{E}[\mathbb{E}[I_{A}(\Phi(t,\cdot,y))]|_{y=U(\omega)}]\\
    &=\mathbb{E}[P(t,y,A)]|_{y=U(\omega)}]\\
    &=\int_{\mathbf{R}^{n}_{+}}P(t,y,A)\mu^{\sigma}_{P}(dy),
    \end{array}
\end{displaymath}
that is, (\ref{sta}) holds. This completes the proof.
\end{proof}

\begin{remark}\label{re1}
The above proof for $\mu^\sigma_P$ to be stationary is probabilistic. Instead, we can give a dynamical proof, which is presented in the following.

The random equilibrium $U(\omega)$ for the random dynamical system $\Theta:=(\theta,\Phi)$ generates an invariant measure $\mu$, whose factorization $\mu_{\omega}$ is a random Dirac measure, i.e., $\mu_{\omega}=\delta_{U(\omega)}$. It is easy to see that $\mu_{\omega}(\cdot)$ is $\mathcal{F}_{-}$ measurable. Hence, $\mathbb{E}[\mu_{\cdot}|\mathcal{F}_{+}] = \mathbb{E}\mu_{\cdot}=\mu^\sigma_P$. \cite[Theorem 2.3.45 p.107]{A} asserts that $\mathbb{P}\times \mu^\sigma_P$ is an invariant measure for $\Theta$. Therefore,  for  any $t\geq 0$ and $A \in \mathcal{B}(\mathbf{R}_+^n)$, we have
\begin{displaymath}
    \begin{array}{rl}
    &\int_{\mathbf{R}^{n}_{+}}P(t,y,A)\mu^{\sigma}_{P}(dy)\\
    =&\int_{\mathbf{R}^{n}_{+}}\int_{\Omega}I_{A}(\Phi(t,\omega,y))\mathbb{P}(d\omega)\mu^{\sigma}_{P}(dy)\\
    =&\int_{\Omega}I_{\Omega}(\theta_{t}\omega)\int_{\mathbf{R}^{n}_{+}}I_{A}(\Phi(t,\omega,y))\mu^{\sigma}_{P}(dy)\mathbb{P}(d\omega)\\
    =&\int_{\Omega\times\mathbf{R}^{n}_{+}}I_{\Omega\times A}\big(\theta_{t}\omega,\Phi(t,\omega,y)\big)\mathbb{P}\times \mu^\sigma_P(d\omega,dy)\\
    =&\int_{\Omega\times\mathbf{R}^{n}_{+}}I_{\Omega\times A}(\omega,y)\mathbb{P}\times \mu^\sigma_P(d\omega,dy)\\
    =&\mathbb{P}\times \mu^\sigma_P(\Omega\times A)=\mu^\sigma_P(A),
    \end{array}
\end{displaymath}
in the fourth equality, we have used the invariance for $\mathbb{P}\times \mu^\sigma_P$. This shows that $\mu^\sigma_P$ is stationary.
\end{remark}
\begin{remark}\label{re2}
A stationary measure $\mu$ for a system of stochastic ordinary differential equations is called {\it regular} if it admits a continuous density function $v$ with respect to the Lebesgue measure, i.e., $d\mu(x) = v(x)dx$. We claim that $\mu^\sigma_P$ is not regular.
\end{remark}
Otherwise, assume the density $v$ is continuous. Let
$W=\{y=(y_{1},\cdots,y_{n})\in{\rm
Int}\mathbf{R}_{+}^{n}: \frac{y_i}{p_i}\neq\frac{y_j}{p_j}, i\neq j, i,j=1,\cdots,n\}$
which is an open dense subset in $\mathbf{R}_+^n$. Then it follows from that

$$\int_0^{y_1}\int_0^{y_2}...\int_0^{y_n}v(s_1,s_2,...,s_n)ds_1ds_2...ds_n= \int_0^{\frac{y_k}{p_k}}p^\sigma(s)ds\ ,\ {\rm
 for\ some}\ k.$$
Differentiating on both sides of the above equation, we obtain that $v(y_1,y_2,...,y_n)=0$ on $W$. Together with the continuity of $v$, we have $v\equiv 0$ on  $\mathbf{R}_+^n$. This implies that $\mu^\sigma_P=0$, which is impossible from (\ref{sm}).

However, when restrict $\mu^\sigma_P$ on the ray $L(P)$, its density is $p^{\sigma}$.

\begin{theorem}[Weak Convergence and Ergodicity]\label{wce}

{\rm (i)} $\mu^\sigma_P(\cdot) \stackrel{w}{\rightarrow}
\delta_P(\cdot)$ as $\sigma \rightarrow 0$.

{\rm (ii)} Suppose that $P$ is a globally asymptotically stable equilibrium in
${\rm Int}\mathbf{R}_+^n$. Then for each $y\in {\rm
Int}\mathbf{R}_+^n$, $P(t,y,\cdot)  {\rightarrow}
\mu^\sigma_P$ weakly as $t\rightarrow \infty$, and
\begin{equation}\label{wc}
\lim_{t\rightarrow \infty}P(t,y,A)=\mu^\sigma_P(A), \ {\rm for\ any}\
A\in
 \mathcal{B}(\mathbf{R}_+^n).
\end{equation}
Moreover, $\mu^\sigma_P$ is the unique stationary measure for the
Markov semigroup $P_t$ in ${\rm Int} \mathbf{R}_+^n$, and hence, it
is ergodic.
\end{theorem}

\begin{proof}
{\rm (i)} In order to prove $\mu^\sigma_P(\cdot) \rightarrow \delta_P(\cdot)$
weakly as $\sigma \rightarrow 0$, we have only to verify
$\mu^\sigma_P(\cdot) \rightarrow \delta_P(\cdot)$ in vague topology as
$\sigma \rightarrow 0$ because  $\mu^\sigma_P(\cdot),\ \delta_P(\cdot)$ are
all probability measures. Equivalently, for arbitrary $f\in
C_c(\mathbf{R}_+^n)$, we need to prove
\begin{equation}\label{vag}
\lim_{\sigma \rightarrow
0}\int_{\mathbf{R}_+^n}f(y)\mu^\sigma_P(dy)=f(P),
\end{equation}
where $C_c(\mathbf{R}_+^n)$ denotes the set of all continuous
functions with compact support in $\mathbf{R}_+^n$. In particular,
$\lim_{y\rightarrow \infty}f(y)=0$ for any $f\in
C_c(\mathbf{R}_+^n)$.

For any nonnegative integers $m_1, m_2,\cdots,m_n$, let
$\alpha=\frac{2r}{\sigma^2}$ and $\widetilde{f}(y)={\exp}\{-(m_1y_1+m_2y_2+\cdots+m_ny_n)\}$.
Then
\begin{displaymath}
    \begin{array}{rl}
    \int_{\mathbf{R}_+^n}
    \widetilde{f}(y)\mu^\sigma_P(dy)=&\int_{\mathbf{R}_+^n}{\exp}\{-(m_1y_1+\cdots+m_ny_n)\}\mu^\sigma_P(dy)\\
    [2pt]
    =& \int_{L(P)}{\exp}\{-(m_1y_1+\cdots+m_ny_n)\}\mu^\sigma_P(dy)\\
    [2pt]
    =& \int_0^{\infty}{\exp}\{-(m_1p_1+\cdots+m_np_n)s\}p^\sigma(s)ds\\
    [2pt]
    =&\int_0^{\infty}\frac{\alpha^\alpha}{\Gamma(\alpha)}s^{\alpha-1}{\exp}\{-(m_1p_1+\cdots+m_np_n+\alpha)s\}ds\\
    [2pt]
    =&\frac{\alpha^\alpha}{(m_1p_1+\cdots+m_np_n+\alpha)^\alpha}\\
    \rightarrow & {\exp}\{-(m_1p_1+\cdots+m_np_n)\}= \widetilde{f}(P)
    \end{array}
\end{displaymath}
as $\sigma \rightarrow 0$. This shows (\ref{vag}) holds for such
exponent functions, hence it still holds for linear combination for
these exponent functions.

For any $f\in C_c(\mathbf{R}_+^n)$, make the transformation
$t_i=e^{-y_i},\ i=1,2,\cdots,n$. Then $t_i\in (0,1], i=1,2,\cdots,n$ and
$$g(t_1,\cdots,t_n):=f(-\ln t_1,\cdots,-\ln t_n)$$
is continuous on $(0,1]^n$. By the assumption on $f$, $\lim_{t_i
\rightarrow 0}g(t_1,\cdots,t_n)=0$. Define $g(t_1,\cdots,t_n)=0$ if there
is at least an $i$ with $t_i=0$. Then g is continuous on $[0,1]^n$.
By the Weierstrass Theorem, for any $\epsilon >0$, there is a polynomial
$P_m$  on $[0,1]^n$ such that
$$\displaystyle\max_{[0,1]^n}|g(t_1,\cdots,t_n)-P_m(t_1,\cdots,t_n)|
<\frac{\epsilon}{3}.$$  In particular,
$$| P_m(e^{-p_1},\cdots,e^{-p_n})-f(P)| <\frac{\epsilon}{3}.$$
 The last paragraph has shown that there is a
$\sigma_0$ such that as $|\sigma| <\sigma_0$,
$$|\int_{\mathbf{R}_+^n}P_m(e^{-y_1},e^{-y_2},\cdots, e^{-y_n})\mu^\sigma_P(dy)-P_m(e^{-p_1},e^{-p_2},\cdots,e^{-p_n})|
<\frac{\epsilon}{3}.$$ Thus, when $|\sigma| <\sigma_0$,
\begin{displaymath}
\begin{array}{rl}
&| \int_{\mathbf{R}_+^n}f(y)\mu^\sigma_P (dy)-f(P)|\\
[4pt]
\leq&|\int_{\mathbf{R}_+^n}\big(f(y)-P_m(e^{-y_1},e^{-y_2},\cdots,e^{-y_n})\big)\mu^\sigma_P(dy)|\\
[3pt]
&+|\int_{\mathbf{R}_+^n}P_m(e^{-y_1},e^{-y_2},\cdots,e^{-y_n})\mu^\sigma_P(dy)-P_m(e^{-p_1},e^{-p_2},\cdots,e^{-p_n})| \\
&+| P_m(e^{-p_1},e^{-p_2},\cdots,e^{-p_n)}-f(P)|\\
<&\epsilon.
\end{array}
\end{displaymath}
This proves (\ref{vag}).

{\rm (ii)} Assume that  $P$ is a globally asymptotically stable equilibrium in
${\rm Int}\mathbf{R}_+^n$. Let $f$ be a bounded continuous function
on $\mathbf{R}_+^n$ and fix $y\in {\rm Int}\mathbf{R}_+^n$,
\begin{displaymath}
    \begin{array}{rl}
    \displaystyle\lim_{t\rightarrow \infty}\int_{\mathbf{R}_+^n}f(z)P(t,y,dz)=&\lim_{t\rightarrow \infty} \int_\Omega f(\Phi(t,\omega,y))\mathbb{P}(d\omega)\\
    =& \lim_{t\rightarrow \infty}\int_\Omega f(\Phi(t,\theta_{-t}\omega,y))\mathbb{P}(d\omega)\\
    [3pt]
    =& \int_\Omega f(u(\omega)P)\mathbb{P}(d\omega)\\
    [3pt]
    =& \int_{\mathbf{R}_+^n} f(z)\mu^\sigma_P(dz),
    \end{array}
\end{displaymath}
by the assumption in (ii) and Lebesgue dominated convergence
theorem and Theorem \ref{long-run}. This deduces that $P(t,y,\cdot)  \stackrel{w}{\rightarrow}
\mu_P^\sigma$  as $t\rightarrow \infty$.

By (\ref{df}), the distribution function $F_P^\sigma$ of the measure
$\mu^\sigma_P$ is continuous, thus all points in $\mathbf{R}_+^n$ are
continuous for the measure $\mu^\sigma_P$. Thus, the weak convergence
for $P(t,y,\cdot)$ is equivalent to (\ref{wc}), which implies that
\begin{equation}\label{avtp}
\lim_{T\rightarrow \infty}\frac{1}{T}\int_0^TP(t,y,A)dt
=\mu^\sigma_P(A),\ {\rm for\ any}\ y\in {\rm Int}\mathbf{R}_+^n,\ {\rm
and}\ A\in
 \mathcal{B}(\mathbf{R}_+^n).
\end{equation}

Suppose $\nu$ is an arbitrary  stationary measure for the
Markov semigroup $P_t$ in ${\rm Int} \mathbf{R}_+^n$. Then
\begin{equation}\label{another}
 \int_{{\rm Int} \mathbf{R}_+^n}\nu(dy)P(t,y,A)=\nu(A).
\end{equation}
Integrating (\ref{another}) with respect to $t$ from $0$ to $T$ and using (\ref{avtp}), we get that
$$\mu^\sigma_P(A)=\nu(A).$$
This completes the proof.
\end{proof}

In the case  that $P$ is a nontrivial boundary  equilibrium for (\ref{dlv}), Theorems \ref{ess} and \ref{wce} still hold, which are stated as follows and can be proved by the same argument.
\begin{theorem}\label{bwce}
Suppose that $P$ is any nonzero  equilibrium for (\ref{dlv}). Then the system (\ref{sslv}) always has
 stationary solution $U(\omega):=u(\omega)P$, whose support is the ray $L(P)$ and distribution
 function is
 \begin{equation}\label{bdf}
 F^\sigma_P(y)=\int_0^{{\rm
 min}\{\frac{y_i}{p_i}:\ p_i\neq 0\}}p^\sigma(s)ds,\quad y\in\mathbf{R}^{n}_{+}.
 \end{equation}
 Let $\mu^\sigma_P$ denote the probability measure decided by the distribution
function $F^\sigma_P$. Then for any $A\in \mathcal{B}(\mathbf{R}_+^n)$,
\begin{equation}\label{bsm}
\mu^\sigma_P(A)= \mathbb{P}(U\in A)
\end{equation}
defines a stationary measure, and $\mu^\sigma_P(\cdot) \stackrel{w}{\rightarrow}
\delta_P(\cdot)$ as $\sigma \rightarrow 0$.

In addition, for each $y\in \mathcal{A}(P)$, $P(t,y,\cdot)  {\rightarrow}
\mu^\sigma_P$ weakly as $t\rightarrow \infty$, and
\begin{equation}\label{bwc}
\lim_{t\rightarrow \infty}P(t,y,A)=\mu^\sigma_P(A), \ {\rm for\ any}\
A\in
 \mathcal{B}(\mathbf{R}_+^n).
\end{equation}
Moreover, $\mu^\sigma_P$ is the unique stationary measure for the
Markov semigroup $P_t$ in $\mathcal{A}(P)$, and hence, it
is ergodic.
\end{theorem}

Theorems \ref{ess}, \ref{wce} and \ref{bwce} serve us to provide examples to have a continuum of stationary motions, which comes from the continuum of equilibria for deterministic systems.
\vskip 0.1cm

{\it Example} 4.1. Consider three-dimensional competitive LV system:
\begin{equation}\label{Ex1}
    \begin{array}{l}
        \displaystyle d y_1=y_1(1- y_1- y_{2}- y_{3})dt + \sigma y_1\circ dB_t,\\
        \noalign{\medskip}
        \displaystyle d y_2=y_2(1- y_1- y_{2}- y_{3})dt +\sigma y_2\circ dB_t, \\
        \noalign{\medskip}
        \displaystyle dy_3=y_3(1- y_1- y_{2}- y_{3})dt +\sigma y_3\circ dB_t.
    \end{array}
\end{equation}
The standard simplex $\Delta :=\{(y_1, y_2, y_3): y_1 + y_2 + y_3 = 1, y_1\geq 0, y_2\geq 0, y_3 \geq 0\}$ is the nonzero equilibria set for corresponding system without noise. Nontrivial stationary motions for (\ref{Ex1}) are
$\{u(\omega)P : P\in \Delta\}$, and
(\ref{wc}) holds for $y\in L(P)$.
\vskip 0.1cm

{\it Example} 4.2. Consider three-dimensional competitive LV system:
\begin{equation}\label{Ex2}
    \begin{array}{l}
        \displaystyle d y_1=y_1(1- 2y_1- y_{2}- y_{3})dt + \sigma y_1\circ dB_t,\\
        \noalign{\medskip}
        \displaystyle d y_2=y_2(1- y_1- 2y_{2}- y_{3})dt +\sigma y_2\circ dB_t, \\
        \noalign{\medskip}
        \displaystyle dy_3=y_3(1- \frac{3}{2}y_1- \frac{3}{2}y_{2}- y_{3})dt +\sigma y_3\circ dB_t.
    \end{array}
\end{equation}
 The nonzero equilibria set for corresponding system without noise is
 $$\mathcal{E}=\{(\alpha,\alpha , 1-3\alpha): 0\leq \alpha \leq \frac{1}{3}\}.$$
Nontrivial stationary motions for (\ref{Ex2}) are $\{u(\omega)P : P\in \mathcal{E}\}$ and
(\ref{wc}) holds for $y\in \mathcal{A}(P)$, which is a surface.

The stationary measure discussed above all originate from equilibria for system (\ref{dlv}) via the decomposition formula (\ref{sdfs}). We will investigate the other types of stationary measure coming from nontrivial omega limit sets of (\ref{dlv}).

\begin{theorem}[The Existence of Stationary Solution by Limit Set]\label{lsss}
Suppose that $\Psi(t,y)$ is a bounded
trajectory for (\ref{dlv}) with $y\in \mathbf{R}^{n}_{+}$.
Then the system (\ref{sslv}) admits a
stationary measure. Furthermore, if the origin $O$ is a repeller
and initial value $y\neq O$,
then this stationary measure is not the Dirac measure at the origin.
\end{theorem}
\begin{proof}
Since the trajectory of $\Psi(t,y)$ for (\ref{dlv}) is bounded,
there exists positive constant $N$ such that
\begin{equation}\label{bin1}
\|\Psi(t,y)\| \leq N,\ {\rm for\ all}\ t > 0.
\end{equation}
By the Khasminskii theorem \cite[p.65]{KHAS} and Chebyshev inequality, for
proving the existence of stationary measure for system (\ref{sslv}), it will suffice to prove that there exists a constant $M$ such that
\begin{equation}\label{L2b}
\mathbb{E}\|\Phi(t,\omega,y)\|^2 \leq M,\ {\rm for\ all}\ t\geq 0.
\end{equation}
It follows from the decomposition formula (\ref{sdfs}) that
$$\Phi(t,\omega,y)=g(t,\omega,1)\Psi(\int_0^tg(s,\omega,1)ds,y).$$
Therefore, $\|\Phi(t,\omega,y)\|\leq Ng(t,\omega,1).$ We have only to prove $ \mathbb{E}| g(t,\omega,1)|^2$ is bounded.
Applying the It\^{o} formula to (\ref{slogisticI}), we obtain that
$$g^2(t,\omega,1)=1+2(r+\sigma^2)\int_0^tg^2(s,\omega,1)ds-2r\int_0^tg^3(s,\omega,1)ds + 2\sigma\int_0^tg^2(s,\omega,1)dB_s.$$
Taking the expectation in the two sides of the above equation and utilizing the Fubini theorem, we have
$$\mathbb{E}g^2(t,\omega,1)=1+2(r+\sigma^2)\int_0^t\mathbb{E} g^2(s,\omega,1)ds-2r\int_0^t\mathbb{E}g^3(s,\omega,1)ds.$$
Differentiating of the above equality, we get that
$$\frac{d}{dt}\mathbb{E}g^2(t,\omega,1)=2(r+\sigma^2)\mathbb{E} g^2(t,\omega,1)-2r\mathbb{E}g^3(t,\omega,1).$$
By the H\"{o}lder inequality, we have $\mathbb{E}g^2(t,\omega,1) \leq (\mathbb{E}g^3(t,\omega,1))^{\frac{2}{3}},$ which deduces that
$$\frac{d}{dt}\mathbb{E}g^2(t,\omega,1)\leq 2(r+\sigma^2)\mathbb{E} g^2(t,\omega,1)-2r(\mathbb{E}g^2(t,\omega,1)^{\frac{3}{2}}.$$
The differential inequality theory implies that
\begin{equation}\label{elog}
\mathbb{E}g^2(t,\omega,1)\leq (1+\frac{\sigma^2}{r})^2.
\end{equation}
This shows that (\ref{L2b}) holds. The Khasminskii theorem \cite[p.65]{KHAS} asserts that there exists a stationary measure $\nu_y^\sigma$ for the semigroup $\{P_t,t\geq 0\}$.

If in addition the origin $O$ is a repeller, there is positive constant $k>0$ such that
\begin{equation}\label{bin}
k \leq \|\Psi(t,y)\|,\ {\rm for\ all}\ t > 0.
\end{equation}
Recall the proof of the Khasminskii theorem (see \cite[p.66]{KHAS}), there is a sequence $\{T_n\}$ tending
to infinity such that the probability measure sequence
\begin{equation}\label{pms}
P_n(A):=\frac{1}{T_n}\int_0^{T_n}P(t,y,A)dt
\end{equation}
converges weakly to the stationary measure $\nu_y^\sigma$ for $\{P_t,t\geq 0\}$. Finally, we will prove that $\nu_y^\sigma$ is not the Dirac measure at the origin.

In fact, let $B_R:=\{y\in \mathbf{R}^n_+:\|y\|<R\}$ denote the open ball with the center of the origin
and radius $R$ in $\mathbf{R}^n_+$. Then it follows from $P_n \stackrel{w}{\rightarrow}
\nu_y^\sigma$  and the Portmanteau theorem (see \cite[Theorem 2.1(iv)]{BILL}) that
\begin{equation}\label{emo}
\nu_y^\sigma(B_R)\leq \liminf_{n\rightarrow \infty}P_n(B_R)=\liminf_{n\rightarrow \infty} \frac{1}{T_n}\int_0^{T_n}P(t,y,B_R)dt,
\end{equation}
where $P(t,y,B_R)=\mathbb{P}(\|\Phi(t,\omega,y)\|<R)\leq \mathbb{P}(g(t,\omega,1)<\frac{R}{k})$ by the decomposition formula (\ref{sdfs}) and (\ref{bin}). Using \cite[Theorem 4.1]{CHEN}, we have
$$\lim_{t\rightarrow \infty}\mathbb{P}(g(t,\omega,1)<\frac{R}{k})=\int_0^{\frac{R}{k}}p^\sigma(s)ds.$$
Thus, for any given $\epsilon >0$, there is a $T(\epsilon)>0$ such that for all $t>T(\epsilon)$ ,
$$\mathbb{P}(g(t,\omega,1)<\frac{R}{k})<\frac{\epsilon}{2}+\int_0^{\frac{R}{k}}p^\sigma(s)ds.$$
Choose $n_0$ sufficiently large such that as $n>n_0$, $$T_n>T(\epsilon),\ {\rm and }\ \frac{1}{T_n}\int_0^{T(\epsilon)}P(t,y,B_R)dt<\frac{\epsilon}{2}.$$
We conclude that as $n>n_0$,
$$\frac{1}{T_n}\int_0^{T_n}P(t,y,B_R)dt<\epsilon + \int_0^{\frac{R}{k}}p^\sigma(s)ds.$$
Combining with (\ref{emo}), we have verified that
$$\nu_y^\sigma(B_R)\leq \epsilon + \int_0^{\frac{R}{k}}p^\sigma(s)ds.$$
Letting $\epsilon \rightarrow 0$, we obtain that
\begin{equation}\label{mob}
\nu_y^\sigma(B_R)\leq \int_0^{\frac{R}{k}}p^\sigma(s)ds.
\end{equation}
As a result, $\nu_y^\sigma(\{O\}) =\lim_{R\rightarrow 0}\nu_y^\sigma(B_R)=0$, in other words, $\nu_y^\sigma$ is not the Dirac measure at the origin.
   \end{proof}
\begin{remark}
Note that the pull-back omega-limit set $\Gamma_y(\omega)$ of the trajectory $\Phi(t,\theta_{-t}\omega,y)$ is $u(\omega)\omega(y)$ from Theorem \ref{long-run}, which implies that the difference between $\Phi(t,\omega,y)$ and $u(\theta_t\omega)\omega(y)$ converges to zero in probability as $t\rightarrow \infty$. This is the evidence to encourage
us to conjecture that the support of  $\nu_y^\sigma$ is contained in the cone $\Lambda(\omega(y)).$
The following assertion shows that this is true.
\end{remark}
\begin{theorem}[The Support of Stationary Measure]\label{ssm}
Suppose that $\Psi(t,y)$ is a bounded trajectory for (\ref{dlv}) with $y\neq O$. Then the support for stationary measure  $\nu_y^\sigma$ is  contained in the cone $\Lambda(\omega_F(y)).$
\end{theorem}
\begin{proof}

For $y\in \mathbf{R}_+^n$ with $y\neq O$, assume that $\nu_{y}^{\sigma}$ is a limit point in weak topology for probability measure family $\{\frac{1}{T}\int_0^TP(t,y,\cdot)dt:T>0\}$ as $T\rightarrow \infty$. We shall prove that
\begin{equation}\label{interior}
\nu_{y}^{\sigma}\big(\Lambda(\omega_F(y))\big)=1.
\end{equation}
In order to prove (\ref{interior}), it suffices to show that
\begin{equation}\label{interior1}
\nu_{y}^{\sigma}\big(\Lambda(U_{\epsilon}^c)\big)=0\ \ {\rm for}\ \ 0<\epsilon \ll 1
\end{equation}
where $U_{\epsilon}(\omega_F(y)):= \{x\in \mathbf{R}^n_+: {\rm dist}(x,\omega_F(y))\leq \epsilon\}$ and $U_{\epsilon}^c$ denotes its complement.

For any $\delta>0$, set $T_g^\delta(\omega)=\{t\geq0:\ g(t,\omega,g_0)\in(0,\delta]\}$, by the ergodic property of $g$ (see Corollary 4.3 in \cite[p.111-112]{KHAS}),
we have
\begin{eqnarray}\label{eq time 3}
\lim_{T\rightarrow\infty}\frac{1}{T}\int_0^T\mathbb{E}(I_{(0,\delta]}(g(s,\omega,g_0)))ds
=\lim_{T\rightarrow\infty}\mathbb{E}\frac{L(T_g^\delta(\omega)\cap[0,T])}{T}
=\mu_g((0,\delta])
\end{eqnarray}
where $L$ denotes the Lebesgue measure on $\mathbf{R}$ and $\mu_g$ is the nontrivial stationary measure.

Let $a\geq 0$. Then
\begin{equation}\label{stopping}
\tau(\omega,a):= {\rm inf}\{t>0: \int_0^tg(s,\omega,g_0)ds >a\}
\end{equation}
is a stopping time, which obviously satisfies that
\begin{equation}\label{stopping1}
 \int_0^{\tau(\omega,a)}g(s,\omega,g_0)ds = a,\ \ {\rm for\ \ all}\ \ \omega\in\Omega.
\end{equation}
This easily deduces that
$$L\big([0,\tau(\omega,a)]\cap(T_g^\delta(\omega))^c\big) \leq \frac{a}{\delta}, \ \  \omega \in\Omega.$$

Since $\lim_{t\rightarrow \infty}{\rm dist}\big(\Psi(t,y),\omega_F(y)\big)=0$, there exists $t_0=t_0(\epsilon)$ such that $\Psi(t,y)\in \Lambda\big({\rm Int}(U_{\epsilon}(\omega_F(y)))\big)$ for $t\geq t_0$.
By the stochastic decomposition formula (\ref{sdfs}) or the Cone Invariance Theorem,
$$\{\omega\in \Omega:\Phi(t,\omega,y)\in \Lambda(U_{\epsilon}^c)\}=\{\omega\in \Omega:\Psi(\int_0^tg(s,\omega,g_0)ds,y)\in \Lambda(U_{\epsilon}^c)\}.$$
Thus, we obtain that
\begin{displaymath}
    \begin{array}{rl}
     & \frac{1}{T}\int_0^TP\big(t,y,\Lambda (U_{\epsilon}^c)\big)dt\\
     [2pt]
    =&\frac{1}{T}\int_0^T\mathbb{P}\{\Psi(\int_0^tg(s,\omega,g_0)ds,y)\in \Lambda (U_{\epsilon}^c)\}dt\\
     [2pt]
    = &\mathbb{E}\frac{1}{T}\int_0^TI_{\Lambda (U_{\epsilon}^c)}\big(\Psi(\int_0^tg(s,\omega,g_0)ds,y)\big)dt\\
    [2pt]
    \leq & \mathbb{E}\frac{1}{T}L[0,T\wedge \tau(\omega,t_0)]\\
    [2pt]
    =& \mathbb{E}\{\frac{1}{T}L\big([0,T\wedge \tau(\omega,t_0)]\cap T_g^\delta(\omega)\big)+\frac{1}{T}L\big([0,T\wedge\tau(\omega,t_0)]\cap(T_g^\delta(\omega))^c\big)\}\\
     [2pt]
    \leq& \mathbb{E}\{\frac{1}{T}L\big([0,T]\cap T_g^\delta(\omega)\big)+\frac{1}{T}L\big([0,\tau(\omega,t_0)]\cap(T_g^\delta(\omega))^c\big)\}\\
    [2pt]
    \leq& \mathbb{E}\{\frac{1}{T}L\big([0,T]\cap T_g^\delta(\omega)\big)\}+\frac{t_0}{T\delta}.
    \end{array}
\end{displaymath}
From this inequality and (\ref{eq time 3}), we have that
$$\overline{\lim}_{T\rightarrow\infty}\frac{1}{T}\int_0^TP\big(t,y,\Lambda (U_{\epsilon}^c)\big)dt\leq \mu_g\big((0,\delta]\big)=\int_0^{\delta}p^{\sigma}(s)ds.$$
Due to the openness of $\Lambda(U_{\epsilon}^c)$, the Portmanteau theorem and $\delta$ being arbitrary, (\ref{interior1}) follows. This shows $\nu_y^\sigma(\Lambda(\omega_F(y)))=1$.
\end{proof}

\begin{remark}
Suppose that $\Psi(t,y)$ is a nontrivial periodic orbit $\Gamma$ for (\ref{dlv}). Then $\Lambda(\Gamma)$ is a cone surface with the origin as vertex. It follows from \cite[Proposition 4.13]{JN} that
$\omega_F(x)=\Gamma$ for all $x\in \Lambda(\Gamma)\setminus \{O\}$, that is, $\Gamma$ is a global attractor when the flow $\Psi$ is restricted to $\Lambda(\Gamma)\setminus \{O\}$. By the cone invariance, $\Lambda(\Gamma)$ is invariant for both $\Phi(t,\omega,\cdot)$ and $\Phi(t,\theta_{-t}\omega,\cdot)$. Applying Theorem \ref{long-run}, we know that $u(\omega)\Gamma$ is a global attractor for pull-back flow $\Phi(t,\theta_{-t}\omega,\cdot)$ restricted on $\Lambda(\Gamma)\setminus \{O\}$.
In three dimensional stochastic competitive Lotka-Volterra system $\nu_y^\sigma$ is a unique nontrivial stationary measure
(in Theorem \ref{smucone}).
We guess that $\nu_y^\sigma$ is a unique nontrivial stationary measure supported on $\Lambda(\Gamma)$ and $u(\theta_t\omega)\Psi(\int_0^tu(\theta_s\omega)ds,y)$ is just such a stationary process in this general situation, but we cannot prove it. Here we leave it an open problem. However, in the following, we are able to show that $\nu_y^\sigma$ converges weakly to the Haar measure supported on $\Gamma$ as $\sigma \rightarrow 0$ (see Theorem \ref{main} and Corollary \ref{mainc}). A similar  problem can be proposed for a quasiperiodic orbit $\Psi(t,y)$.
\end{remark}
It is easy to see that all these stationary measures are not regular.
\vskip 0.1cm
{\it Example} 4.3. Consider the following three-dimensional prey-predator LV system:
\begin{equation}\label{sys:LV3}
    \begin{array}{l}
        \displaystyle\frac{d y_1}{dt}=y_1(1- y_1+2 y_{2}-3 y_{3}),\\
        \noalign{\medskip}
        \displaystyle\frac{d y_2}{dt}=y_2(1-3 y_1-y_{2}+ y_{3}), \\
        \noalign{\medskip}
        \displaystyle\frac{d y_3}{dt}=y_3(1+ y_1-4 y_{2}-y_{3}).
    \end{array}
\end{equation}
It is easy to calculate that the system (\ref{sys:LV3}) has a unique positive equilibrium $E_0=(\frac{3}{8},\frac{1}{4},\frac{3}{8})$. \cite[Example 3.1]{JN} has shown that (\ref{sys:LV3}) admits a family of  invariant cone  surfaces $\Lambda(h)$: $$\frac{y_1y_2y_3}{(2y_1+3y_2+2y_3)^3}\equiv h,\ \ 0<h\leq \frac{1}{324}$$ on which there is no equilibrium except $h=\frac{1}{324}$. Hence, every trajectory on $\Lambda(h)$ will converge to a periodic orbit on it. These periodic orbits must lie on the center manifold for $P_0$, which is transversal to each $\Lambda(h)$ and intersects with $\Lambda(h)$ on the unique closed orbit $\Gamma(h)$.

Now we study the noise disturbed system:
\begin{equation}\label{sys:NLV3}
    \begin{array}{l}
        \displaystyle dy_1=y_1(1- y_1+2 y_{2}-3 y_{3})dt + \sigma y_1\circ dB_t,\\
        \noalign{\medskip}
        \displaystyle dy_2=y_2(1-3 y_1-y_{2}+ y_{3})dt + \sigma y_2\circ dB_t, \\
        \noalign{\medskip}
        \displaystyle dy_3=y_3(1+ y_1-4 y_{2}-y_{3})dt + \sigma y_3\circ dB_t.
    \end{array}
\end{equation}
Applying Theorems \ref{ssm}, \ref{ess}, and \ref{long-run}, we conclude that there exists a stationary measure $\nu_h^\sigma$ supported on $\Lambda(h) (0<h\leq \frac{1}{324})$ and every nontrivial pull-back trajectory on
$\Lambda(h)$ tends to $u(\omega)\Gamma(h)$ as $t\rightarrow \infty$. The subsequent theorem will show that $\nu_h^\sigma$ converges weakly to the Haar measure on the closed orbit $\Gamma(h)$ as $\sigma \rightarrow 0.$

Example 4.3 illustrates (\ref {sys:NLV3}) has a family of stationary measures coming from the continuum of periodic orbits for (\ref {sys:LV3}). Such stationary processes are not isolated. The following gives  an example to possess as least three isolated stationary processes.
\vskip 0.1cm
{\it Example} 4.4. Consider four-dimensional white noise perturbed prey-predator Lotka-Volterra system:

\begin{equation}\label{sys:NLV4}
    \begin{array}{l}
        \displaystyle dy_1=y_1(2- \frac{3}{4}y_1+ y_{2}-\frac{3}{2} y_{3}-2 y_{4})dt+\sigma y_1\circ dB_t,\\
        \noalign{\medskip}
        \displaystyle dy_2=y_2(2 +3 y_1-3y_{2}-\frac{33}{2} y_{3}-4y_4)dt+\sigma y_2\circ dB_t, \\
        \noalign{\medskip}
        \displaystyle dy_3=y_3(2 +\frac{2959}{4000} y_1-\frac{9}{2} y_{3}-\frac{989}{125}y_{4})dt+\sigma y_3\circ dB_t\\
        \noalign{\medskip}
        \displaystyle dy_4=y_4(2 +\frac{1}{2}y_1- y_{2}-3y_{3}-6y_4)dt+\sigma y_4\circ dB_t.
    \end{array}
\end{equation}
The deterministic system without noise in each equation was investigated in \cite[Example 3.2]{JN}. This deterministic system has a unique equilibrium $P$ and at least two limit cycles $\Gamma_1$ and $\Gamma_2$ ({\bf isolated} closed orbits). It follows from Theorems \ref{ssm} and \ref{ess} that (\ref{sys:NLV4}) admits at least three isolated stationary measures, named by $\nu_1^\sigma$, $\nu_2^\sigma$, and $\mu_P^\sigma$, which support on $\Lambda(\Gamma_1)$, $\Lambda(\Gamma_2)$, and $L(P)$, respectively.

In the language of dynamics, the stationary measures $\{\nu_h^\sigma\}$ in Example 4.3 are degenerate, while  $\nu_1^\sigma$, $\nu_2^\sigma$, and $\mu_P^\sigma$ in Example 4.4 are hyperbolic.

\section{Limiting Measures for Stationary Measures and Their Supports}

In this section, we will exploit the weak convergence for stationary measures as the noise intensity $\sigma$ tends to zero.  The paper \cite{CJ} has established the frame to study limiting behavior of stationary measures with small noise intensity. According to the frame, the study is divided into three steps: the first step is to prove that the solution of (\ref{sslv}) converges to the solution of (\ref{dlv}) uniformly starting from the same initial point on any compact set as $\sigma\rightarrow 0$; the second step is to prove the tightness for the family of stationary measures and then to show that  any limiting measure is an invariant measure for deterministic system (\ref{dlv}); the third step is to deduce that any limiting measure is supported in the Birkhoff center for (\ref{dlv}).

Let us start with the first step. Before that, we will present the dissipation assumption.

The system (\ref{dlv}) is said to be {\it dissipative}, if there is a compact invariant set $D$, called the {\it fundamental attractor}, which uniformly attracts each compact set of initial values. Wilson \cite{Wilson} proved that $D$ has a $C^{\infty}$ Lyapunov function $V: \mathbf{R}_+^n\rightarrow \mathbf{R}: V(\Psi(t,y))<V(y)$ if $t>0,\ y\in \mathbf{R}_+^n\setminus D$ with $V(y)\rightarrow \infty$ as $y\rightarrow \infty$. By the Sard theorem \cite{SARD}, $V$ has a sequence of regular values $\alpha_n \rightarrow \infty.$ Then $V^{-1}(\alpha_n)$ is the boundary of the compact set $M_n =V^{-1}(-\infty,\alpha_n]$, which is a neighborhood of $D$; and the flow
enters $M_n$ transversely along $\partial M_n=V^{-1}(\alpha_n)$ for each $n\geq 0$.

 Throughout this section, we assume that the system (\ref{dlv}) is dissipative and use the notation
$M_0 =V^{-1}(-\infty,\alpha_0]$, where $\alpha_0$ is a regular point for $V$.

Because we are concerned with the variation for the solution $\Phi(t,\omega,y)$ as $\sigma \rightarrow 0$, we let
 $\Phi^\sigma(t,\omega,y)$ denote the solution of (\ref{sslv}) from now on, similarly for $g^\sigma(t,\omega,1)$.
\begin{proposition}
\label{procon}
 Let $K\subset \mathbf{R}_+^{n}$ be a compact set and $T>0$ an arbitrary number. Then there is a constant $C$£¬ depending on $K$ and $T$,  such that
\begin{equation}\label{LP}
\sup_{y\in K}\mathbb{E}[\displaystyle \|\Phi^\sigma(T,\omega,y) -\Psi(T,y)\|]\leq C|\sigma|,
\end{equation}
which implies that for any $\delta >0$
\begin{equation}\label{probc}
\lim_{\sigma\to 0}
\sup_{y\in K}\mathbb{P}\{\displaystyle
\|\Phi^\sigma(T,\omega,y)-\Psi(T,y)\|\geq\delta\}=0.
\end{equation}
\end{proposition}
\begin{proof}
Without loss of generality, we may assume that $K\subset M_0$. Let $F(y)$ denote the vector field for the right hand of (\ref{dlv}). Since $M_0$ is compact, there is a constant $C_0$ such that $\|y\|+\|F(y)\|\leq C_0$ for all $y\in M_0$.

Utilizing the decomposition formula (\ref{sdfs}), we get that
\begin{displaymath}
\begin{array}{rl}
&\displaystyle\Phi^\sigma(t,\omega,y) -\Psi(t,y)\\
[4pt]
=&\big(g^\sigma(t,\omega,1)-1\big)\Psi(\int_0^tg^\sigma(s,\omega,1)ds,y) + \big(\Psi(\int_0^tg^\sigma(s,\omega,1)ds,y)-\Psi(t,y)\big)\\
[3pt]
=&\big(g^\sigma(t,\omega,1)-1\big)\Psi(\int_0^tg^\sigma(s,\omega,1)ds,y)\\ &+\int_0^1F\big(\Psi(\lambda\int_0^tg^\sigma(s,\omega,1)ds+(1-\lambda)t,y\big)d\lambda\int_0^t\big(g^\sigma(s,\omega,1)-1\big)ds.\\
\displaystyle
\end{array}
\end{displaymath}
Hence, for all $y\in M_0$, we have
\begin{equation}\label{esti1}
\mathbb{E}\|\Phi^\sigma(T,\omega,y) -\Psi(T,y)\|\leq C_0\big[\mathbb{E}|g^\sigma(T,\omega,1)-1|+\int_0^T \mathbb{E}|g^\sigma(t,\omega,1)-1|dt\big].
\end{equation}

By H\"{o}lder inequality, we have for any $t\in [0,T]$,
\begin{displaymath}
\begin{array}{rl}
&\displaystyle \mathbb{E}|g^\sigma(t,\omega,1)-1|\\
[4pt]
=&\mathbb{E}|\frac{1}{g^\sigma(t,\omega,1)}-1|g^\sigma(t,\omega,1)\\
[3pt]
\leq& \sqrt{\mathbb{E}|\frac{1}{g^\sigma(t,\omega,1)}-1|^2}\sqrt{\mathbb{E}|g^\sigma(t,\omega,1)|^2}.\\
\displaystyle
\end{array}
\end{displaymath}
From (\ref{elog}) it follows that
\begin{equation}\label{gL1}
\mathbb{E}|g^\sigma(t,\omega,1)-1|\leq (1+\frac{\sigma^2}{r})\sqrt{\mathbb{E}|\frac{1}{g^\sigma(t,\omega,1)}-1|^2}.
\end{equation}
 Let $h^\sigma(t,\omega,1):= \frac{1}{g^\sigma(t,\omega,1)}$. Then we need to estimate $\mathbb{E}|h^\sigma(t,\omega,1)-1|^2$.

 Using (\ref{slogisticI}) and the It\^{o} formula, we derive that
 \begin{equation}\label{heq}
dh^\sigma_{t} = [r+(\frac{\sigma^2}{2}-r)h^\sigma_{t}]dt-\sigma h^\sigma_{t} dB_t.
\end{equation}
Applying the It\^{o} formula to $(h^\sigma_{t})^2$, and then taking the expectation in the two sides, we obtain that
$$\mathbb{E}(h^\sigma_{t})^2= 1+2r\int_0^t\mathbb{E}h^\sigma_{s} ds+2(\sigma^2-r)\int_0^t\mathbb{E}(h^\sigma_{s})^2 ds,$$
which implies that
\begin{displaymath}
    \begin{array}{rl}
    \frac{\mathbb{E}(h^\sigma_{t})^2}{dt}=&2r\mathbb{E}h^\sigma_{t}+2(\sigma^2-r)\mathbb{E}(h^\sigma_{t})^2\\
    [2pt]
    \leq & 2\sqrt{\mathbb{E}(h^\sigma_{t})^2}\big[r-(r-\sigma^2)\sqrt{\mathbb{E}(h^\sigma_{t})^2}\big].
    \end{array}
\end{displaymath}
This shows that
\begin{equation}\label{hesti}
\mathbb{E}(h^\sigma_{t})^2 \leq (\frac{r}{r-\sigma^2})^2.
\end{equation}
It follows from (\ref{heq}) that
$$h^\sigma_{t}-1= r\int_0^t(1-h^\sigma_{s})ds+\frac{\sigma^2}{2}\int_0^th^\sigma_{s} ds-\sigma\int_0^th^\sigma_{s} dB_s.$$
Let $T>0$ and any $t\in[0,T]$. Then
\begin{displaymath}
\begin{array}{rl}
&\mathbb{E}[\displaystyle\sup_{0\leq s\leq t}(h^\sigma_{s}-1)^2]\\
[4pt]
\leq&3\big\{r^2T\int_0^t\mathbb{E}[\displaystyle\sup_{0\leq l\leq s}(h^\sigma_{l}-1)^2]ds+T\frac{\sigma^4}{4}\int_0^t\mathbb{E}(h^\sigma_{s})^2 ds +\sigma^2\mathbb{E}[\displaystyle\sup_{0\leq s\leq t}(\int_0^s h^\sigma_{l} dB_l)^2]\big\}\\
[3pt]
\leq& 3\big\{r^2T\int_0^t\mathbb{E}[\displaystyle\sup_{0\leq l\leq s}(h^\sigma_{l}-1)^2]ds+(T\frac{\sigma^4}{4}+4\sigma^2)\int_0^t\mathbb{E}(h^\sigma_{s})^2 ds \big\}\\
[3pt]
\leq& 3\big\{r^2T\int_0^t\mathbb{E}[\displaystyle\sup_{0\leq l\leq s}(h^\sigma_{l}-1)^2]ds+T(T\frac{\sigma^4}{4}+4\sigma^2)(\frac{r}{r-\sigma^2})^2 \big\}.
\end{array}
\end{displaymath}
Here in the second inequality, we have used Doob's maximal inequality (\cite[p.14]{KSH}) and the It\^{o} isometry
(\cite[p.137]{KSH}), and in the third inequality, we have applied (\ref{hesti}). The Grownwall inequality is applied here  so that we conclude that for all $t\in [0,T]$,
\begin{equation}\label{hL1}
\mathbb{E}[\displaystyle\sup_{0\leq s\leq t}(h^\sigma_{s}-1)^2]\leq 3T(T\frac{\sigma^4}{4}+4\sigma^2)(\frac{r}{r-\sigma^2})^2\exp(3r^2T^2).
\end{equation}
(\ref{LP}) follows from (\ref{esti1}), (\ref{gL1}), and (\ref{hL1}) immediately, and the Chebyshev inequality implies (\ref{probc}).
\end{proof}

 From the Khasminskii theorem (see \cite[p.65]{KHAS}), we know that any limiting measure $\nu_y^\sigma$ in weak sense for a subsequence of  probability measures
\begin{equation}\label{0pms}
P_T(y,\cdot):=\frac{1}{T}\int_0^{T}P(t,y,\cdot)dt,
\end{equation}
is a stationary measure for (\ref{sslv}) (see Theorem \ref{lsss}), where $y\in \mathbf{R}_+^n$.
Because what we are interested in is limit behavior for stationary measures as $\sigma \rightarrow 0$, we pay our attention to small $\sigma$. Hence, we restrict $0<|\sigma|\leq \sigma_0$ for $\sigma_0$ sufficiently small. Now denote by $\mathcal{M}_S(\sigma_0)$ all the stationary measures obtained in a manner just stated. The following proposition answers the tightness of this stationary measures set.
\begin{proposition}\label{tightp}
$\mathcal{M}_S(\sigma_0)$ is tight.
\end{proposition}
\begin{proof}
For any $N>0$ and $\nu_y^{\sigma}$ with $0<|\sigma|\leq \sigma_0$ and $y\in \mathbf{R}_+^n$, we have
$$\nu_y^{\sigma}(B_N^c)\leq \frac{1}{N}\int_{\mathbf{R}_+^n}\|x\| \nu_y^{\sigma}(dx). $$
In order to prove the tightness of $\mathcal{M}_S(\sigma_0)$, we only have to prove that there is a positive constant $C$, independent of $y$ and $\sigma$,  such that
\begin{equation}\label{bound}
\int_{\mathbf{R}_+^n}\|x\| \nu_y^{\sigma}(dx) \leq C,\ \ {\rm for \ \ any}\ \ \nu_y^{\sigma}\in \mathcal{M}_S(\sigma_0).
\end{equation}

 Now for any given $\nu_y^{\sigma}\in \mathcal{M}_S(\sigma_0)$, there is a time sequence $T_n\uparrow \infty$ such that
 \begin{equation}\label{weak}
 \frac{1}{T_n}\int_0^{T_n}P(t,y,\cdot)dt \stackrel{w}{\rightarrow} \nu_y^{\sigma}\ \ {\rm as}\ n\rightarrow \infty.
 \end{equation}
 Since $M_0$ is compact and $\Psi(t,y)$ is bounded, there are positive constants $C_0$  and $C_y$ such that $\|x\|\leq C_0$ for all $x\in M_0$ and $\|\Psi(t,y)\|\leq C_y$ for any $t\geq 0$. The dissipation assumption implies that there is a $t_0$ such that $\Psi(t,y)\in M_0$ for all $t \geq t_0$. By (\ref{stopping}), $\tau(\omega,t_0)$ is a stopping time satisfying (\ref{stopping1}) and $\tau(\omega,t_0) < \infty$. Therefore,
 \begin{equation}\label{stopping2}
 \lim_{T\rightarrow \infty}\mathbb{P}\big(\tau(\cdot,t_0)> T\big)=0,
 \end{equation}
 which implies that there is a $T=T_y$ such that
 \begin{equation}
 \mathbb{P}\big(\tau(\cdot,t_0)> T_y\big)< \frac{1}{C_y^2}.
 \end{equation}

For any $N>0$, define a continuous function $f_N\in C_b(\mathbf{R}_+^n)$:
\begin{displaymath}
f_N(x)=
\left\{
\begin{array}{ll}
\|x\|, \ & \|x\|\leq N;\\
\\
0, \ & \|x\|\geq N+1
\end{array}
\right.
\end{displaymath}
such that $f_N(x) \leq \|x\|$ for all $x\in \mathbf{R}_+^n$. By (\ref{weak}), we have that
\begin{displaymath}
    \begin{array}{rl}
     & \int_{\mathbf{R}_+^n}f_N(x) \nu_y^{\sigma}(dx)\\
     [2pt]
    = &\displaystyle\lim_{n\rightarrow \infty}\frac{1}{T_n}\int_0^{T_n} \mathbb{E}f_N\big(\Phi(t,\omega,y)\big)dt\\
    [2pt]
    = &\displaystyle\lim_{n\rightarrow \infty}\frac{1}{T_n}\int_{T_y}^{T_n} \mathbb{E}f_N\big(\Phi(t,\omega,y)\big)dt\\
    [3pt]
    \leq& \displaystyle\overline{\lim}_{n\rightarrow \infty}\frac{1}{T_n}\int_{T_y}^{T_n} \mathbb{E}\|\Phi(t,\omega,y)\|dt\\
    [3pt]
    \leq& \displaystyle\overline{\lim}_{n\rightarrow \infty}\frac{1}{2T_n}\int_{T_y}^{T_n} [\mathbb{E}g^2(t,\omega,1)+\mathbb{E}\|\Psi(\int_0^tg(s,\omega,1)ds,y)\|^2]dt\\
    [2pt]
    \leq& \frac{1}{2}\big(1+\frac{\sigma^2}{r}\big)^2+\displaystyle\overline{\lim}_{n\rightarrow \infty}\frac{1}{2T_n}\int_{T_y}^{T_n}\mathbb{E}\|\Psi(\int_0^tg(s,\omega,1)ds,y)\|^2dt\\
    [3pt]
    =& \frac{1}{2}\big(1+\frac{\sigma^2}{r}\big)^2+\displaystyle\overline{\lim}_{n\rightarrow \infty}\frac{1}{2T_n}\int_{T_y}^{T_n}\mathbb{E}\|I_{\{\tau(\omega,t_0)\leq T_y\}}(\omega)\Psi(\int_0^tg(s,\omega,1)ds,y)\|^2dt\\
    &+ \displaystyle\overline{\lim}_{n\rightarrow \infty}\frac{1}{2T_n}\int_{T_y}^{T_n}\mathbb{E}
    \|I_{\{\tau(\omega,t_0)>T_y\}}(\omega)\Psi(\int_0^tg(s,\omega,1)ds,y)\|^2dt\\
    [3pt]
    \leq& \Big(\frac{1}{2}\big(1+\frac{\sigma_0^2}{r}\big)^2+ \frac{C_0^2}{2}+\frac{1}{2}\Big).
    \end{array}
\end{displaymath}
Since $f_N(x)$ tends to $\|x\|$ as $N\rightarrow \infty$, we obtain (\ref{bound}) with $C=\frac{1}{2}\big((1+\frac{\sigma_0^2}{r})^2+ C_0^2+1\big)$ by letting $N\rightarrow \infty$ in the above inequality. This completes the proof.
\end{proof}
\begin{proposition}\label{invariance}
Let $\mu^i:= \nu_{y_0^i}^{\sigma^i}\in \mathcal{M}_S(\sigma_0),\ i= 1,2,\cdots$.
Assume that $\mu^i \stackrel{w}{\rightarrow}\mu$ as $\sigma^i\rightarrow 0$, $i\rightarrow \infty$.
Then $\mu$ is an invariant measure for $\Psi$, that is, $\mu \Psi^{-1}(T,\cdot)=\mu$ for any $T>0$.
\end{proposition}

\begin{proof}
Let $\mu^{i} \stackrel{w}{\rightarrow}
\mu $ as $i \rightarrow \infty$.
  It suffices to prove that for any nonzero $g\in C_b(\mathbf{R}_+^n)$ and $T>0$,
  \begin{equation}\label{wc1}
  \int g(y)\mu \circ\Psi_T^{-1}(dy)=\int g(y)\mu(dy),
  \end{equation}
  equivalently,
  $$\int g(\Psi(T,y))\mu(dy)=\int g(y)\mu(dy).$$
  $\{\mu^{i}\}$ is tight by Proposition \ref{tightp}.  For every $\eta>0$,  there exists a compact set
$K\subset \mathbf{R}_+^n$ such that $\displaystyle\inf_{i}\mu^{i}(K)\geq
1-\frac{\eta}{\|g\|}$.
\begin{displaymath}
    \begin{array}{rl}
     &|\int g(y)\mu^{i}\circ\Psi(T,\cdot)^{-1}(dy)-\int g(y)\mu^{i}(dy)|\\
     [2pt]
    =&|\int g(\Psi(T,y))\mu^{i}(dy)-\int \mathbb{E}g(\Phi^{\sigma^i}(T,\omega,y))\mu^{i}(dy)|\\
     [2pt]
    \leq &\int \mathbb{E}|g(\Psi(T,y))-g(\Phi^{\sigma^i}(T,\omega,y))|\mu^{i}(dy)\\
    [2pt]
    =&\int I_{K}(y)\mathbb{E}|g(\Psi(T,y))-g(\Phi^{\sigma^i}(T,\omega,y))|\mu^{i}(dy)\\
    &+\int I_{K^{c}}(y)\mathbb{E}|g(\Psi(T,y))-g(\Phi^{\sigma^i}(T,\omega,y))|\mu^{i}(dy)\\
     [1pt]
    \leq&\int\mathbb{E}|I_{K}(y)[g(\Psi(T,y))-g(\Phi^{\sigma^i}(T,\omega,y))]|\mu^{i}(dy)+2\eta.
    \end{array}
\end{displaymath}
  It is easy to see that $G:=\Psi(T, K)\subset \mathbf{R}_+^n$ is a compact set. Hence,  there is a $\delta>0$ such that $|g(y)-g(z)|<\eta$ whenever $y\in G, z\in \mathbf{R}_+^n$ with $\|y-z\|<\delta$.
Thus,  one can derive that
\begin{displaymath}
    \begin{array}{rl}
    &\int\mathbb{E}|I_{K}(y)[g(\Psi(T,y))-g(\Phi^{\sigma^i}(T,\omega,y))]|\mu^{i}(dy)\\
     [2pt]
     =&\int_K\mathbb{E}|I_{\{\|\Psi(T, y)-\Phi^{\sigma^i}(T, \omega, y)\|\geq\delta \}}(\omega)[g(\Psi(T,y))-g(\Phi^{\sigma^i}(T,\omega,y)))]|\mu^{i}(dy)\\
      [2pt]
    &+\int_K\mathbb{E}|I_{\{\|\Psi(T,y)-\Phi^{\sigma^i}(T, \omega, y)\|<\delta \}}(\omega)[g(\Psi(T,y))-g(\Phi^{\sigma^i}(T,\omega,y))]|\mu^{i}(dy)\\
     [2pt]
    \leq&2\|g\|\displaystyle\sup_{y\in K}
  \mathbb{P}(\displaystyle \|\Psi(T,y)-\Phi^{\sigma^i}(T, \omega, y)\|\geq\delta)+\eta\\
  <& 2\eta
    \end{array}
\end{displaymath}
for $i$ sufficiently large. Here we have used Proposition \ref{procon}. As a consequence, we have proved that for any $\eta>0$,
$$|\int g(y)\mu^{i}\circ\Psi(T,\cdot)^{-1}(dy)-\int g(y)\mu^{i}(dy)|<4\eta $$
for all sufficiently large $i$. Letting $i\rightarrow \infty$, we obtain that
$$|\int g(y)\mu \circ\Psi(T,\cdot)^{-1}(dy)-\int g(y)\mu (dy)|\leq 4\eta. $$
(\ref{wc1}) follows from $\eta$ being arbitrary. The proof is complete.
\end{proof}

By the Poincar\'{e} recurrence theorem (see, e.g., Ma\~{n}\'{e} \cite[Theorem 2.3, p. 29]{Ma}),
we can obtain the following consequence immediately.
\begin{proposition}\label{PRe}Assume that $\mu$ is an invariant probability measure of
the flow $\Psi$. Let ${\rm supp}(\mu)$ denote the support of $\mu$ and $B(\Psi)$ be
the {\rm Birkhoff}'s center of $\Psi$. Then the support of $\mu$ is contained
in the {\rm Birkhoff}'s center of $\Psi$, {\rm i.e.},
$$ {\rm supp}(\mu)\subset B(\Psi), $$
where $B(\Psi)=\overline{\{y\in\mathbf{R}_+^{n}:y\in\omega_F(y)\}}$.
\end{proposition}

The main result in this section is summarized as follows.

\begin{theorem}\label{main} Let $\Psi$ be dissipative. Then $\mathcal{M}_S(\sigma_0)$ is tight.
If $\mu^i:= \nu_{y_0^i}^{\sigma^i}\in \mathcal{M}_S(\sigma_0),\ i= 1,2,\cdots$, satisfying
$\sigma^i\rightarrow 0$ as $i\rightarrow \infty$, and $\mu^i \stackrel{w}{\rightarrow}\mu$
as $i \rightarrow \infty$,  then $\mu$ is an invariant
measure of $\Psi$, whose support is contained in its Birkhoff's center.
\end{theorem}
\begin{proof}Follows from Propositions \ref{procon}-\ref{PRe}.
\end{proof}

Before finishing this section, we will present applications to Stratonovich stochastic competitive differential equations:
\begin{equation}\label{ssclv}
   dy_i = y_i(r-\displaystyle\sum^n_{j=1}a_{ij}y_j)dt+\sigma y_i\circ dB_t,\ i=1,2,...,n,
 \end{equation}
whose corresponding system without noise is
\begin{equation}\label{dclv}
  \frac{dy_i}{dt} = y_i(r-\displaystyle\sum^n_{j=1}a_{ij}y_j),\
 i=1,2,...,n,
 \end{equation}
 where $r>0, a_{ij}>0,\ i,j=1,2,\cdots,n$.

\begin{theorem}{\rm(Hirsch \cite{H88})}\label{theorem:hirsch}
The system (\ref{dclv}) admits an invariant hypersurface $\Sigma$ (called carrying simplex),
 homeomorphic to the closed unit simplex $S_n=\{y\in \mathbf{R}^n_+: \sum_i y_i=1\}$ by radial projection, such that
 every trajectory in $\mathbf{R}^n_+ \setminus \{O\}$ is asymptotic to one in $\Sigma$. In particular, the system is dissipative, and $\Sigma\bigcup \{O\}$ is the fundamental attractor.
\end{theorem}

Combining the stochastic decomposition formula and Hirsch's carrying simplex theorem, we immediately obtain the following.

\begin{corollary}[Stochastic Carrying Simplex]
Stochastic competitive LV system (\ref{ssclv}) possesses a stochastic carrying simplex $\Sigma(\omega):=u(\omega)\Sigma$, which is invariant for pull-back flow $\Phi(t,\theta_{-t}\omega,y)$ and attracts any nontrivial pull-back trajectory.
\end{corollary}
\begin{theorem}\label{produce}
$\mathcal{M}_S(\sigma_0)$ is produced by all solutions from $\Sigma \cup \{O\}$.
\end{theorem}
\begin{proof}
For any given $\nu_{y}^{\sigma}$ with $y\neq O$, there is a time sequence $T_n\uparrow \infty$ such that (\ref{weak}) holds. By Theorem \ref{theorem:hirsch}, there is a $z\in \Sigma$ such that the solutions $\Psi(t,y)$ and $\Psi(t,z)$ are asymptotic, that is, for any $\epsilon > 0$, there is a $t_0$ such that as $t\geq t_0$
\begin{equation}\label{asym}
\|\Psi(t,y)-\Psi(t,z)\|<\epsilon.
\end{equation}
Without loss of generality, we may assume that
$$ \frac{1}{T_n}\int_0^{T_n}P(t,z,\cdot)dt \stackrel{w}{\rightarrow} \nu_z^{\sigma}\ \ {\rm as}\ n\rightarrow \infty.$$
We claim that $\nu_{y}^{\sigma}=\nu_z^{\sigma}$. It suffices to prove that
for arbitrary $f\in
C_c(\mathbf{R}_+^n)$,
\begin{equation}\label{measureeq}
\int_{\mathbf{R}_+^n}f(x)\nu^\sigma_y(dx)=\int_{\mathbf{R}_+^n}f(x)\nu^\sigma_z(dx).
\end{equation}

Firstly, we will show that (\ref{measureeq}) holds for $f(x)= {\exp}\{-(m_1x_1+m_2x_2+\cdots+m_nx_n)\}$
with any nonnegative integers $m_1, m_2,\cdots,m_n$. Obviously, there is a constant $B$ such that $\|\nabla f(x)\|\leq B$.

Using (\ref{stopping2}), it follows that for any $\epsilon >0$, there is a $T_0=T_0(\epsilon)$ such that
\begin{equation}\label{asym}
\mathbb{P}\big(\tau(\cdot,t_0)> T_0\big)<\epsilon.
\end{equation}
\begin{displaymath}
    \begin{array}{rl}
     & |\int_{\mathbf{R}_+^n}f(x) \nu_y^{\sigma}(dx)-\int_{\mathbf{R}_+^n}f(x) \nu_z^{\sigma}(dx)|\\
     [2pt]
    = &\displaystyle\lim_{n\rightarrow \infty}\big |\frac{1}{T_n}\int_{T_0}^{T_n} \big[\mathbb{E}f \big(\Phi(t,\omega,y)\big)-\mathbb{E}f \big(\Phi(t,\omega,z)\big)\big]dt\big |\\
    \leq& \displaystyle\overline{\lim}_{n\rightarrow \infty}\frac{B}{T_n}\int_{T_0}^{T_n} \mathbb{E}|g(t,\omega,1)|\|\Psi(\int_0^tg(s,\omega,1)ds,y)-\Psi(\int_0^tg(s,\omega,1)ds,z)\|dt\\
    [3pt]
    \leq& B\big(1+\frac{\sigma^2}{r}\big)\displaystyle\overline{\lim}_{n\rightarrow \infty}\frac{1}{T_n}\int_{T_0}^{T_n}\Big(\mathbb{E}\| \Psi(\int_0^tg(s,\omega,1)ds,y)-\Psi(\int_0^tg(s,\omega,1)ds,z)\|^2\Big)^{\frac{1}{2}}dt\\
    [2pt]
    =& B\big(1+\frac{\sigma^2}{r}\big)\Big[\displaystyle\overline{\lim}_{n\rightarrow \infty}\frac{1}{T_n}\int_{T_0}^{T_n}\Big(\mathbb{E}I_{\{\tau\leq T_0\}}\| \Psi(\int_0^tg(s,\omega,1)ds,y)-\Psi(\int_0^tg(s,\omega,1)ds,z)\|^2\Big)^{\frac{1}{2}}dt\\
    +& \displaystyle\overline{\lim}_{n\rightarrow \infty}\frac{1}{T_n}\int_{T_0}^{T_n}\Big(\mathbb{E}I_{\{\tau > T_0\}}\| \Psi(\int_0^tg(s,\omega,1)ds,y)-\Psi(\int_0^tg(s,\omega,1)ds,z)\|^2\Big)^{\frac{1}{2}}dt\Big]\\
    \leq & B\big(1+\frac{\sigma^2}{r}\big)(1+B_{yz})\epsilon
    \end{array}
\end{displaymath}
where $B_{yz}$ is a constant, depending on the bounds for the trajectories $\Psi(t,y)$ and $\Psi(t,z)$. Since $\epsilon$ is arbitrary, (\ref{measureeq}) holds, hence it still holds for linear combination for
these exponent functions. (\ref{measureeq}) follows from the Stone-Weierstrass Theorem immediately.
\end{proof}

\begin{corollary}\label{mainc}
$\mathcal{M}_S(\sigma_0)$ is tight. Let $\mu^i:= \nu_{y_0^i}^{\sigma^i}\in \mathcal{M}_S(\sigma_0)$ with
$y_0^i\neq O$, $i= 1,2,\cdots$. Assume that $\mu^i \stackrel{w}{\rightarrow}\mu$ as
$\sigma^i\rightarrow 0$, $i\rightarrow \infty$.
Then $\mu$ is an invariant
measure of $\Psi$, whose support is contained in its Birkhoff's center.
Moreover, $\mu(\Sigma)=1$.
\end{corollary}
\begin{proof}It is only necessary to show that $\mu(\{O\})=0$, others follow from Theorem \ref{main}.
For every $y_0^i=(y_{0,1}^i,\cdots,y_{0,n}^i)\neq O$, from Theorem \ref{produce}, there exists $z_0^i\in\Sigma$
such that $\mu^i= \nu_{y_0^i}^{\sigma^i}=\nu_{z_0^i}^{\sigma^i}$. Since $\Sigma$ is
invariance, there is a constant $k>0$ such that
$$ \displaystyle\sup_{i}\|\Psi(t,z_0^i)\|\geq k,\ {\rm for\ all}\ t > 0.   $$
Let $R<k$ in (\ref{mob}). Then we utilize the Portmanteau theorem to get that
$$\mu(B_R)\leq \liminf_{i\rightarrow \infty}\mu^i(B_R)\leq
\lim_{\sigma^i \rightarrow 0}\int_0^{\frac{R}{k}}p^{\sigma^{i}}(s)ds =\delta_1([0,\frac{R}{k}])=0,$$
where $\delta_1(\cdot)$ is the Dirac measure at point $\{1\}$ on $\mathbf{R}_+$, and
the second inequality has used (\ref{mob}).
This implies that $\mu(\{O\})=0$.
\end{proof}
\begin{remark} The notion of carrying simplex is just the manifold to carry out turbulence by
Busse et al. \cite{BusseExample,Busse1980science,Busse1980nonlinear}.
\end{remark}
\section{The Complete Classification for 3-Dim Stochastic Competitive LV System}
This section focuses on  three dimensional Stratonovich stochastic competitive LV equations:

\begin{equation}\label{3DSLV}
    \begin{array}{l}
        \displaystyle dy_1=y_1(r-a_{11}y_1-a_{12} y_{2}-a_{13} y_{3})dt+\sigma y_1\circ dB_t,\\
        \noalign{\medskip}
        \displaystyle dy_2=y_2(r-a_{21} y_1-a_{22}y_{2}-a_{23} y_{3})dt+\sigma y_2\circ dB_t,\\
        \noalign{\medskip}
        \displaystyle d y_3=y_3(r-a_{31} y_1-a_{32} y_{2}-a_{33}y_{3})dt+\sigma y_3\circ dB_t.
    \end{array}
\end{equation}
Here $r>0, a_{ij}>0,\ i,j=1,2,3$. We will classify the long-run behavior of stochastic system (\ref{3DSLV}) both in pull-back trajectory  and in stationary measure. To achieve this goal, we have to introduce the classification results for the corresponding deterministic three dimensional competitive  LV equations:
\begin{equation}\label{3DLV}
    \begin{array}{l}
        \displaystyle \frac{dy_1}{dt}=y_1(r-a_{11}y_1-a_{12} y_{2}-a_{13} y_{3}),\\
        \noalign{\medskip}
        \displaystyle \frac{dy_2}{dt}=y_2(r-a_{21} y_1-a_{22}y_{2}-a_{23} y_{3}),\\
        \noalign{\medskip}
        \displaystyle \frac{dy_3}{dt}=y_3(r-a_{31} y_1-a_{32} y_{2}-a_{33}y_{3}),
    \end{array}
\end{equation}
which are given in \cite{JN}.
\subsection{Review  Classification for 3-Dim Deterministic Competitive  LV System}
Zeeman \cite{Z993} classified the stable nullcline classes for general three dimensional competitive LV equations, which permit different intrinsic growth rates. The stable nullcline class means that their boundary equilibria are hyperbolic and have the same local dynamics on $\Sigma$ after a permutation of the indices $\{1,2,3\}.$ She got that
general three dimensional competitive LV equations admit in total $33$ stable nullcline classes. Nevertheless, among the same stable nullcline class, two systems may have different dynamics, global dynamics is unknown for six stable nullcline classes. However,  in the case of the identical intrinsic growth rate, global dynamics for all stable nullcline classes can be classified in the competitive parameters $a_{ij}$, as done in \cite{JN}.

\begin{theorem}{\rm (\cite[Theorem 4.12]{JN})}\label{thm_class}
There are exactly $37$ dynamical classes in $33$ stable nullcline classes for system (\ref{3DLV}). Each class is given by inequalities in competitive coefficients permitting permutation of indices, all trajectories tend to equilibria for classes $1$-$25$, $26$ {\rm a)}, $26$ {\rm c)}, $27$ {\rm a)} and $28$-$33$, a center on $\Sigma$ only occurs in $26$ {\rm b)} and $27$ {\rm b)}, and the heteroclinic cycle attracts all orbits except $L(P)$ in class $27$ {\rm c)}. All are depicted on $\Sigma$ and  presented in Table {\rm \ref{biao0}} in Appendix A.
\end{theorem}
Let us explain what the notations on $\Sigma$ in  Table {\rm \ref{biao0}} mean and how to get global dynamical behavior from the pictures in Table {\rm \ref{biao0}}.  By Hirsch's Theorem \ref{theorem:hirsch}, the carrying simplex $\Sigma$ is homeomorphic to the closed unit simplex $S_3$ by radial projection. So we regard $S_3$ as $\Sigma$ and draw pictures on the standard simplex $S_3$,  where three vertexes $\{R_1, R_2, R_3\}$ represent three axial equilibria for (\ref{3DLV}). Let us take the class 14 in Appendix A (see Fig.\ref{fig:1}) as an example to explain the notations and their meaning. A closed dot $\bullet$ denotes an attracting equilibrium (see $R_2, V_2$) on $\Sigma$,  an open dot $\circ$ denotes the repelling one (see $R_1$) on $\Sigma$, and the intersection of stable and unstable manifolds is a saddle on $\Sigma$ (see $R_3, V_1$). The asymptotic behavior for every trajectory on $\Sigma$ is clearly seen from Fig. \ref{fig:1}.

\begin{figure}[ht]
 \begin{center}
    \includegraphics[width=0.35\textwidth]{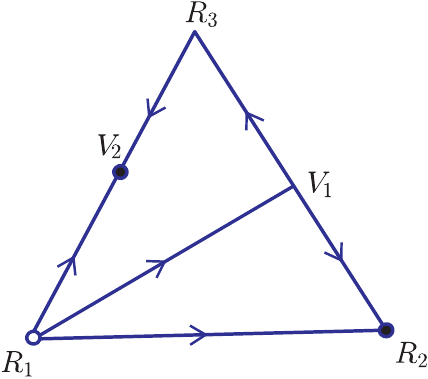}
  \end{center}
\caption{The dynamics in $\Sigma$.} \label{fig:1}
\end{figure}

 Let $\mathcal{A}^{\Sigma}(Q)$ denote the attracting domain for an equilibrium $Q\in \mathcal{E}$ on $\Sigma$. It follows from \cite[Proposition 4.13]{JN} that any pair of nonzero points on $L(y)$ have the same omega limit set. We can obtain the attracting domain for $Q$ as follows
 \begin{equation}\label{attra}
 \mathcal{A}(Q)=\bigcup\{L(y)\setminus \{O\}: y\in \mathcal{A}^{\Sigma}(Q)\}.
\end{equation}
Therefore, the attracting domain for a given $Q$ can be derived by $\mathcal{A}^{\Sigma}(Q)$ drawn in Table {\rm \ref{biao0}} and (\ref{attra}). This has given precise long-term behavior for 34 classes :$1$-$25$, $26$ {\rm a)}, $26$ {\rm c)}, $27$ {\rm a)} and $28$-$33$ in Table {\rm \ref{biao0}}.

It remains to describe the rest three classes: class 26 b), class 27 b), and class 27 c). For this aim, define
\begin{equation}\label{beta_alpha}
\alpha_i=a_{i+1,i+1}-a_{i,i+1},\ \ \beta_i=a_{i,i-1}-a_{i-1,i-1}, \quad i \ \ {\rm mod}\ \ 3,\ {\rm and}
\end{equation}
\vspace{-4mm}
\begin{equation}\label{Def_theta}
\begin{array}{rl}
  \theta:=& \displaystyle\prod_{i=1}^3(a_{i,i-1}-a_{i-1,i-1})-\displaystyle\prod_{i=1}^3(a_{i+1,i+1}-a_{i,i+1}) =\beta_1\beta_2\beta_3-\alpha_1\alpha_2\alpha_3. \\
\end{array}
\end{equation}
The system (\ref{3DLV}) admits nontrivial periodic orbits if and only if $\theta=0$ (see \cite[Theorem 4.3]{JN}), which only occurs in class 26 b) and class 27 b). Both classes possess  heteroclinic cycle connecting three equilibria,  interior of which on $\Sigma$ a family of continuum periodic orbits $\{\Gamma(h)\}:h\in I\}$ are full of. Each closed orbit $\{\Gamma(h)\}$ is the intersection of the carrying simplex $\Sigma$ and invariant cone surface given by
\begin{equation}\label{invariance}
 \Lambda(h):  V(y):=y_1^\mu y_2^\nu y_3^\omega(\beta_2\alpha_3 y_1+\alpha_1\alpha_3 y_2+\beta_1\beta_2 y_3)\equiv h,
\end{equation}
where $\mu=-\beta_2\beta_3/D$, $\nu=-\alpha_1\alpha_3/D$, $\omega=-\alpha_1\beta_2/D$, $D=  (\beta_2\beta_3+\beta_2\alpha_1+\alpha_1\alpha_3)$, and $\alpha_i,\beta_i$ are given in (\ref{beta_alpha}). We depict typical closed orbit and its attracting cone surface for these two classes in Fig.2 and Fig.3. The readers are referred to \cite[Theorem 4.13]{JN} for details.

\begin{figure}[ht]
 \begin{center}
    \includegraphics[width=0.6\textwidth]{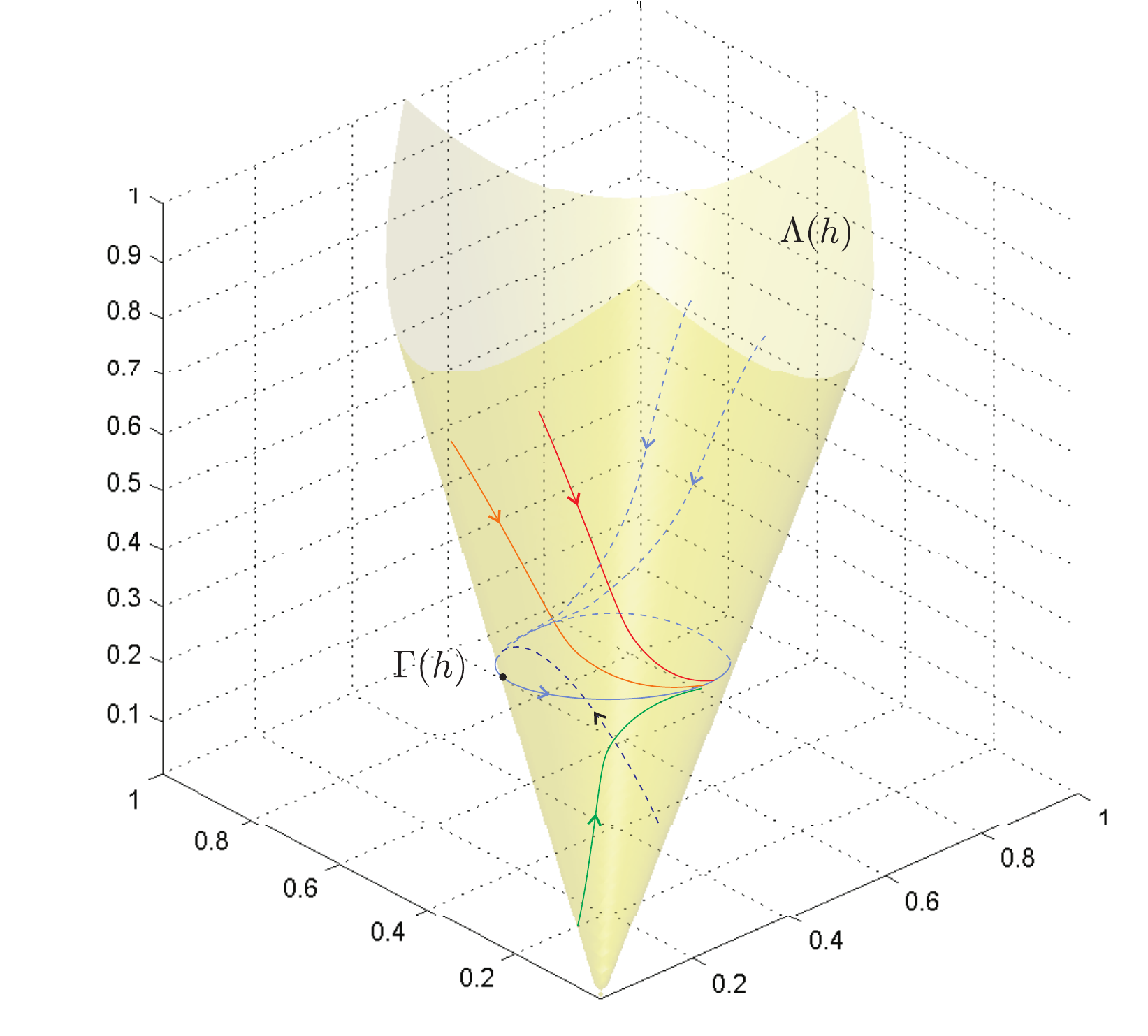}
  \end{center}
\caption{The attracting domain for the closed orbit $\Gamma(h)$ is a cone $\Lambda(h)$.} \label{fig:2}
\end{figure}
\begin{figure}[ht]
 \begin{center}
     \includegraphics[width=0.5\textwidth]{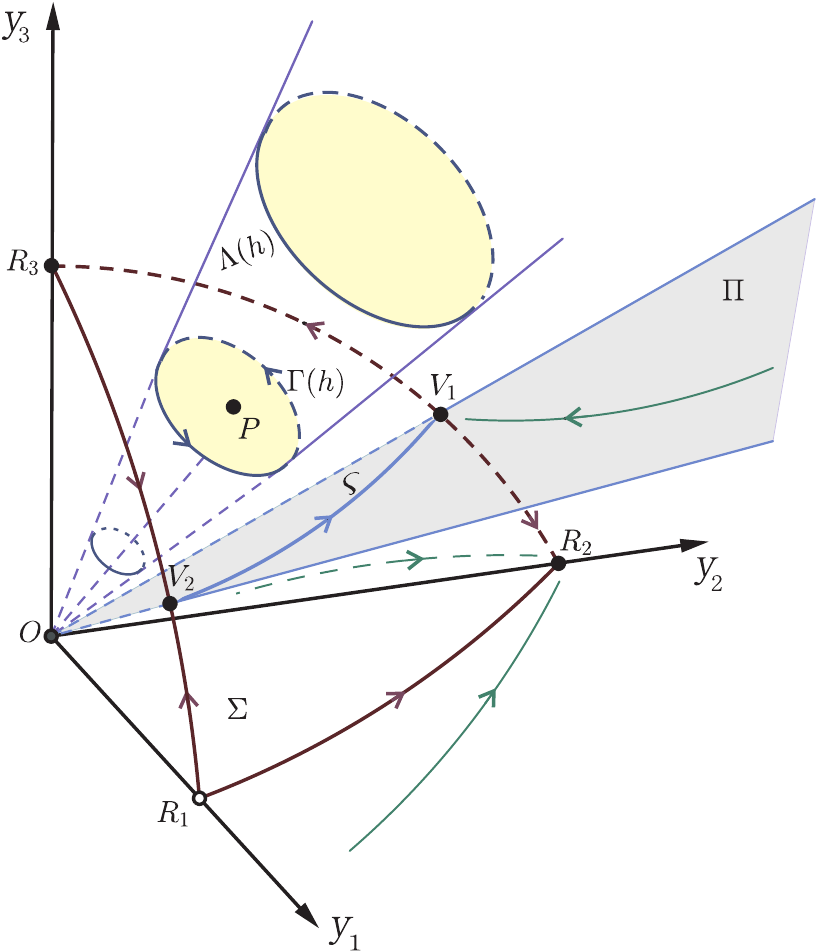}
  \end{center}
\caption{The global phase portraits for a system in class {\rm $26\ b)$}.} \label{fig:3}
\end{figure}

Now we summarize the long-run behavior for these three classes as follows.

\begin{theorem}{\rm(Chen, Jiang, and Niu \cite{JN})}\label{thm4.14}
\begin{enumerate}
\item[{\rm (a)}]   Let the competitive parameters satisfy inequalities in class 26 b) besides $\theta=0$. Then the unique positive equilibrium $P$ attracts $L(P)\setminus \{O\}$; the closed orbit $\Gamma(h)$ attracts $\Lambda(h)\setminus \{O\}$; all other trajectories converge an equilibrium.
\item[{\rm (b)}] Let the competitive parameters satisfy inequalities in class 27 b) besides $\theta=0$. Then the unique positive equilibrium $P$ attracts $L(P)\setminus \{O\}$; the closed orbit $\Gamma(h)$ attracts $\Lambda(h)\setminus \{O\}$.
\item[{\rm (c)}] Let the competitive parameters inequalities in class 27 c) hold. Then  $\mathcal{A}(\mathcal{H})= \mathbf{R}^3_+\setminus L(P)$, where $\mathcal{H}$ is the heteroclinic cycle .
\end{enumerate}
\end{theorem}
\begin{remark} Among 37 classes, the class 27 c) is the only one for statistical limit cycle, or turbulence founded in Busse et al. \cite{BusseExample,Busse1980science,Busse1980nonlinear}, to occur.
\end{remark}
\subsection{The Complete Classification for Long-Run Behavior via Pull-Back Trajectory}

Combing Theorems \ref{long-run}, \ref{thm_class} and \ref{thm4.14},  we can completely classify the long-run behavior of pull-back trajectories for three dimensional stochastic competitive LV system (\ref{3DSLV}).

\begin{theorem}\label{equilibria}
Among classes $1$-$25$, $26$ {\rm a)}, $26$ {\rm c)}, $27$ {\rm a)} and $28$-$33$, each pull-back trajectory $\Phi(t,\theta_{-t}\omega,y)$ converges a random equilibrium. More precisely, for a given equilibrium $Q\in \mathcal{E}$, $\Phi(t,\theta_{-t}\omega,y)\rightarrow u(\omega)Q$ as $t\rightarrow \infty$ for all $y\in \mathcal{A}(Q)$. The same result hold for the remain three classes when $y$ is located in an attracting domain of an equilibrium.
\end{theorem}

\begin{theorem}\label{periodic}
 Assume that $\theta=0$ and the competitive parameters inequalities in class 26 b) or class 27 b) hold. Then
  the pull-back omega-limit set $\Gamma_y(\omega)$ of the trajectory $\Phi(t,\theta_{-t}\omega,y)$ emanating from $y$ is $u(\omega)\Gamma(h)$ if and only if $y\in \Lambda(h)$.
 \end{theorem}
\begin{theorem}\label{heterclinic}
Assume that $\theta>0$ and the competitive parameters inequalities in class 27 c) hold. Then
  the pull-back omega-limit set $\Gamma_y(\omega)$ of the trajectory $\Phi(t,\theta_{-t}\omega,y)$ emanating from $y$ is $u(\omega)\mathcal{H}$ if and only if $y\notin  L(P)$, where $\mathcal{H}$ is the heteroclinic orbit for (\ref{3DLV}).
\end{theorem}
\begin{remark} When a random element is introduced into the time dependence of the system, every sample path not emanating from $L(P)$ cyclically fluctuates in class 27 c). The turbulent fluid state is characterized by three stationary solutions, all of which are unstable so that the actually realized state wanders from a neighbourhood of one of the stationary solutions to that of the next.
\end{remark}
\subsection{The Classification via Stationary Measures}
First, let us consider the case for trajectory of (\ref{3DLV}) to converge to an equilibrium.
\begin{theorem}\label{pdfc}
Let $Q\in \mathcal{E}$. Then for each $y\in \mathcal{A}(Q)$, $P(t,y,\cdot)  {\rightarrow}\
\mu^\sigma_Q$ weakly as $t\rightarrow \infty$, and
\begin{equation}\label{pdfsc}
\lim_{t\rightarrow \infty}P(t,y,A)=\mu^\sigma_Q(A), \ {\rm for\ any}\
A\in
 \mathcal{B}(\mathbf{R}_+^3).
\end{equation}
Moreover, $\mu^\sigma_Q$ is the unique stationary  measure for the
Markov semigroup $P_t$ in $\mathcal{A}(Q)$, and hence, it
is ergodic when the system is restricted on $\mathcal{A}(Q)$, and $\mu^\sigma_Q(\cdot) \stackrel{w}{\rightarrow}
\delta_Q(\cdot)$ as $\sigma \rightarrow 0$. These results are available for classes $1$-$25$, $26$ {\rm a)}, $26$ {\rm c)}, $27$ {\rm a)} and $28$-$33$
as well as any equilibrium in classes 26 b), 27 b) and 27 c) when we restrict the state space in its stable manifold.
\end{theorem}
\begin{proof}
For a given equilibrium $Q\in \mathcal{E}$, it follows from the cone invariance that $\Phi(t,\omega,y)\in \mathcal{A}(Q)$ for any $y\in \mathcal{A}(Q)$.  Then the probability distribution function $P(t,y,\cdot)$ only supports in $\mathcal{A}(Q)$ if $y\in \mathcal{A}(Q)$. Thus, replacing ${\rm Int}\mathbf{R}_+^n$ by $\mathcal{A}(Q)$, we can verify this theorem in the quite same manner as that of Theorem \ref{wce}. We omit it.
\end{proof}

\begin{theorem}\label{allsm}
Suppose that (\ref{3DLV}) is one of systems in  classes $1$-$25$, $26$ {\rm a)}, $26$ {\rm c)}, $27$ {\rm a)} and $28$-$33$. Then all its stationary measures are the convex combinations of ergodic stationary measures $\{\mu^\sigma_Q: Q\in \mathcal{E}\}$.
As $\sigma \rightarrow 0$, all their limiting measures are  the convex combinations of the Dirac measures $\{\delta_Q(\cdot): Q\in \mathcal{E}\}$.
\end{theorem}
\begin{proof}
Assume that (\ref{3DLV}) is one of systems of the given 34 classes. Then $\mathbf{R}_+^3=\bigcup\{\mathcal{A}(Q): Q\in \mathcal{E}\}$. Let $Q\in \mathcal{E}$. Then (\ref{pdfsc}) implies that
\begin{equation}\label{3avtp}
\lim_{T\rightarrow \infty}\frac{1}{T}\int_0^TP(t,y,A)dt
=\mu^\sigma_Q(A),\ {\rm for\ any}\ y\in \mathcal{A}(Q),\ {\rm
and}\ A\in
 \mathcal{B}(\mathbf{R}_+^3).
\end{equation}
Suppose $\nu$ is an arbitrary  stationary measure for the
Markov semigroup $P_t$ in $\mathbf{R}_+^3$. Then for any $ A\in
 \mathcal{B}(\mathbf{R}_+^3)$,
 $$\int_{\mathbf{R}_+^3}\nu(dy)P(t,y,A)=\nu(A),$$
that is,
\begin{equation}\label{3another}
\sum_{Q\in \mathcal{E}}\int_{\mathcal{A}(Q)}\nu(dy)P(t,y,A)=\nu(A).
\end{equation}
Integrating (\ref{3another}) with respect to $t$ from $0$ to $T$ and using (\ref{pdfsc}), we get that
$$\sum_{Q\in \mathcal{E}}\nu(\mathcal{A}(Q))\mu^\sigma_Q(A)=\nu(A).$$
However,
$$\sum_{Q\in \mathcal{E}}\nu(\mathcal{A}(Q))=\nu(\mathbf{R}_+^3)=1.$$
This shows that $\nu$ is the convex combination of $\{\mu^\sigma_Q: Q\in \mathcal{E}\}$. The remain result follows from Theorem \ref {pdfc} immediately.
\end{proof}

\begin{theorem}\label{smucone}
  Assume that $\theta=0$ and the competitive parameters inequalities in class 26 b) or class 27 b) hold. Then there exists a unique ergodic nontrivial stationary measure $\nu_h^{\sigma}$ supporting on the cone

\begin{equation}\label{invariance27}
 \Lambda(h):  V(y):=y_1^\mu y_2^\nu y_3^\omega(\beta_2\alpha_3 y_1+\alpha_1\alpha_3 y_2+\beta_1\beta_2 y_3)\equiv h\in I,
\end{equation}
where $\mu=-\beta_2\beta_3/D$, $\nu=-\alpha_1\alpha_3/D$, $\omega=-\alpha_1\beta_2/D$, $D=  (\beta_2\beta_3+\beta_2\alpha_1+\alpha_1\alpha_3)$, $\alpha_i,\beta_i$ are given in (\ref{beta_alpha}), and $I$ is the feasible image interval for $V$.
 $\nu_h^\sigma$ converges weakly to the Haar measure on the closed orbit $\Gamma(h)$ as $\sigma \rightarrow 0.$
 \end{theorem}

 \begin{proof}
 Fix $h\in I$ and
$y_0\in\Gamma(h)$, define $\varphi(y)=\inf\{t>0,\ \Psi(t,y_0)=y\}$ for any $y\in\Gamma(h)$, and denote $\Upsilon=\varphi(y_0)$ which is the
period of the orbit $\Psi(t,y_0)$. Let $S:=\textbf{R}_+\mod\Upsilon$ denote a circle. Then it is difficult to see that $\varphi: \Gamma(h)\rightarrow S$ is a homeomorphism.  By Theorem \ref{theorem:hirsch}, for any $\Lambda(h)\setminus\{O\}$, there are unique $\lambda>0$ and $z\in \Gamma(h)$ such that $y=\lambda z$. Define $\psi:\ \Lambda(h)\setminus\{O\}\rightarrow \textbf{R}\times S$ by
\begin{eqnarray*}
\psi(y):=\Big(\ln\lambda,\ \varphi(z)\Big),\ y\in \Lambda(h)\setminus\{O\}
\end{eqnarray*}
where $y=\lambda z$ with $\lambda >0$ and $z\in \Gamma(h)$.
 It is easy to see that $\psi: \Lambda(h)\setminus\{O\} \rightarrow \textbf{R}\times S$ is a homeomorphism, its inverse is $\psi^{-1}(x,\tau)=e^{x}\Psi(\tau,y_0)$.

\vskip 0.2cm
For any $y=\lambda z\in \Lambda(h)\setminus\{O\}$ with $\lambda >0$ and $z\in \Gamma(h)$,  it follows from  (\ref{decom}) that
$$
 \Phi(t,\omega,y)=g(t,\omega,\lambda)\Psi(\int_0^tg(s,\omega,\lambda)ds,z).
$$
Obviously, $\Psi(\int_0^tg(s,\omega,g_0)ds,z)\in\Gamma(h)$. Set
$$H(t,\omega,H_0)=\ln(g(t,\omega,\lambda))\ {\rm and}\ T(t,\omega,T_0)=\varphi(\Psi(\int_0^tg(s,\omega,\lambda)ds,z).$$
Denote by $H_0$ and $ T_0$ the numbers $\ln \lambda$ and $\varphi(z)$, respectively. Then applying $\rm It\hat{o}$ formula, we have
\begin{equation}\label{Eq-change}
    \begin{array}{l}
        \displaystyle H(t,H_0)=H_0+r\int_0^t(1-e^{H(s,H_0)})ds+\int_0^t\sigma dB_s,\\
        \noalign{\medskip}
        \displaystyle T(t,T_0)=(T_0+\int_0^te^{H(s,H_0)}ds) \mod \Upsilon.\\
    \end{array}
\end{equation}
By the definition,
$$\psi(\Phi(t,\omega,y))=\Big(\ln(g(t,\omega,\lambda)),\ \varphi(\Psi(\int_0^tg(s,\omega,\lambda)ds,z))\Big)=\Big(H(t,\omega,H_0),\ T(t,\omega,T_0)\Big).$$
The ergodicity  for $\Phi$ on $\Lambda(h)\setminus\{O\}$ is equivalent to that $(H,T)$ is ergodic on $\textbf{R}\times  S$.

\vskip 0.2cm
For any metric space $E$, denote by $\mathcal{B}(E)$ the Borel $\sigma$-field and by $\mathcal{B}_b(E)$ the class of bounded measurable functions on $E$, respectively.
Now,
we prove that $(H,T)$ is strong Feller(SF) and irreducible(I) on $\textbf{R}\times  S$, that is,
\begin{itemize}
\item[(SF)] For any $t>0$, and $F\in\mathcal{B}_b(\textbf{R}\times  S)$,
$$
(H_0,T_0)\in\textbf{R}\times  S\rightarrow\mathbb{E}F(H(t,H_0),T(t,T_0))\ \textrm{is\ continuous};
$$
\item[(I)] For any $t>0,\ (H_0,T_0)\in\textbf{R}\times  S$ and open set $A\in\mathcal{B}(\textbf{R}\times  S)$,
$$
\mathbb{P}\Big((H(t,H_0),T(t,T_0))\in A\Big)>0.
$$
\end{itemize}

Consider the following equations
\begin{eqnarray}\label{SF}
&&H(t,H_0)=H_0+r\int_0^t(1-e^{H(s,H_0)})ds+\int_0^t\sigma dB_s,\nonumber\\
&& \widetilde{T}(t,\widetilde{T}_0)=\widetilde{T}_0+\int_0^te^{H(s,H_0)}ds\nonumber,\\
&& (H_0,\widetilde{T}_0)\in \textbf{R}^2.
\end{eqnarray}
By Theorem 4.2 in \cite{DONG}, the semigroup $(\widetilde{P}_t)_{t\geq0}$ associated with (\ref{SF}) is strong Feller on $\textbf{R}^2$, i.e.
for any $t>0$, $f\in\mathcal{B}_b(\textbf{R}^2)$,
$$
(H_0,\widetilde{T}_0)\in \textbf{R}^2\rightarrow \mathbb{E}f(H(t,H_0),\widetilde{T}(t,\widetilde{T}_0))\textrm{ is continuous}.
$$
Hence, for any $F\in\mathcal{B}_b(\textbf{R}\times  S)$, set $f_F(H,\widetilde{T})=F(H, \widetilde{T}\mod \Upsilon)$, we have
$f_F\in \mathcal{B}_b(\textbf{R}^2)$, and then
\begin{eqnarray*}
(H_0,T_0)\in \textbf{R}\times  S\rightarrow &&\mathbb{E}F(H(t,H_0),T(t,T_0))\\
&&=\mathbb{E}f_F(H(t,H_0),\widetilde{T}(t,T_0))\textrm{ is continuous}.
\end{eqnarray*}
This implies that $(H,T)$ is strong Feller on $\textbf{R}\times  S$.

\vskip 0.3cm

Now we prove that $(H,T)$ is irreducible on $\textbf{R}\times  S$. We only need to prove that for any $a,b\in \textbf{R}$ with $a<b$,
$c,d\in S$ with $c<d$ and $A:=(a,b)\times (c,d)$,
$$
\mathbb{P}\Big((H(t,H_0),T(t,T_0))\in A\Big)>0,\ \ \textrm{for any}\ t>0\ {\rm and}\ (H_0,T_0)\in \textbf{R}\times  S.
$$

\vskip 0.2cm
Define the open set
$$\mathcal{A}(c,d;T_0,\Upsilon)=\bigcup_{n=0}^\infty \Big(c+n\Upsilon-T_0,d+n\Upsilon-T_0\Big)   .$$
By (\ref{sl}) and the definition of $H$,
\begin{eqnarray}\label{eq H}
&&\mathbb{P}\Big((H(t,H_0),T(t,T_0))\in A\Big)\\
&=&
\mathbb{P}\Big(
            \Big(\frac{e^{rt+\sigma B_t}}{e^{-H_0}+r\int_0^te^{rs+\sigma B_s}ds},
            \int_0^t \frac{e^{rs+\sigma B_s}}{e^{-H_0}+r\int_0^se^{rl+\sigma B_l}dl} ds\Big)
            \in  (e^a,e^b)\times \mathcal{A}(c,d;T_0,\Upsilon)\Big)
           \nonumber.
\end{eqnarray}

\vskip 0.2cm

Denote
\begin{eqnarray*}
B=\Big\{&& f\in C([0,t],\textbf{R}_+):  f(0)=1,\  \\
&&\Big(\frac{f(t)}{e^{-H_0}+r\int_0^tf(s)ds},\int_0^t \frac{f(s)}{e^{-H_0}+r\int_0^sf(l)dl} ds \Big)
\in (e^a,e^b)\times\mathcal{A}(c,d;T_0,\Upsilon)
  \Big\}.
\end{eqnarray*}
We claim that $B\neq \emptyset$.
In fact, let
\begin{eqnarray*}
\widetilde{B}=\Big\{
     && h\in C([0,t],\textbf{R}_+):\  h(0)=e^{H_{0}},\  \\
     &&\Big(h(t),\int_0^t h(s) ds \Big) \in (e^a,e^b)\times\mathcal{A}(c,d;T_0,\Upsilon)
  \Big\}.
\end{eqnarray*}
Then we first show that $\widetilde{B}\neq \emptyset$.

Since $\frac{e^{H_{0}}+e^{b}}{2}t$ is a given constant,  we define $\widetilde{n}:=\inf\{n: c+n\Upsilon-T_0\geq \frac{e^{H_{0}}+e^{b}}{2}t\}$, which exists. Choose a constant  $\widetilde{h}$ such that the area in the shadow domain of Fig.\ref{fig:h} is the mean value of $c+\widetilde{n}\Upsilon-T_0$ and $d+\widetilde{n}\Upsilon-T_0$. Thus  $\widetilde{h}=\frac{c+d+2\widetilde{n}\Upsilon-2T_{0}}{t}-\frac{e^{H_{0}}}{2}-\frac{e^{a}+e^{b}}{4}$. Let $h$ be defined as  the broken line in Fig.\ref{fig:h}. Then it is easy to see that $h(0)=e^{H_0}, h(t)= \frac{e^a+e^b}{2}\in (e^a,e^b)$ and the integral $\int_{0}^{t}h(s)ds=\frac{c+d}{2}+\widetilde{n}\Upsilon-T_{0}\in (c+\widetilde{n}\Upsilon-T_0, d+\widetilde{n}\Upsilon-T_0)$.
This implies that $h\in \widetilde{B}$.
\begin{figure}[ht]
 \begin{center}
 \includegraphics[width=0.45\textwidth]{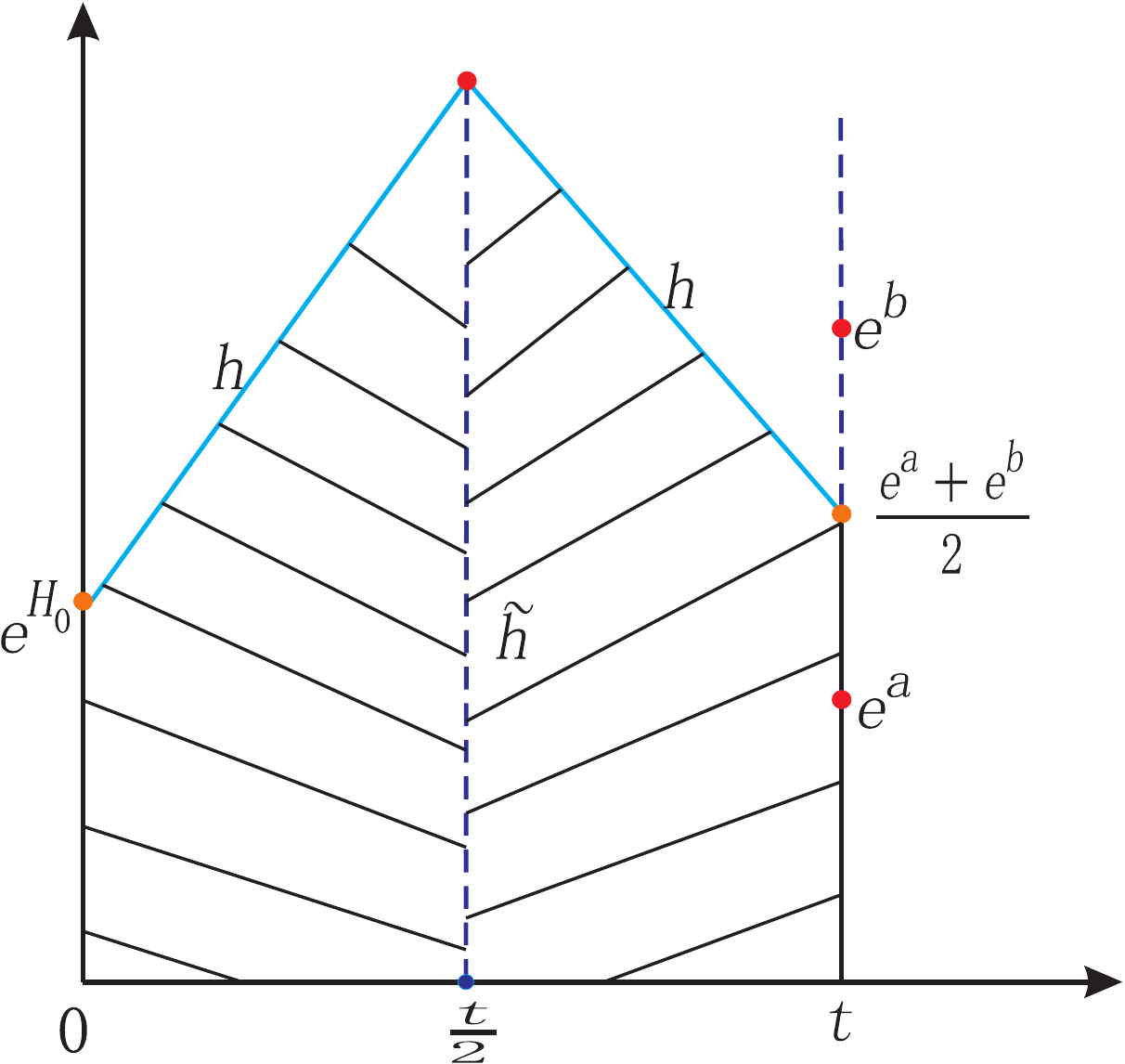}
  \end{center}
\caption{The image of $h$.} \label{fig:h}
\end{figure}

Take $h\in \widetilde{B}$, and let
$$f(s):=e^{-H_{0}}h(s)+re^{-H_{0}}h(s)\int_{0}^{s}h(l)e^{r\int_{l}^{s}h(\tau)d\tau}dl,\ s\in[0,t]
.$$
Then $f\in C([0,t],\textbf{R}_+)$ and
$\int_{0}^{s}f(l)dl=e^{-H_{0}}\int_{0}^{s}h(l)e^{r\int_{l}^{s}h(\tau)d\tau}dl$.
It is clear that $f(0)=h(0)e^{-H_{0}}=1$, $\frac{f(t)}{e^{-H_0}+r\int_0^tf(s)ds}=h(t)\in (e^a,e^b)$
and $$\int_0^t \frac{f(s)}{e^{-H_0}+r\int_0^sf(l)dl} ds=\int_0^t h(s) ds
                  \in
                \mathcal{A}(c,d;T_0,\Upsilon) . $$
This implies that $f\in B$, that is, $B\neq \emptyset$. From the above construction, we know that $f:[0,t]\rightarrow {\rm Int} \textbf{R}_+$.

Assume that $\widetilde{f}\in B$ such that $\widetilde{f}:[0,t]\rightarrow {\rm Int} \textbf{R}_+$. Then we define the map $L:\ C([0,t],\textbf{R}_+)\rightarrow \textbf{R}^2$ by
$$
L(f)=\Big(\frac{f(t)}{e^{-H_0}+r\int_0^tf(s)ds},\ \int_0^t \frac{f(s)}{e^{-H_0}+r\int_0^sf(l)dl} ds\Big).
$$
Then it is easy to see that $L:\ C([0,t],\textbf{R}_+)\rightarrow \textbf{R}^2$ is continuous.
Thus the set $B=L^{-1}\Big( (e^a,e^b)\times\mathcal{A}(c,d;T_0,\Upsilon)  \Big)$ is an open set containing $\widetilde{f}$.
This shows that there exists $\epsilon>0$ such that $C_\epsilon^{\widetilde{f}}=\{g\in C([0,t],\textbf{R}_+),\ g(0)=1,\ \sup_{s\in[0,t]}|g(s)-\widetilde{f}(s)|<\epsilon\}\subset B$. Then there exists an open set $D$ in the space $\{p\in C([0,t],\textbf{R}),\ p(0)=0\}$ with sup norm such that
$$
e^{r\cdot+\sigma p(\cdot)}\in C_\epsilon^{\widetilde{f}},\ \ \forall p\in D.
$$
By (\ref{eq H}),
\begin{eqnarray}
\mathbb{P}\Big((H(t,H_0),T(t,T_0))\in A\Big)
\geq
\mathbb{P}\Big(B(\cdot,\omega)\in D \Big)>0.
\end{eqnarray}
The second inequality follows from the fact of Classical Wiener space (see e.g. \cite{Nua,Shi}).
This implies that $(H,T)$ is irreducible on $\textbf{R}\times  S$ and that $\Phi$ is ergodic  on $\Lambda(h)\setminus\{O\}$. Furthermore, combining \cite[Theorem 3.2.4(iii)]{Da} and the fact that $\nu_h^{\sigma}$ takes zero measure at the origin (see Theorem \ref{lsss}),  $\Phi$ is also ergodic on $\Lambda(h)$. Again using \cite[Theorem 3.2.4(iii)]{Da}, we obtain that $\nu_h^{\sigma}$  is an ergodic stationary measure for $\Phi$ on $\textbf{R}_+^3$.

Finally, applying Corollary \ref{mainc}, we conclude that  $\nu_h^\sigma$ converges weakly to the Haar measure on the closed orbit $\Gamma(h)$ as $\sigma \rightarrow 0.$
\end{proof}

Theorems \ref{pdfc} and \ref{smucone} have given all ergodic stationary measures for all classes except class 27c). From ergodic decomposition theorem \cite[$\S$1.2]{Sko}, every stationary measure is expressed by ergodic stationary measures, which is stated in the following.
 \begin{theorem}\label{ergodic decomp}
 Assume that $\theta=0$ and the competitive parameters inequalities in class 26 b) or class 27 b) hold. Let $\mathcal{E}^{26}=\{O,P,V_1, V_2, R_1, R_2, R_3\}$ and $\mathcal{E}^{27}=\{O,P, R_1, R_2, R_3\}$ denote the equilibria set of the classes 26 and 27, respectively. Then the ergodic stationary measure set is
 $$\mathcal{M}^e(\Phi)=\{\nu_h^{\sigma}:h\in I\}\bigcup \{\mu_Q^{\sigma}:Q\in \mathcal{E}^{i}\}, \ \  i=26, 27.$$
 There exists a probability measure $\nu_{\mu}$ on $\mathcal{M}^e(\Phi)$ such that
  $$\mu(\cdot)=\int_{\mathcal{M}^e(\Phi)}\eta(\cdot)d\nu_{\mu}(\eta)$$
  for any stationary measure $\mu$ of $\Phi$.
 \end{theorem}
 \begin{remark} We can express all stationary measures more precisely.

Define $L^\sigma:\ I\bigcup \mathcal{E}\rightarrow \mathcal{M}^e(\Phi)$ as
\begin{eqnarray*}
 L^\sigma:&&\ h\in I\rightarrow \nu_h^{\sigma}\\
          &&\ Q\in\mathcal{E}\rightarrow \mu_Q^{\sigma}.
\end{eqnarray*}
Then $L^\sigma$ is a bijective mapping. Set
$$
\mathcal{A}=\Big\{\{\vartheta\in I\bigcup \mathcal{E},\ L^\sigma(\vartheta)\in O\},\ \ \forall O\in \mathcal{B}(\mathcal{M}^e(\Phi))\Big\}.
$$
For the above probability measure $\nu_{\mu}$ on $\mathcal{M}^e(\Phi)$, let
$$
m_\mu\Big(\{\vartheta\in I\bigcup \mathcal{E},\ L^\sigma(\vartheta)\in O\}\Big):= \nu_{\mu}(O),\ \ \ \forall O\in \mathcal{B}(\mathcal{M}^e(\Phi)).
$$
Then $m_\mu$ is a probability measure on $(I\bigcup \mathcal{E}, \mathcal{A})$, and
$$
\mu(\cdot)=\int_{I\bigcup \mathcal{E}}L^\sigma(\vartheta)(\cdot)m_{\mu}(d \vartheta).
$$

 \end{remark}

\begin{theorem}\label{speriodic}
 Assume that $\theta=0$ and the competitive parameters inequalities in class 26 b) or class 27 b) hold. Let $\mu^i:= \nu_{h^i}^{\sigma^i},\ i= 1,2,\cdots$ satisfy $\sigma^i\rightarrow 0$ and $\mu^i \stackrel{w}{\rightarrow}\mu$ as $i\rightarrow \infty$, where $\nu_{h^i}^{\sigma^i}$ is the unique ergodic nontrivial stationary measure supporting on the cone $\Lambda(h^i)$. Suppose that each $\Gamma(y_0^i)$ is the closed orbit generating the cone $\Lambda(h^i), \ i= 1,2,\cdots$ and that  $y_0^i \rightarrow y_0$ as $i\rightarrow \infty$. Then if $y_0$ lies in the interior of the heteroclinic cycle $\mathcal{H}$, then $\mu$ is the Haar measure on $\Gamma(y_0)$ for $y_0\neq P$, or the Dirac measure $\delta_P(\cdot)$ at $P$ for $y_0= P$. If $y_0\in \mathcal{H}$, then
\begin{equation}\label{equisupp1}
  \mu(\{E_1, E_2, E_3\})=1,
  \end{equation}
where $E_1, E_2, E_3$ are three equilibria of heteroclinic cycle $\mathcal{H}$ in class 26 b) or class 27 b).
 \end{theorem}
 \begin{proof}
 Let $\mu^i:= \nu_{h^i}^{\sigma^i},\ i= 1,2,\cdots$ satisfy $\sigma^i\rightarrow 0$ and $\mu^i \stackrel{w}{\rightarrow}\mu$ as $i\rightarrow \infty$. Suppose that each $\Gamma(y_0^i)$ is the closed orbit generating the cone $\Lambda(h^i), \ i= 1,2,\cdots$ and that  $y_0^i \rightarrow y_0$ as $i\rightarrow \infty$. We first consider the case that $y_0$ lies in the interior of $\mathcal{H}$ on $\Sigma$ with $y_0\neq P$. If there is a subsequence of $\{y_0^i\}$ lying on $\Gamma(y_0)$,  then Theorem \ref{smucone} implies that $\mu$ is the Haar measure on $\Gamma(y_0)$. Otherwise, we suppose that  all points in $\{y_0^i\}$ are different. If $\{y_0^i\}$ are in the interior of $\Gamma(y_0)$ on $\Sigma$, then we may assume that $y_0^i$ lies in the interior of  $\Gamma(y_0^{i+1})$ on  $\Sigma$ for $i=1,2,\cdots.$ Thus, the first part result deduces that $\mu^{i}(\Lambda(\Gamma(y_0^{i})))=1$ for $i=1,2,\cdots.$ Let $D_i$ and $D_0$ denote the interior of the closed orbits $\Gamma(y_0^{i})$ and $\Gamma(y_0)$ on $\Sigma$, respectively. Then $\mu^k(\Lambda(D_i))=0$ for $1\leq i\leq k$. However, $\Lambda(D_i)\setminus\{O\}$ is an open subset in $\mathbf{R}_+^3$. For each $i\geq 1$, it follows from the Portmanteau theorem (see \cite[Theorem 2.1(iv)]{BILL}) that
\begin{equation}\label{3emo}
\mu(\Lambda(D_i)\setminus\{O\})\leq \liminf_{k\rightarrow \infty}\mu^k(\Lambda(D_i)\setminus\{O\})=0.
\end{equation}
In addition, $\mu(\{O\})=0$ by Corollary \ref{mainc}. This proves that $\mu(\Lambda(D_i))=0$ for each $i\geq 1$. Using the continuity of probability measure, $\mu(\Lambda(D_0))=0$. Again utilizing the Portmanteau theorem (see \cite[Theorem 2.1(iii)]{BILL}) that $\mu(\Lambda(\overline{D_0}))=1.$  Hence $\mu(\Lambda(\Gamma(y_0)))=1$. Since the recurrent points on
$\Lambda(\overline{D_0})$ is $\Gamma(y_0)\cup \{O\}$, $\mu$ is the Haar measure on $\Gamma(y_0).$  The case that $y_0$ lies outside of  $\Gamma(y_0)$ on $\Sigma$ can be treated analogously.

Secondly, we assume that $y_0=P$, $y_0^i\neq P$ for each $i$, and  that $y_0^i$ lies in the interior of  $\Gamma(y_0^{i-1})$ on  $\Sigma$ for $i=2,3,\cdots.$  Then $\mu^k(\Lambda(\overline{D_i}))=1$ for $1\leq i\leq k$.
The Portmanteau theorem (see \cite[Theorem 2.1(iii)]{BILL}) implies that $\mu(\Lambda(\overline{D_i}))=1$  for each given $i$. $\mu(L(P))=1$ follows from the continuity of the probability measure $\mu$.

Thirdly, suppose $y_0\in \mathcal{H}$. Then without loss of generality, we may assume that $y_0^i \in D_{i+1}$ for each $i$. By a similar way, we can obtain (\ref{3emo}) and $\mu(\Lambda(D_i))=0$ for $i=1,2,\cdots.$ Let $D^*$ denote the interior of $\mathcal{H}$ on $\Sigma$. Then  $\mu(\Lambda(D^*))=0$. Applying \cite[Theorem 2.1(iii)]{BILL}, we conclude that $\mu(\Lambda(D^*)\cup \Lambda(\mathcal{H}))=1$, and hence that  $\mu(\Lambda(\mathcal{H}))=1$. It is not difficult to see that the recurrent points on $\mathcal{H}$ are $\{E_1, E_2, E_3, O\}$. Consequently, (\ref{equisupp1}) follows from Corollary \ref{mainc}.
The proof is complete.
\end{proof}

 \begin{theorem}\label{sheterclinic}
Assume that $\theta>0$ and the competitive parameters inequalities in class 27 c) hold. Then $\nu_{y}^{\sigma}$ will support on the three nonnegative axes for any $\nu_{y}^{\sigma}$ with $y\in {\rm Int}\mathbf{R}_+^3\backslash L(P)$.
Let $\mu^i:= \nu_{y_0^i}^{\sigma^i}\in \mathcal{M}_S(\sigma_0) ,\ i= 1,2,\cdots$.
If $\mu^i \stackrel{w}{\rightarrow}\mu$ as $\sigma^i\rightarrow 0$, $i\rightarrow \infty$. Then
\begin{equation}\label{equisupp}
  \mu(\{R_1, R_2, R_3\})=1,
  \end{equation}
where  $R_1, R_2, R_3$ are three axial equilibria for (\ref{3DLV}).
\end{theorem}

\begin{proof}
By Theorem \ref {ssm}, ${\rm supp}(\nu_y^{\sigma})\subset \partial \mathbf{R}_+^3$.
In the following, we shall show that ${\rm supp}(\nu_y^{\sigma})=\cup_{j=1}^3\mathbf{R}_+^j,$ where $\mathbf{R}_+^j$ denotes the nonnegative $y^j-$axis for $j=1,2,3$. For this purpose, we only have to prove
\begin{equation}\label{bounary}
\nu_y^{\sigma}(\partial \mathbf{R}_+^3\setminus \cup_{j=1}^3\mathbf{R}_+^j)=0.
\end{equation}

Suppose that $p(p_1,p_2,0)$ and $q(q_1,q_2,0)$ lie on $\mathcal{H}$ such that $p$ is close to $R_1$ and $q$ is close to $R_2$ as far as we wish. Let $C$ denote the trajectory from $p$ to $q$ and $s$ denote the time length for the trajectory to run from $p$ to $q$. Assume that $\epsilon>0$ is sufficiently small and
$$B_{\epsilon}(C):= \{x\in \mathbf{R}_+^3:{\rm dist}(x,C)< \epsilon\}.$$
Since $\Psi(t,y)$ is asymptotic to the heteroclinic cycle $\mathcal{H}$, $\Psi(t,y)$ will enter and then go out of $B_{\epsilon}(C)$ with infinitely many times. By the continuity of $\Psi$ with respect to initial points, the time length from entering $B_{\epsilon}(C)$ to going out of $B_{\epsilon}(C)$ for the trajectory $\Psi(t,y)$ is approximately $s$. However, since $R_1$, $R_2$ and $R_3$ are saddle, the time for $\Psi(t,y)$ to spend in the vicinity of $R_j$ is proportional to the total time elapsed up to that stage $t$ (see the detail estimation in \cite{May1}).
\vskip 0.3cm

Define $t_1^1=\inf\{t\geq0,\ \Psi(t,y)\in B_{\epsilon}(C)\}$, $t_2^1=\inf\{t\geq t_1^1,\ \Psi(t,y)\notin B_{\epsilon}(C)\}$,
$t_1^n=\inf\{t\geq t_2^{n-1},\ \Psi(t,y)\in B_{\epsilon}(C)\}$, $t_2^n=\inf\{t\geq t_1^n,\ \Psi(t,y)\notin B_{\epsilon}(C)\}$, for $n\geq 2$.
Denote $T_2:=\{t\geq0:\ \Psi(t,y)\in B_{\epsilon}(C)\}$. Then
\begin{eqnarray*}\label{eq time 1}
T_2=\cup_{n=1}^\infty [t_1^n,t_2^n].
\end{eqnarray*}
By the above discussion, we have
\begin{eqnarray}\label{eq time 2}
\lim_{T\rightarrow\infty}\frac{L(T_2\cap[0,T])}{L([0,T])}=0.
\end{eqnarray}
Define
$$
T_2^S(\omega):=\{t\geq0:\ \int_0^tg(s,\omega,g_0)ds\in T_2\}
=
\{t\geq0:\ \Psi(\int_0^tg(s,\omega,g_0)ds,y)\in B_{\epsilon}(C)\}.
$$
Since $\int_0^tg(s,\omega,g_0)ds$ is monotonously increasing,
$$
T_2^S(\omega)=\cup_{n=1}^\infty [t_1^{S,n}(\omega),t_2^{S,n}(\omega)]
$$
where $t_i^{S,n}:=\tau(\omega,t_i^n)$ and hence $\int_0^{t_i^{S,n}(\omega)}g(s,\omega,g_0)ds=t_i^n,\ \ i=1,2$. It is easy to see that $t_i^n, t_i^{S,n}\rightarrow \infty$ as $n\rightarrow \infty$.

We have
\begin{eqnarray}\label{eq time 4}
&&L([t_1^{S,n}(\omega),t_2^{S,n}(\omega)])\nonumber\\
&=&
L([t_1^{S,n}(\omega),t_2^{S,n}(\omega)]\cap T_g^\delta(\omega))
+
L([t_1^{S,n}(\omega),t_2^{S,n}(\omega)]\cap (T_g^\delta(\omega))^c)\nonumber\\
&\leq&
L([t_1^{S,n}(\omega),t_2^{S,n}(\omega)]\cap T_g^\delta(\omega))
+
(t_2^n-t_1^n)/\delta.
\end{eqnarray}

For any fixed $T>0$, let $N:={\rm max}\{n: t_2^{S,n}\leq T\}$. Then
$$T_2^S(\omega)\cap[0,T]=\cup_{n=1}^N[t_1^{S,n},t_2^{S,n}]\cup[t_1^{S,N+1},t_2^{S,N+1}\wedge T].$$
Applying (\ref{eq time 4}), we have
\begin{displaymath}
    \begin{array}{rl}
     &L\big(T_2^S(\omega)\cap[0,T]\big)\\
     [2pt]
    =&\sum_{n=1}^NL\big([t_1^{S,n},t_2^{S,n}]\big)+L\big([t_1^{S,N+1},t_2^{S,N+1}\wedge T]\big)\\
     [2pt]
    \leq & \sum_{n=1}^NL\big([t_1^{S,n},t_2^{S,n}]\cap T_g^\delta(\omega)\big)+L\big([t_1^{S,N+1},t_2^{S,N+1}\wedge T]\cap T_g^\delta(\omega)\big)\\
    +& \frac{1}{\delta}[\sum_{n=1}^N(t_2^n-t_1^n)+{\rm max}\{0, \int_0^Tg(s,\omega,g_0)ds-t_1^{N+1}\}] \\
     [1pt]
    \leq& L\big([0,T]\cap T_g^\delta(\omega)\big)+\frac{1}{\delta}L\big([0,\int_0^Tg(s,\omega,g_0)ds]\cap T_2\big).
    \end{array}
\end{displaymath}

For any $\ T_0>0$, denote
$$\Omega_{T_0}=\{\omega\in\Omega:\ \sup_{t\in[T_0,\infty]}|\frac{1}{t}\int_0^tg(s,\omega,g_0)ds-1|\leq 1\}.$$
$\Omega_{T_0}$ is increasing with respect to $T_0$. By (\ref{sys44.1}),
\begin{eqnarray}\label{eq time 5}
\lim_{T_0\rightarrow\infty}\mathbb{P}(\Omega_{T_0})=1.
\end{eqnarray}
For any $\omega\in \Omega_{T_0}$, $T\geq T_0$,
$$
\int_0^Tg(s,\omega,g_0)ds\leq 2T.
$$
By the above estimations, we have
\begin{eqnarray*}
&&\frac{1}{T}\int_0^TI_{B_{\epsilon}(C)}\Big(\Psi(\int_0^tg(s,\omega,g_0)ds,y)\Big)dt\\
&=&
\frac{L(T_2^S(\omega)\cap[0,T])}{T}\\
&\leq&
\frac{L([0,T]\cap T_g^\delta(\omega))}{T}
+
\frac{L([0,\int_0^Tg(s,\omega,g_0)ds]\cap T_2)}{\delta T}\\
&\leq&
\frac{L([0,T]\cap T_g^\delta(\omega))}{T}
+
\frac{L([0,2T]\cap T_2)}{\delta T}.
\end{eqnarray*}
Then
\begin{eqnarray*}
&&\mathbb{E}\Big(\frac{1}{T}\int_0^TI_{B_{\epsilon}(C)}\Big(\Psi(\int_0^tg(s,\omega,g_0)ds,y)\Big)dt\Big)\\
&=&
\mathbb{E}\Big(\frac{L(T_2^S(\omega)\cap[0,T])}{T}\Big)\\
&\leq&
\mathbb{P}\Big((\Omega_{T_0})^c\Big)
+
\mathbb{E}\Big(\frac{L([0,T]\cap T_g^\delta(\omega))}{T}\Big)
+
\frac{L([0,2T]\cap T_2)}{\delta T}.
\end{eqnarray*}
Combining this with (\ref{eq time 2}), (\ref{eq time 3}) and (\ref{eq time 5}),
\begin{eqnarray}
\overline{\lim}_{T\rightarrow\infty}\mathbb{E}\Big(\frac{1}{T}\int_0^TI_{B_{\epsilon}(C)}\Big(\Psi(\int_0^tg(s,\omega,g_0)ds,y)\Big)dt\Big)
\leq
\mu_g((0,\delta]),
\end{eqnarray}
which implies that $$\lim_{T\rightarrow\infty}\mathbb{E}\Big(\frac{1}{T}\int_0^TI_{B_{\epsilon}(C)}\Big(\Psi(\int_0^tg(s,\omega,g_0)ds,y)\Big)dt\Big)=0.$$
We obtain that $\nu_y^{\sigma}\big(\Lambda(B_{\epsilon}(C))\big)=0$, which implies that $\nu_y^{\sigma}$  takes zero measure on the interior of nonnegative $(y_1, y_2)-$plane. Similarly, $\nu_y^{\sigma}$ takes zero measure on the interiors of other two nonnegative planes. This proves that  ${\rm supp}(\nu_y^{\sigma})=\cup_{j=1}^3\mathbf{R}_+^j$.

Suppose that $\sigma^i$, $y_0^i$ and $\mu^i$ satisfy the conditions in the theorem. Then $\mu^i\big(\cup_{j=1}^3\mathbf{R}_+^j\big)=1$. Applying the Portmanteau theorem, we derive that
$$\mu\big(\cup_{j=1}^3\mathbf{R}_+^j\big) \geq \limsup_{i\rightarrow \infty}\mu^i\big(\cup_{j=1}^3\mathbf{R}_+^j\big)=1.$$
All recurrent points for $\Psi$ on $\cup_{j=1}^3\mathbf{R}_+^j$ are $\{O,R_1,R_2,R_3\}$. By Corollary \ref{mainc}, $\mu(\{O\})=0$. So we conclude that $\mu(\{R_1,R_2,R_3\})=1$ by Corollary \ref{mainc}. The proof is complete.
\end{proof}
\begin{remark}

According to Busse et al. \cite{BusseExample,Busse1980science,Busse1980nonlinear}, 27 c) corresponds to the case for turbulence to occur in deterministic system. Theorem \ref{heterclinic} illustrates that almost every pull-back trajectory  cyclically oscillates around the boundary of the stochastic carrying simplex which is characterized by three unstable stationary solutions. Theorem \ref{sheterclinic} only describes the support of stationary measures. Appendix B will show that stochastic turbulence has nonuniqueness and nonergodicity characteristics in the limit of time average of probability measures. We will reveal that the essential reason for both peculiar characteristics is that solutions concentrate around $R_1, R_2, R_3$ very long time (approximately infinite) with probability nearly one.
\end{remark}

\section{Conclusions and Discussion}
This paper has proved the stochastic decomposition formula: every solution process for stochastic Lotka-Volterra  systems with identical intrinsic growth rate is expressed in terms of a solution for the corresponding deterministic Lotka-Volterra system without noise perturbation multiplied by an appropriate solution process of the scalar Logistic equation with the same type noise perturbation. Using this decomposition,  we have shown that every pull-back omega limit set for the considered stochastic Lotka-Volterra  systems is an omega limit set of the corresponding deterministic Lotka-Volterra system  multiplied by the random equilibrium of the scalar stochastic Logistic equation with the same type of noise. This illustrates the interesting dynamics in trajectory of deterministic Lotka-Volterra system is preserved if identical intrinsic growth rate is perturbed by a white noise. Employing the stochastic decomposition formula, the Khasminskii theorem and  the Portmanteau theorem, it is shown that a bounded orbit for deterministic Lotka-Volterra system  deduces the existence of a stationary measure for stochastic Lotka-Volterra  system supported in a lower dimensional cone which consists of all rays connecting the origin and all points in the omega limit set of this orbit.  In particular, an equilibrium $Q$ for deterministic Lotka-Volterra system produces a stationary measure $\mu^\sigma_Q$ for stochastic Lotka-Volterra  system supported in a ray connecting the origin and the equilibrium $Q$, which has a continuous distribution function and weakly converges to the Dirac measure at $Q$ as $\sigma$ vanishes by the Weierstrass theorem. Besides, that a trajectory $\Psi(t,y)$ converges to $Q$ is equivalent to that the pull-back trajectory through $y$ converges to the stationary solution corresponding to $\mu^\sigma_Q$. This means that the probability transition function $P(t,y,A)$ converges to $\mu^\sigma_Q(A)$ for any Borel set $A$ as the time tends to infinity, which helps us to provide the necessary and sufficient conditions for Markov semigroup to have a unique and ergodic stationary measure.  A closed orbit $\Psi(t,y)$ for deterministic Lotka-Volterra system  deduces the existence of stationary measure $\nu^\sigma$ for stochastic Lotka-Volterra  system supported in a two dimensional cone surface with the origin as the vertex decided by this closed orbit, which weakly converges to the Haar measure on the closed orbit as $\sigma$ vanishes. The solutions for stochastic Lotka-Volterra  system are invariant when restricted on this cone surface.  As above, any stationary measure is always not regular. This paper reveals the close connection between the dynamics of deterministic Lotka-Volterra system and long-run behavior for
stochastic Lotka-Volterra  system. This makes us to be able to construct many examples to possess a continuum of stationary measures or multiple isolated stationary measures or even others, which are not obtained by the way of convex combination of them.

Suppose that the deterministic Lotka-Volterra system ${\rm (E_{0})}$  is dissipative. Then we prove that the set of stationary measures with small noise intensity is tight, and that their limiting measures in weak topology are invariant with respect to the flow of  ${\rm (E_{0})}$  as the noise intensity $\sigma$ tends to zero, whose supports are contained in the Birkhoff center of ${\rm (E_{0})}$. This means that on the global attractor of  ${\rm (E_{0})}$ any limiting measure takes the complement of the Birkhoff center measure zero. In the case that ${\rm (E_{0})}$  is competitive, the global attractor is the compact invariant set surrounded by the carrying simplex $\Sigma$ and the boundary of $\mathbf{R}_+^n$.  However, the Birkhoff center consists of the recurrent points in the carrying simplex and the origin. This means that our result gives much more precise description for support of limiting measures than that Huang, Ji, Liu and Yi \cite{HJLY5} have given.

Finally, we provide the complete dynamics classification for three dimensional competitive ${\rm (E_{\sigma})}$ both in pull-back trajectory and in stationary motion. There are exactly 37 dynamic scenarios in terms of competitive coefficients. Among them,  each pull-back trajectory in 34 classes is asymptotically stationary, but possibly different stationary solution for different trajectory in same class. For any given system in these 34 classes, all its stationary measures are the convex combinations of $\{\mu^\sigma_Q: Q\in \mathcal{E}\}$. As $\sigma \rightarrow 0$, all their limiting measures are  the convex combinations of the Dirac measures $\{\delta_Q(\cdot): Q\in \mathcal{E}\}$. Two of the remain classes possess a family of stochastic closed orbits, and there exists a continuum of invariant cone surfaces $\Lambda(h)$ decided by the origin and the closed orbits for the corresponding deterministic Lotka-Volterra system. For each $\Lambda(h)$, the system admits a unique nontrivial ergodic stationary measures $\nu_h^{\sigma}$ supported in it, which weakly converge to the Haar measures of periodic orbits  as the noise intensity tends to zero. In addition,  any limiting measure for a sequence of stationary measures $\nu_{y_0^i}^{\sigma^i},\ i= 1,2,\cdots,$ satisfying $\sigma^i\rightarrow 0$ and $y_0^i \rightarrow \mathcal{H}$ with $y_0^i\in {\rm Int}\mathbf{R}_+^3$, will support in the three equilibria on $\mathcal{H}$. In the final class, the most interesting and complicated one,  almost every pull-back trajectory  cyclically oscillates around the boundary of the stochastic carrying simplex which is characterized by three unstable stationary solutions $u(\theta_t\omega)R_i,\ i=1,2,3$.  The time average probability measure for transition probability function of a solution not passing through the ray connecting the origin and the positive equilibrium of ${\rm (E_{0})}$ does not weakly converge, but has infinite limit measures which are not ergodic and support in three positive axes. As the noise intensity tends to zero, these stationary measures weakly converge to a convex combination of Dirac measures on three unstable axis equilibrium.  We will reveal in the Appendix B that the essential reason for these peculiar characteristics is that solutions concentrate around $R_1, R_2, R_3$ very long time (approximately infinite) with probability nearly one. All these are subject to the turbulent characteristics.  This rigorously proves that a stochastic version for so called statistical limit cycle exists and  that the turbulence in a fluid layer heated from below and rotating about a vertical axis is robust under stochastic disturbances.

Observing Table \ref{biao0},  there are four classes to possess a heteroclinic cycle, which are 26 b), 27 a), 27 b), and
27 c). In the classes 26 b), 27 b), and 27 c), it holds that any limiting measure for a sequence of stationary measures $\nu_{y_0^i}^{\sigma^i},\ i= 1,2,\cdots,$ satisfying $\sigma^i\rightarrow 0$ and $y_0^i \rightarrow \mathcal{H}$ with $y_0^i\in {\rm Int}\mathbf{R}_+^3$, will support in the three equilibria on $\mathcal{H}$. However, in the class 27 a), $\nu_{y}^{\sigma}=\mu^\sigma_P$ for any $y\in {\rm Int}\mathbf{R}_+^3$. What is the reason for this difference? The reason is that in the classes 26 b), 27 b), and 27 c), the heteroclinic cycle $\mathcal{H}$ is either neutrally stable, or asymptotically stable, while in the class 27 a), the heteroclinic cycle $\mathcal{H}$ is unstable. Solutions for stochastic ordinary differential equations (SODEs) is usually defined in the nonnegative time, therefore, its probability transition function is defined in the nonnegative time, which causes $\nu_{y}^{\sigma}=\mu^\sigma_P$ for any $y\in {\rm Int}\mathbf{R}_+^3$. If one considers two-sided Brownian motion, then solutions for SODEs can be defined in the entire real time (see \cite{A}). This consideration permits $T_n\rightarrow -\infty$ in (\ref{0pms}) and may prove the existence of stationary measures in generalized meaning supported in the boundary of the first orthant, which weakly converges to an invariant measure supported in three equilibria on $\mathcal{H}$.

Before finishing this paper, we point out that although all results are presented for Stratonovich stochastic differential equations (\ref{sslv}) they are valid for It\^{o} stochastic differential equations (\ref{slv}) as long as $\sigma^2<2r$.

\begin{acknowledgements} This work was supported
by the National Natural Science Foundation of China (NSFC)(Nos. 11371252, 11271356, 11371041, 11431014, 11401557),
Research and Innovation Project of Shanghai
Education Committee (No. 14zz120),
Key Laboratory of Random Complex Structures and Data Science,
Academy of Mathematics and Systems Science, CAS,
the Fundamental Research Funds for the Central Universities (No. WK0010000048), and Shanghai Gaofeng Project for University Academic Program Development.

The authors are greatly grateful for Professors Renming Song and Zuohuan Zheng for their valuable discussions.
\end{acknowledgements}
\newpage
\section{Appendix A. The Complete Dynamical Classification for both Autonomous and Stochastic Three Dimensional Competitive LV Systems with Identical Intrinsic Growth Rate on the Carrying Simplex}
\begin{center}
  \begin{longtable}{c@{\extracolsep{\fill}}c@{\extracolsep{\fill}}c}
\caption*{Table 1. Description  How to Understand the Dynamics on the Carrying Simplex. \\[4pt]
{\bf Autonomous Case}: The total of $37$ dynamical classes among the $33$ stable nullcline equivalence classes for (\ref{3DLV}), where the parameters $a_{ij}$ and $\Sigma$ are given by a representative system of that class. The notation $\bullet$ and $\circ$ denote an attractor and a repeller on $\Sigma$, respectively, while a saddle on $\Sigma$ is the intersection of its stable and unstable manifolds (We refer to \cite{JN}).
 \\{\bf Stochastic Perturbation Case}:  The carrying simplex $\Sigma$ in  autonomous case is replaced by the fiber $u(\omega)\Sigma~(\omega\in \Omega)$; an equilibrium $Q$, a closed orbit $\Gamma$ and a heteroclinic cycle $\mathcal{H}$ are understood as $u(\omega)Q$, $u(\omega)\Gamma$ and $u(\omega)\mathcal{H}$, respectively. All trajectories are understood pull-back ones.}\\[-2pt]
        \hline
         Class & The Corresponding Parameters   & Phase Portrait in $\Sigma$\\
        \hline
        \endfirsthead
        \caption[]{(continued)}\\
        \hline
        Class & The Corresponding Parameters   & Phase Portrait in $\Sigma$\\
        \hline
&&\\
        \endhead
        \hline
        \endfoot
        \endlastfoot
&&\\
1 &
\begin{tabular}{l}
 {\tiny $a_{11}<a_{21}, a_{11}<a_{31}, a_{22}>a_{12}, a_{22}>a_{32}, a_{33}>a_{13}, a_{33}<a_{23}$}
\end{tabular}
&
    \parbox{2cm}{\vspace{2pt}\includegraphics[width=1.6cm,height=1.4cm]{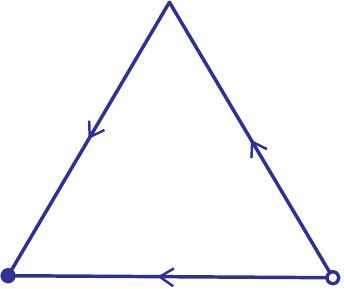}} \\

&&\\
    2 &
\begin{tabular}{l}
 {\tiny(i)~ $a_{11}<a_{21}, a_{11}<a_{31}, a_{22}<a_{12}, a_{22}>a_{32}, a_{33}>a_{13}, a_{33}<a_{23}$}\\
{\tiny(ii)~ $a_{31}(a_{22}-a_{12})+a_{32}(a_{11}-a_{21})-(a_{11}a_{22}-a_{12}a_{21})>0$}
\end{tabular}
 &
    \parbox{2cm}{\vspace{2pt}\includegraphics[width=1.6cm,height=1.4cm]{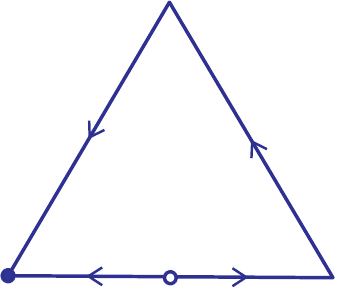}} \\
&&\\
    3 &
\begin{tabular}{l}
 {\tiny(i)~ $a_{11}<a_{21}, a_{11}<a_{31}, a_{22}>a_{12}, a_{22}<a_{32}, a_{33}>a_{13}, a_{33}<a_{23}$}\\
{\tiny(ii)~ $a_{12}(a_{33}-a_{23})+a_{13}(a_{22}-a_{32})-(a_{22}a_{33}-a_{23}a_{32})>0$}
\end{tabular}
 &
    \parbox{2cm}{\vspace{2pt}\includegraphics[width=1.6cm,height=1.4cm]{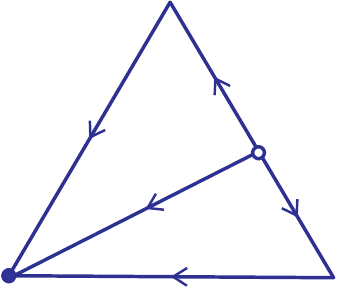}\vspace{2pt}}\\
&&\\
4 &
\begin{tabular}{l}
 {\tiny(i)~ $a_{11}>a_{21}, a_{11}<a_{31}, a_{22}>a_{12}, a_{22}<a_{32}, a_{33}>a_{13}, a_{33}<a_{23}$}\\
{\tiny(ii)~ $a_{12}(a_{33}-a_{23})+a_{13}(a_{22}-a_{32})-(a_{22}a_{33}-a_{23}a_{32})>0$}\\
{\tiny(iii)~$a_{31}(a_{22}-a_{12})+a_{32}(a_{11}-a_{21})-(a_{11}a_{22}-a_{12}a_{21})>0$}
\end{tabular}
 &
    \parbox{2cm}{\vspace{2pt}\includegraphics[width=1.6cm,height=1.4cm]{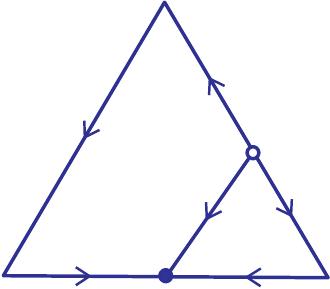}\vspace{2pt}} \\
&&\\
    5 &
\begin{tabular}{l}
 {\tiny(i)~ $a_{11}>a_{21}, a_{11}>a_{31}, a_{22}>a_{12}, a_{22}<a_{32}, a_{33}<a_{13}, a_{33}>a_{23}$}\\
{\tiny(ii)~ $a_{31}(a_{22}-a_{12})+a_{32}(a_{11}-a_{21})-(a_{11}a_{22}-a_{12}a_{21})>0$}
\end{tabular}
 &
    \parbox{2cm}{\vspace{2pt}\includegraphics[width=1.6cm,height=1.4cm]{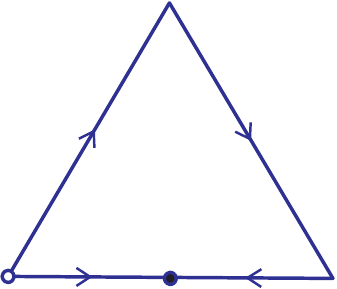}\vspace{2pt}} \\
&&\\
6 &
\begin{tabular}{l}
 {\tiny(i)~ $a_{11}>a_{21}, a_{11}>a_{31}, a_{22}<a_{12}, a_{22}>a_{32}, a_{33}<a_{13}, a_{33}>a_{23}$}\\
{\tiny(ii)~ $a_{12}(a_{33}-a_{23})+a_{13}(a_{22}-a_{32})-(a_{22}a_{33}-a_{23}a_{32})>0$}
\end{tabular}
 &
    \parbox{2cm}{\vspace{2pt}\includegraphics[width=1.6cm,height=1.4cm]{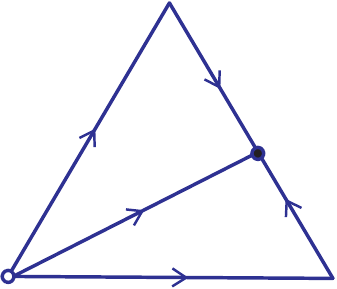}\vspace{2pt}} \\
&&\\
    7 &
\begin{tabular}{l}
 {\tiny(i)~ $a_{11}>a_{21}, a_{11}>a_{31}, a_{22}>a_{12}, a_{22}>a_{32}, a_{33}<a_{13}, a_{33}<a_{23}$}\\
{\tiny(ii)~ $a_{31}(a_{22}-a_{12})+a_{32}(a_{11}-a_{21})-(a_{11}a_{22}-a_{12}a_{21})<0$}
\end{tabular}
 &
    \parbox{2cm}{\vspace{2pt}\includegraphics[width=1.6cm,height=1.4cm]{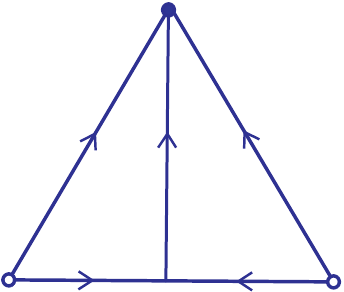}\vspace{2pt}} \\
&&\\
8 &
\begin{tabular}{l}
 {\tiny(i)~ $a_{11}>a_{21}, a_{11}>a_{31}, a_{22}<a_{12}, a_{22}<a_{32}, a_{33}>a_{13}, a_{33}<a_{23}$}\\
{\tiny(ii)~ $a_{12}(a_{33}-a_{23})+a_{13}(a_{22}-a_{32})-(a_{22}a_{33}-a_{23}a_{32})>0$}\\
{\tiny(iii)~$a_{21}(a_{33}-a_{13})+a_{23}(a_{11}-a_{31})-(a_{11}a_{33}-a_{13}a_{31})<0$}
\end{tabular}
 &
    \parbox{2cm}{\vspace{2pt}\includegraphics[width=1.6cm,height=1.4cm]{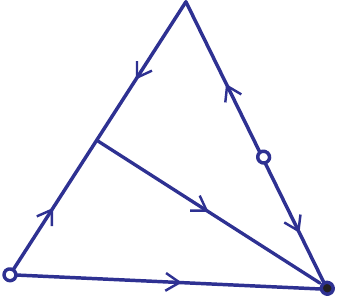}\vspace{2pt}} \\
&&\\
    9 &
\begin{tabular}{l}
 {\tiny(i)~ $a_{11}>a_{21}, a_{11}>a_{31}, a_{22}>a_{12}, a_{22}>a_{32}, a_{33}<a_{13}, a_{33}>a_{23}$}\\
{\tiny(ii)~ $a_{12}(a_{33}-a_{23})+a_{13}(a_{22}-a_{32})-(a_{22}a_{33}-a_{23}a_{32})>0$}\\
{\tiny(iii)~$a_{31}(a_{22}-a_{12})+a_{32}(a_{11}-a_{21})-(a_{11}a_{22}-a_{12}a_{21})<0$}
\end{tabular}
 &
    \parbox{2cm}{\vspace{2pt}\includegraphics[width=1.6cm,height=1.4cm]{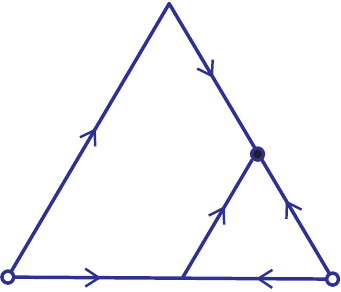}\vspace{2pt}} \\
&&\\
10 &
\begin{tabular}{l}
{\tiny (i)~ $a_{11}>a_{21}, a_{11}>a_{31}, a_{22}>a_{12}, a_{22}>a_{32}, a_{33}<a_{13}, a_{33}>a_{23}$}\\
{\tiny(ii)~ $a_{12}(a_{33}-a_{23})+a_{13}(a_{22}-a_{32})-(a_{22}a_{33}-a_{23}a_{32})<0$}\\
{\tiny(iii)~$a_{31}(a_{22}-a_{12})+a_{32}(a_{11}-a_{21})-(a_{11}a_{22}-a_{12}a_{21})>0$}
\end{tabular}
 &
    \parbox{2cm}{\vspace{2pt}\includegraphics[width=1.6cm,height=1.4cm]{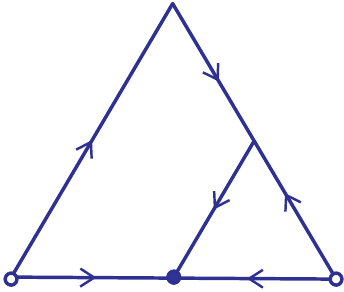}\vspace{2pt}} \\
&&\\
    11 &
\begin{tabular}{l}
 {\tiny(i)~ $a_{11}>a_{21}, a_{11}>a_{31}, a_{22}>a_{12}, a_{22}<a_{32}, a_{33}>a_{13}, a_{33}<a_{23}$}\\
{\tiny(ii)~ $a_{12}(a_{33}-a_{23})+a_{13}(a_{22}-a_{32})-(a_{22}a_{33}-a_{23}a_{32})>0$}\\
{\tiny(iii)~$a_{21}(a_{33}-a_{13})+a_{23}(a_{11}-a_{31})-(a_{11}a_{33}-a_{13}a_{31})<0$}\\
{\tiny(iv)~$a_{31}(a_{22}-a_{12})+a_{32}(a_{11}-a_{21})-(a_{11}a_{22}-a_{12}a_{21})>0$}
\end{tabular}
 &
    \parbox{2cm}{\vspace{2pt}\includegraphics[width=1.6cm,height=1.4cm]{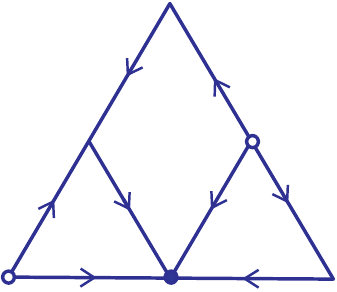}\vspace{2pt}} \\
&&\\
12 &
\begin{tabular}{l}
 {\tiny(i)~ $a_{11}>a_{21}, a_{11}>a_{31}, a_{22}>a_{12}, a_{22}>a_{32}, a_{33}>a_{13}, a_{33}>a_{23}$}\\
{\tiny(ii)~ $a_{12}(a_{33}-a_{23})+a_{13}(a_{22}-a_{32})-(a_{22}a_{33}-a_{23}a_{32})<0$}\\
{\tiny(iii)~$a_{21}(a_{33}-a_{13})+a_{23}(a_{11}-a_{31})-(a_{11}a_{33}-a_{13}a_{31})<0$}\\
{\tiny(iv)~$a_{31}(a_{22}-a_{12})+a_{32}(a_{11}-a_{21})-(a_{11}a_{22}-a_{12}a_{21})>0$}
\end{tabular}
 &
    \parbox{2cm}{\vspace{2pt}\includegraphics[width=1.6cm,height=1.4cm]{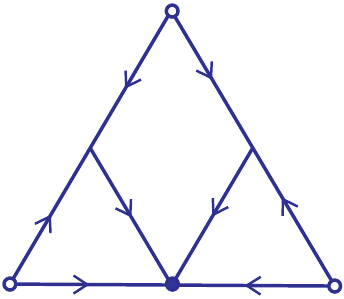}\vspace{2pt}} \\
&&\\
    13 &
\begin{tabular}{l}
 {\tiny(i)~ $a_{11}>a_{21}, a_{11}>a_{31}, a_{22}<a_{12}, a_{22}<a_{32}, a_{33}<a_{13}, a_{33}<a_{23}$}\\
{\tiny(ii)~ $a_{12}(a_{33}-a_{23})+a_{13}(a_{22}-a_{32})-(a_{22}a_{33}-a_{23}a_{32})<0$}
\end{tabular}
 &
    \parbox{2cm}{\vspace{2pt}\includegraphics[width=1.6cm,height=1.4cm]{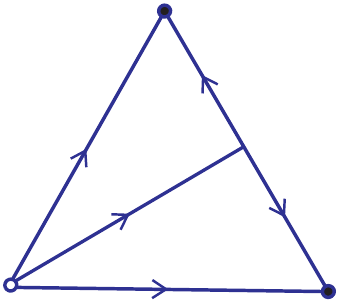}\vspace{2pt}} \\
&&\\
14 &
\begin{tabular}{l}
{\tiny (i)~ $a_{11}>a_{21}, a_{11}>a_{31}, a_{22}<a_{12}, a_{22}<a_{32}, a_{33}>a_{13}, a_{33}<a_{23}$}\\
{\tiny(ii)~ $a_{12}(a_{33}-a_{23})+a_{13}(a_{22}-a_{32})-(a_{22}a_{33}-a_{23}a_{32})<0$}\\
{\tiny(iii)~$a_{21}(a_{33}-a_{13})+a_{23}(a_{11}-a_{31})-(a_{11}a_{33}-a_{13}a_{31})>0$}
\end{tabular}
 &
    \parbox{2cm}{\vspace{2pt}\includegraphics[width=1.6cm,height=1.4cm]{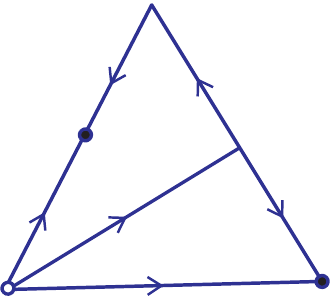}\vspace{2pt}} \\
&&\\
    15 &
\begin{tabular}{l}
 {\tiny(i)~ $a_{11}<a_{21}, a_{11}<a_{31}, a_{22}<a_{12}, a_{22}<a_{32}, a_{33}>a_{13}, a_{33}<a_{23}$}\\
{\tiny(ii)~ $a_{12}(a_{33}-a_{23})+a_{13}(a_{22}-a_{32})-(a_{22}a_{33}-a_{23}a_{32})>0$}\\
{\tiny(iii)~$a_{31}(a_{22}-a_{12})+a_{32}(a_{11}-a_{21})-(a_{11}a_{22}-a_{12}a_{21})<0$}
\end{tabular}
 &
    \parbox{2cm}{\vspace{2pt}\includegraphics[width=1.6cm,height=1.4cm]{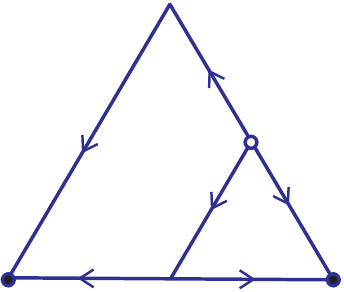}\vspace{2pt}} \\
&&\\
16 &
\begin{tabular}{l}
 {\tiny(i)~ $a_{11}<a_{21}, a_{11}<a_{31}, a_{22}<a_{12}, a_{22}<a_{32}, a_{33}>a_{13}, a_{33}<a_{23}$}\\
{\tiny(ii)~ $a_{12}(a_{33}-a_{23})+a_{13}(a_{22}-a_{32})-(a_{22}a_{33}-a_{23}a_{32})<0$}\\
{\tiny(iii)~$a_{31}(a_{22}-a_{12})+a_{32}(a_{11}-a_{21})-(a_{11}a_{22}-a_{12}a_{21})>0$}
\end{tabular}
 &
    \parbox{2cm}{\vspace{2pt}\includegraphics[width=1.6cm,height=1.4cm]{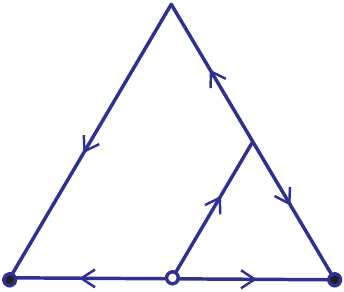}\vspace{2pt}} \\
&&\\
    17 &
\begin{tabular}{l}
 {\tiny(i)~ $a_{11}<a_{21}, a_{11}<a_{31}, a_{22}<a_{12}, a_{22}>a_{32}, a_{33}<a_{13}, a_{33}>a_{23}$}\\
{\tiny(ii)~ $a_{12}(a_{33}-a_{23})+a_{13}(a_{22}-a_{32})-(a_{22}a_{33}-a_{23}a_{32})>0$}\\
{\tiny(iii)~$a_{21}(a_{33}-a_{13})+a_{23}(a_{11}-a_{31})-(a_{11}a_{33}-a_{13}a_{31})<0$}\\
{\tiny(iv)~ $a_{31}(a_{22}-a_{12})+a_{32}(a_{11}-a_{21})-(a_{11}a_{22}-a_{12}a_{21})>0$}
\end{tabular}
 &
    \parbox{2cm}{\vspace{2pt}\includegraphics[width=1.6cm,height=1.4cm]{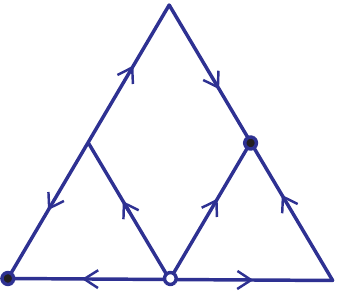}\vspace{2pt}} \\
&&\\
    18 &
\begin{tabular}{l}
 {\tiny(i)~ $a_{11}<a_{21}, a_{11}<a_{31}, a_{22}<a_{12}, a_{22}<a_{32}, a_{33}<a_{13}, a_{33}<a_{23}$}\\
{\tiny(ii)~ $a_{12}(a_{33}-a_{23})+a_{13}(a_{22}-a_{32})-(a_{22}a_{33}-a_{23}a_{32})<0$}\\
{\tiny(iii)~$a_{21}(a_{33}-a_{13})+a_{23}(a_{11}-a_{31})-(a_{11}a_{33}-a_{13}a_{31})<0$}\\
{\tiny(iv)~ $a_{31}(a_{22}-a_{12})+a_{32}(a_{11}-a_{21})-(a_{11}a_{22}-a_{12}a_{21})>0$}
\end{tabular}
 &
    \parbox{2cm}{\vspace{2pt}\includegraphics[width=1.6cm,height=1.4cm]{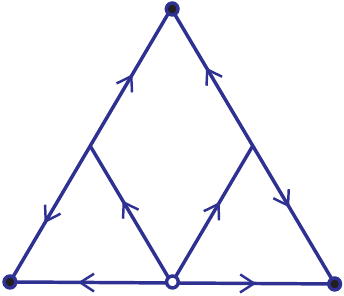}\vspace{2pt}}\\
&&\\
19 &
\begin{tabular}{l}
 {\tiny(i)~ $a_{11}>a_{21}, a_{11}>a_{31}, a_{22}<a_{12}, a_{22}<a_{32}, a_{33}<a_{13}, a_{33}<a_{23}$}\\
{\tiny(ii)~ $a_{12}(a_{33}-a_{23})+a_{13}(a_{22}-a_{32})-(a_{22}a_{33}-a_{23}a_{32})>0$}
\end{tabular}
 &
    \parbox{2cm}{\vspace{2pt}\includegraphics[width=1.6cm,height=1.4cm]{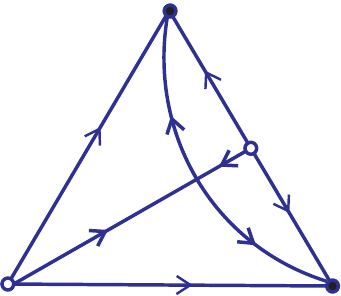}\vspace{2pt}} \\
&&\\
    20 &
\begin{tabular}{l}
 {\tiny(i)~ $a_{11}<a_{21}, a_{11}<a_{31}, a_{22}<a_{12}, a_{22}<a_{32}, a_{33}>a_{13}, a_{33}<a_{23}$}\\
{\tiny(ii)~ $a_{12}(a_{33}-a_{23})+a_{13}(a_{22}-a_{32})-(a_{22}a_{33}-a_{23}a_{32})>0$}\\
{\tiny(iii)~$a_{31}(a_{22}-a_{12})+a_{32}(a_{11}-a_{21})-(a_{11}a_{22}-a_{12}a_{21})>0$}
\end{tabular}
 &
    \parbox{2cm}{\vspace{2pt}\includegraphics[width=1.6cm,height=1.4cm]{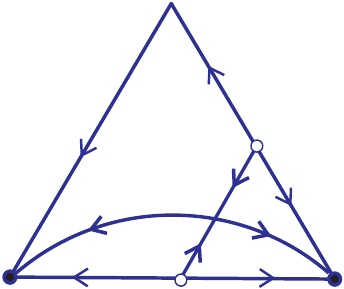}\vspace{2pt}} \\
&&\\
21 &
\begin{tabular}{l}
 {\tiny(i)~ $a_{11}<a_{21}, a_{11}<a_{31}, a_{22}<a_{12}, a_{22}>a_{32}, a_{33}<a_{13}, a_{33}>a_{23}$}\\
{\tiny(ii)~ $a_{12}(a_{33}-a_{23})+a_{13}(a_{22}-a_{32})-(a_{22}a_{33}-a_{23}a_{32})>0$}\\
{\tiny(iii)~$a_{21}(a_{33}-a_{13})+a_{23}(a_{11}-a_{31})-(a_{11}a_{33}-a_{13}a_{31})>0$}\\
{\tiny(iv)~ $a_{31}(a_{22}-a_{12})+a_{32}(a_{11}-a_{21})-(a_{11}a_{22}-a_{12}a_{21})>0$}
\end{tabular}
 &
    \parbox{2cm}{\vspace{2pt}\includegraphics[width=1.6cm,height=1.4cm]{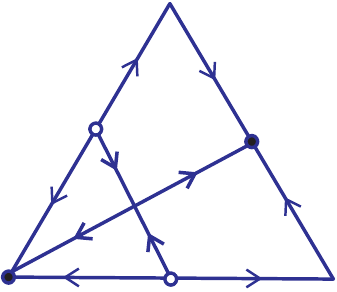}\vspace{2pt}} \\
&&\\
    22 &
\begin{tabular}{l}
 {\tiny(i)~ $a_{11}>a_{21}, a_{11}>a_{31}, a_{22}<a_{12}, a_{22}<a_{32}, a_{33}>a_{13}, a_{33}<a_{23}$}\\
{\tiny(ii)~ $a_{12}(a_{33}-a_{23})+a_{13}(a_{22}-a_{32})-(a_{22}a_{33}-a_{23}a_{32})>0$}\\
{\tiny(iii)~$a_{21}(a_{33}-a_{13})+a_{23}(a_{11}-a_{31})-(a_{11}a_{33}-a_{13}a_{31})>0$}
\end{tabular}
 &
    \parbox{2cm}{\vspace{2pt}\includegraphics[width=1.6cm,height=1.4cm]{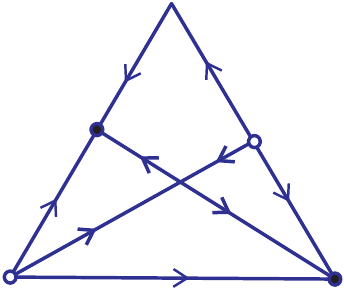}\vspace{2pt}} \\
&&\\
23 &
\begin{tabular}{l}
 {\tiny(i)~ $a_{11}>a_{21}, a_{11}>a_{31}, a_{22}>a_{12}, a_{22}>a_{32}, a_{33}<a_{13}, a_{33}<a_{23}$}\\
{\tiny(ii)~ $a_{31}(a_{22}-a_{12})+a_{32}(a_{11}-a_{21})-(a_{11}a_{22}-a_{12}a_{21})>0$}
\end{tabular}
 &
    \parbox{2cm}{\vspace{2pt}\includegraphics[width=1.6cm,height=1.4cm]{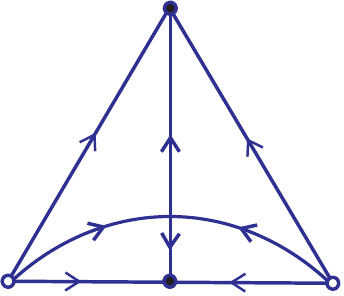}\vspace{2pt}} \\
&&\\
    24 &
\begin{tabular}{l}
 {\tiny(i)~ $a_{11}>a_{21}, a_{11}>a_{31}, a_{22}>a_{12}, a_{22}>a_{32}, a_{33}<a_{13}, a_{33}>a_{23}$}\\
{\tiny(ii)~ $a_{12}(a_{33}-a_{23})+a_{13}(a_{22}-a_{32})-(a_{22}a_{33}-a_{23}a_{32})>0$}\\
{\tiny(iii)~$a_{31}(a_{22}-a_{12})+a_{32}(a_{11}-a_{21})-(a_{11}a_{22}-a_{12}a_{21})>0$}
\end{tabular}
&
    \parbox{2cm}{\vspace{2pt}\includegraphics[width=1.6cm,height=1.4cm]{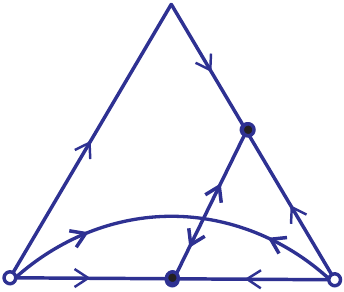}\vspace{2pt}} \\
&&\\
25 &
\begin{tabular}{l}
 {\tiny(i)~ $a_{11}>a_{21}, a_{11}>a_{31}, a_{22}>a_{12}, a_{22}<a_{32}, a_{33}>a_{13}, a_{33}<a_{23}$}\\
{\tiny(ii)~ $a_{12}(a_{33}-a_{23})+a_{13}(a_{22}-a_{32})-(a_{22}a_{33}-a_{23}a_{32})>0$}\\
{\tiny(iii)~$a_{21}(a_{33}-a_{13})+a_{23}(a_{11}-a_{31})-(a_{11}a_{33}-a_{13}a_{31})>0$}\\
{\tiny(iv)~ $a_{31}(a_{22}-a_{12})+a_{32}(a_{11}-a_{21})-(a_{11}a_{22}-a_{12}a_{21})>0$}
\end{tabular}
 &
    \parbox{2cm}{\vspace{2pt}\includegraphics[width=1.6cm,height=1.4cm]{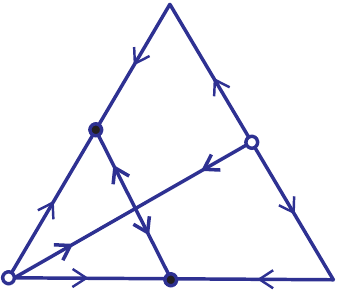}\vspace{2pt}} \\
&&\\
    26~a) &
\begin{tabular}{l}
 {\tiny(i)~ $a_{11}>a_{21}, a_{11}>a_{31}, a_{22}<a_{12}, a_{22}<a_{32}, a_{33}>a_{13}, a_{33}<a_{23}$}\\
{\tiny(ii)~ $a_{12}(a_{33}-a_{23})+a_{13}(a_{22}-a_{32})-(a_{22}a_{33}-a_{23}a_{32})<0$}\\
{\tiny(iii)~$a_{21}(a_{33}-a_{13})+a_{23}(a_{11}-a_{31})-(a_{11}a_{33}-a_{13}a_{31})<0$}\\
{\tiny(a)~ $\theta<0$}
\end{tabular}
 &
    \parbox{2cm}{\vspace{2pt}\includegraphics[width=1.6cm,height=1.4cm]{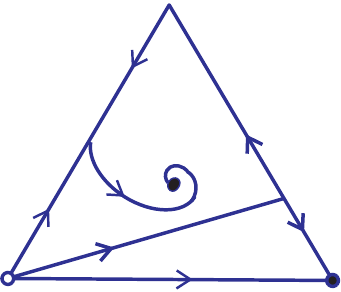}\vspace{2pt}} \\
&&\\
26~b) &
\begin{tabular}{l}
 {\tiny(i)~ $a_{11}>a_{21}, a_{11}>a_{31}, a_{22}<a_{12}, a_{22}<a_{32}, a_{33}>a_{13}, a_{33}<a_{23}$}\\
{\tiny(ii)~ $a_{12}(a_{33}-a_{23})+a_{13}(a_{22}-a_{32})-(a_{22}a_{33}-a_{23}a_{32})<0$}\\
{\tiny(iii)~$a_{21}(a_{33}-a_{13})+a_{23}(a_{11}-a_{31})-(a_{11}a_{33}-a_{13}a_{31})<0$}\\
{\tiny(b)~ $\theta=0$}
\end{tabular}
 &
    \parbox{2cm}{\vspace{2pt}\includegraphics[width=1.6cm,height=1.4cm]{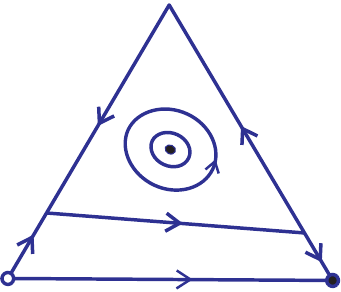}\vspace{2pt}} \\
&&\\
26~c) &
\begin{tabular}{l}
 {\tiny(i)~ $a_{11}>a_{21}, a_{11}>a_{31}, a_{22}<a_{12}, a_{22}<a_{32}, a_{33}>a_{13}, a_{33}<a_{23}$}\\
{\tiny(ii)~ $a_{12}(a_{33}-a_{23})+a_{13}(a_{22}-a_{32})-(a_{22}a_{33}-a_{23}a_{32})<0$}\\
{\tiny(iii)~$a_{21}(a_{33}-a_{13})+a_{23}(a_{11}-a_{31})-(a_{11}a_{33}-a_{13}a_{31})<0$}\\
{\tiny(c)~ $\theta>0$}
\end{tabular}
 &
    \parbox{2cm}{\vspace{2pt}\includegraphics[width=1.6cm,height=1.4cm]{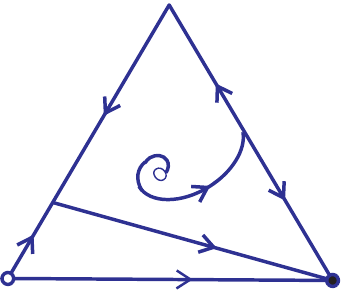}\vspace{2pt}} \\
&&\\
27~a) &
\begin{tabular}{l}
 {\tiny(i)~ $a_{11}>a_{21}, a_{11}<a_{31}, a_{22}<a_{12}, a_{22}>a_{32}, a_{33}>a_{13}, a_{33}<a_{23}$}\\
 {\tiny(a)~ $\theta<0$}
\end{tabular}
 &
    \parbox{2cm}{\vspace{2pt}\includegraphics[width=1.6cm,height=1.4cm]{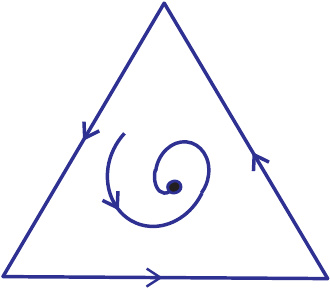}\vspace{2pt}} \\
&&\\
27~b) &
\begin{tabular}{l}
 {\tiny(i)~ $a_{11}>a_{21}, a_{11}<a_{31}, a_{22}<a_{12}, a_{22}>a_{32}, a_{33}>a_{13}, a_{33}<a_{23}$}\\
 {\tiny(b)~ $\theta=0$}
\end{tabular}
 &
    \parbox{2cm}{\vspace{2pt}\includegraphics[width=1.6cm,height=1.4cm]{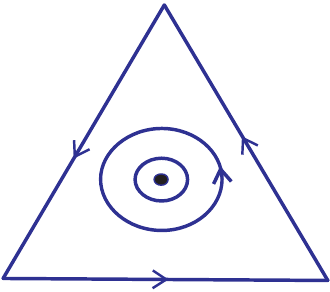}\vspace{2pt}} \\
&&\\
27~c) &
\begin{tabular}{l}
 {\tiny(i)~ $a_{11}>a_{21}, a_{11}<a_{31}, a_{22}<a_{12}, a_{22}>a_{32}, a_{33}>a_{13}, a_{33}<a_{23}$}\\
 {\tiny(c)~ $\theta>0$}
\end{tabular}
 &
    \parbox{2cm}{\vspace{2pt}\includegraphics[width=1.6cm,height=1.4cm]{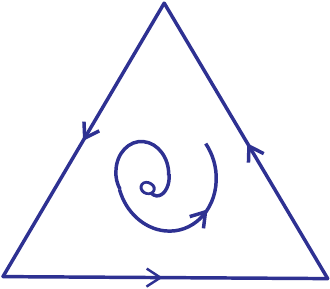}\vspace{2pt}} \\
&&\\
    28 &
\begin{tabular}{l}
 {\tiny(i)~ $a_{11}<a_{21}, a_{11}<a_{31}, a_{22}<a_{12}, a_{22}>a_{32}, a_{33}>a_{13}, a_{33}<a_{23}$}\\
{\tiny(ii)~ $a_{31}(a_{22}-a_{12})+a_{32}(a_{11}-a_{21})-(a_{11}a_{22}-a_{12}a_{21})<0$}
\end{tabular}
&
    \parbox{2cm}{\vspace{2pt}\includegraphics[width=1.6cm,height=1.4cm]{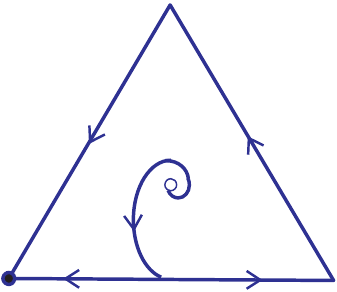}\vspace{2pt}} \\
&&\\
29 &
\begin{tabular}{l}
 {\tiny(i)~ $a_{11}>a_{21}, a_{11}>a_{31}, a_{22}>a_{12}, a_{22}<a_{32}, a_{33}<a_{13}, a_{33}>a_{23}$}\\
{\tiny(ii)~ $a_{31}(a_{22}-a_{12})+a_{32}(a_{11}-a_{21})-(a_{11}a_{22}-a_{12}a_{21})<0$}
\end{tabular}
 &
    \parbox{2cm}{\vspace{2pt}\includegraphics[width=1.6cm,height=1.4cm]{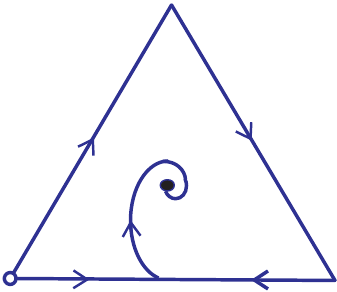}\vspace{2pt}} \\
&&\\
    30 &
\begin{tabular}{l}
{\tiny(i)~ $a_{11}<a_{21}, a_{11}<a_{31}, a_{22}<a_{12}, a_{22}<a_{32}, a_{33}>a_{13}, a_{33}<a_{23}$}\\
{\tiny(ii)~ $a_{12}(a_{33}-a_{23})+a_{13}(a_{22}-a_{32})-(a_{22}a_{33}-a_{23}a_{32})<0$}\\
{\tiny(iii)~$a_{31}(a_{22}-a_{12})+a_{32}(a_{11}-a_{21})-(a_{11}a_{22}-a_{12}a_{21})<0$}
\end{tabular}
&
    \parbox{2cm}{\vspace{2pt}\includegraphics[width=1.6cm,height=1.4cm]{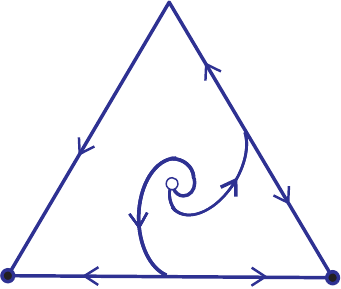}\vspace{2pt}} \\
&&\\
    31 &
\begin{tabular}{l}
{\tiny(i)~ $a_{11}>a_{21}, a_{11}>a_{31}, a_{22}>a_{12}, a_{22}>a_{32}, a_{33}<a_{13}, a_{33}>a_{23}$}\\
{\tiny(ii)~$a_{12}(a_{33}-a_{23})+a_{13}(a_{22}-a_{32})-(a_{22}a_{33}-a_{23}a_{32})<0$}\\
{\tiny(iii)~$a_{31}(a_{22}-a_{12})+a_{32}(a_{11}-a_{21})-(a_{11}a_{22}-a_{12}a_{21})<0$}
\end{tabular}
&
    \parbox{2cm}{\vspace{2pt}\includegraphics[width=1.6cm,height=1.4cm]{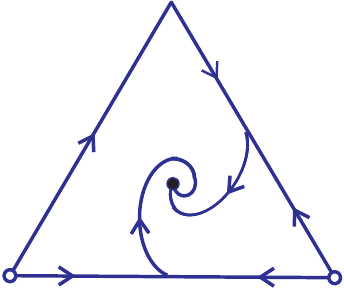}\vspace{2pt}} \\
&&\\
32 &
\begin{tabular}{l}
 {\tiny(i)~ $a_{11}<a_{21}, a_{11}<a_{31}, a_{22}<a_{12}, a_{22}<a_{32}, a_{33}<a_{13}, a_{33}<a_{23}$}\\
{\tiny(ii)~ $a_{12}(a_{33}-a_{23})+a_{13}(a_{22}-a_{32})-(a_{22}a_{33}-a_{23}a_{32})<0$}\\
{\tiny(iii)~$a_{21}(a_{33}-a_{13})+a_{23}(a_{11}-a_{31})-(a_{11}a_{33}-a_{13}a_{31})<0$}\\
{\tiny(iv)~ $a_{31}(a_{22}-a_{12})+a_{32}(a_{11}-a_{21})-(a_{11}a_{22}-a_{12}a_{21})<0$}
\end{tabular}
 &
    \parbox{2cm}{\vspace{2pt}\includegraphics[width=1.6cm,height=1.4cm]{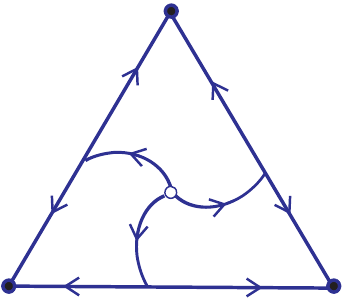}\vspace{2pt}} \\
&&\\
    33 &
\begin{tabular}{l}
 {\tiny(i)~ $a_{11}>a_{21}, a_{11}>a_{31}, a_{22}>a_{12}, a_{22}>a_{32}, a_{33}>a_{13}, a_{33}>a_{23}$}\\
{\tiny(ii)~ $a_{12}(a_{33}-a_{23})+a_{13}(a_{22}-a_{32})-(a_{22}a_{33}-a_{23}a_{32})<0$}\\
{\tiny(iii)~$a_{21}(a_{33}-a_{13})+a_{23}(a_{11}-a_{31})-(a_{11}a_{33}-a_{13}a_{31})<0$}\\
{\tiny(iv)~ $a_{31}(a_{22}-a_{12})+a_{32}(a_{11}-a_{21})-(a_{11}a_{22}-a_{12}a_{21})<0$}
\end{tabular}
&
    \parbox{2cm}{\vspace{2pt}\includegraphics[width=1.6cm,height=1.4cm]{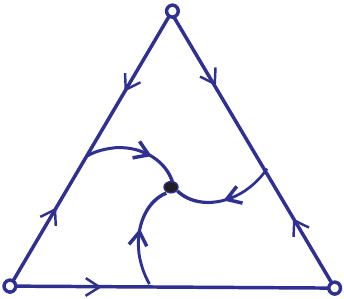}\vspace{2pt}}\\[-8pt]
\label{biao0}
\end{longtable}
\end{center}

\newpage
\section{Appendix B. Turbulent Characteristics: Nonuniqueness and Nonergodicity in Limit for the Time Average Probability Measures}
The time average of transition probability function for each solution weakly converges to an ergodic stationary measure for (\ref{3DSLV}) on the attracting domain of the omega limit set of the orbit for the (\ref{3DLV}) through the same initial point in all classes except class 27 c). But class 27 c) is quite different. If $y\in {\rm Int}\mathbf{R}_+^3\backslash L(P)$, the corresponding time average of transition probability function has infinite weak limit points, which are not ergodic. We will reveal that the essential reason for both peculiar characteristics is that solutions concentrate around $R_1, R_2, R_3$ very long time (approximately infinite) with probability nearly one.

Theorem \ref{produce} tells us that nontrivial stationary measures are produced by the solutions through points in $\Sigma$. So we fix $y\in \Sigma$ with $y\neq P$. In order to prove these by specific estimations, we will consider the symmetric May-Leonard system (\ref{sys1}) with $\alpha=0.8$ and $\beta=1.3$. The other cases are similar.

Firstly, we will prove that limit point of the family
$\{\frac{1}{T}\int_{0}^{T}\delta_{\Psi(t,y)}(\cdot)dt\}_{T>0}$ as $T\rightarrow\infty$,
which is the stationary measure for deterministic system (\ref{sys1}), is not unique.
That is, the weak limit of
\begin{equation}\label{nolim}
  \frac{1}{T}\int_{0}^{T}\delta_{\Psi(t,y)}(\cdot)dt  \stackrel{w}{\rightarrow} \mu(\cdot)
\end{equation}
is not unique.

Let
$$A_{i}=\{y=(y_{1},y_{2},y_{3})\in\Sigma:\|y-R_{i}\|<\frac{1}{2}\}$$
denote the neighborhood of $R_i$ ($i=1, 2, 3$). Then $\Psi(t,y)$ will be spirally asymptotic to $\mathcal{H}$ as the time goes to infinity. Hence, $\Psi(t,y)$ will enter and then depart $A_{i}$ with infinite times.

For $n\geq 2$, define
$$
\begin{array}{rlrl}
  T_{{\rm in}}^1= & \inf\{t\geq0,\ \Psi(t,y)\in A_{1}\},\quad &T_{{\rm out}}^1= & \inf\{t\geq T_{{\rm in}}^1,\ \Psi(t,y)\notin A_{1}\},\\
  [3pt]
  T_{{\rm in}}^n=&\inf\{t\geq T_{{\rm out}}^{n-1},\ \Psi(t,y)\in A_{1}\}, \quad &T_{{\rm out}}^n=&\inf\{t\geq T_{{\rm in}}^n,\ \Psi(t,y)\notin A_{1}\},\\
  [3pt]
  S_{{\rm in}}^1=&\inf\{t\geq T_{{\rm out}}^1,\ \Psi(t,y)\in A_{3}\},\quad &S_{{\rm out}}^1=&\inf\{t\geq S_{{\rm in}}^1,\ \Psi(t,y)\notin A_{3}\},\\
  [3pt]
  S_{{\rm in}}^n=&\inf\{t\geq S_{{\rm out}}^{n-1},\ \Psi(t,y)\in A_{3}\}, \quad &S_{{\rm out}}^n=&\inf\{t\geq S_{{\rm in}}^n,\ \Psi(t,y)\notin A_{3}\}.
\end{array}
$$
Similarly, we denote by $\tau_{{\rm in}}^n$ and $\tau_{{\rm out}}^n$ the time entering and exiting $A_2$ in $n$-th spiral cycle (see Fig. \ref{firsttime}).
\begin{figure}[ht]
 \begin{center}
 \includegraphics[width=0.65\textwidth]{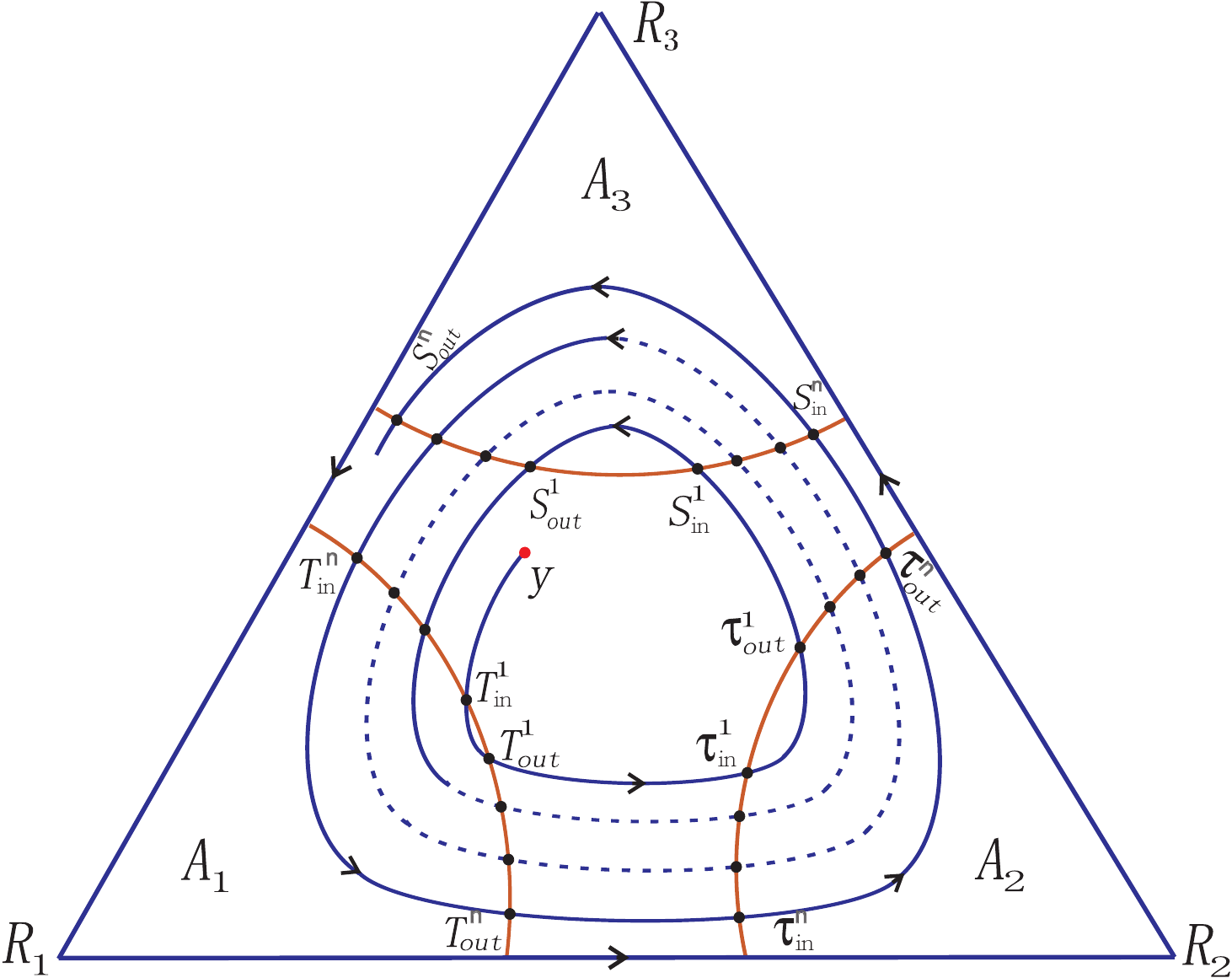}
  \end{center}
\caption{} \label{firsttime}
\end{figure}
By the continuity of $\Psi$, $\tau_{{\rm in}}^n-T_{{\rm out}}^n$, $S_{{\rm in}}^n-\tau_{{\rm out}}^n$ and $T_{{\rm in}}^{n+1}-S_{{\rm out}}^n$ are approximately constants independent of $n$. May and Leonard \cite{May1} showed that the time spent in the neighborhood of $R_i$ is proportional to the total time elapsed up to that stage $t$. In this example, they gave the following estimation:
\begin{equation}\label{estimation}
T_{{\rm out}}^n-T_{{\rm in}}^n\simeq 0.42T_{{\rm out}}^n,\ \tau_{{\rm out}}^n-\tau_{{\rm in}}^n\simeq 0.42\tau_{{\rm out}}^n,\ S_{{\rm out}}^n-S_{{\rm in}}^n\simeq 0.42S_{{\rm out}}^n.
\end{equation}

Choosing two subsequences $\{T_{{\rm out}}^n\}$ and $\{S_{{\rm out}}^n\}$, therefore for sufficiently large $n$, we have
\begin{equation}\label{cruesti1}
   \frac{1}{T_{{\rm out}}^n}\int_{0}^{T_{{\rm out}}^n}\delta_{\Psi(t,y)}(A_{1})dt
= \frac{1}{T_{{\rm out}}^n}\displaystyle\sum_{i=1}^{n}(T_{{\rm out}}^i-T_{{\rm in}}^i) \geq \frac{T_{{\rm out}}^n-T_{{\rm in}}^n}{T_{{\rm out}}^n}=0.42>0,
\end{equation}
\begin{equation}\label{cruesti2}
  \frac{1}{S_{{\rm out}}^n}\int_{0}^{S_{{\rm out}}^n}\delta_{\Psi(t,y)}(A_{1})dt
= \frac{1}{S_{{\rm out}}^n}\displaystyle\sum_{i=1}^{n}(T_{{\rm out}}^i-T_{{\rm in}}^i)
\leq \frac{T_{{\rm out}}^n}{S_{{\rm out}}^n}\leq (0.58)^2\leq 0.34.
\end{equation}
Here we have used the property that $S_{{\rm in}}^n-T_{{\rm out}}^n\simeq 0.42 S_{{\rm in}}^n$, which holds from (\ref{estimation}) and the continuity of $\Psi$.
From (\ref{cruesti1}), (\ref{cruesti2}) and Proposition \ref{PRe}, it easily shows that the limit of (\ref{nolim})
is not unique.

Subsequently, we consider stochastic case (\ref{3DSLV}) where the competitive coefficients are given in symmetric May-Leonard system (\ref{sys1}) with $\alpha=0.8$ and $\beta=1.3$. We will analyze the limit as $T\rightarrow \infty$ for $\Big\{\frac{1}{T}\int_0^{T}I_{A_1}\Big(\Psi(\int_0^tg(s,\omega,1)ds,y)\Big)dt\Big\}_{T>0}$.

Let $\epsilon=0.0001$ and $\Omega_T^\epsilon=\{\omega: \sup_{t\in[T,\infty)}|\frac{1}{t}\int_0^tg(s,\omega,1)ds-1|\leq \epsilon\}$. Then
$\Omega_T^\epsilon\uparrow$ with respect to $T$ and $\lim_{T\rightarrow\infty}\mathbb{P}(\Omega_T^\epsilon)=1$. Thus for $\eta=0.9999$, there exists $T_0>0$ such that
$$
\mathbb{P}(\Omega_T^\epsilon)\geq \eta,\ \ \forall T\geq T_0.
$$

Define $t^n_1(\omega):=\tau(\omega,T_{\rm in}^n)$ and $t^n_2(\omega):= \tau(\omega,T_{\rm out}^n)$ as given in (\ref{stopping}).
 Set $\Omega_{T_0}^n:=\{\omega: t^n_1(\omega)\geq T_0\}$. Then $\Omega_{T_0}^n\uparrow$ with respect to $n$ and $\lim_{n\rightarrow\infty}\mathbb{P}(\Omega_{T_0}^n)=1.$ Thus there exists an $N_0$ such that
$$
\mathbb{P}(\Omega_{T_0}^n)\geq\eta,\ \ \forall n\geq N_0.
$$

\textbf{Step 1.} Let $T_n=T_{\rm out}^n$. We analyze $\frac{1}{T_n}\int_0^{T_n}I_{A_1}\Big(\Psi(\int_0^tg(s,\omega,1)ds,y)\Big)dt$.

For any $n$ satisfying $n\geq N_0$ and $T_n\geq T_0$, choosing  any $\omega\in \Omega_{T_0}^n\cap \Omega^{\epsilon}_{T_0}$, we have
\begin{itemize}
\item[$\bullet$] $(1-\epsilon)T_{\rm out}^n=(1-\epsilon)T_n\leq \int_0^{T_n}g(s,\omega,1)ds\leq(1+\epsilon)T_n =(1+\epsilon)T_{\rm out}^n$,

\item[$\bullet$] $t^n_2(\omega)\geq t^n_1(\omega)\geq T_0$,

\item[$\bullet$] $(1-\epsilon)t^n_1(\omega)\leq \int_0^{t^n_1(\omega)}g(s,\omega,1)ds=T_{\rm in}^n\leq(1+\epsilon)t^n_1(\omega)$,

\item[$\bullet$] $(1-\epsilon)t^n_2(\omega)\leq \int_0^{t^n_2(\omega)}g(s,\omega,1)ds=T_{\rm out}^n\leq(1+\epsilon)t^n_2(\omega)$.
\end{itemize}
Combining the fact that $T_{\rm out}^n-T_{\rm in}^n \simeq 0.42 T_{\rm out}^n$, we have
$$
t^n_1(\omega)\leq T_{\rm out}^n=T_n,\ \ t^n_2(\omega)\geq \frac{T_{\rm out}^n}{1+\epsilon}=\frac{T_n}{1+\epsilon},\ \ \frac{T_{\rm in}^n}{1-\epsilon}\geq t^n_1(\omega).
$$
Hence
\begin{eqnarray*}
&&\frac{1}{T_n}\int_0^{T_n}I_{A_1}\Big(\Psi(\int_0^tg(s,\omega,1)ds,y)\Big)dt\\
&=&
\frac{1}{T_n}\sum_{i=1}^\infty \Big(t_2^i(\omega)\bigwedge T_n-t_1^i(\omega)\bigwedge T_n\Big)\\
&\geq& \frac{t_2^n(\omega)\bigwedge S_n-t_1^n(\omega)}{S_n}\\
&\geq& \frac{\frac{T_{\rm out}^n}{1+\epsilon}-\frac{T_{\rm in}^n}{1-\epsilon}}{T_{\rm out}^n}\\
&\geq& 0.419.
\end{eqnarray*}
Then
\begin{eqnarray}\label{eq 01}
\underline{\lim}_{n\rightarrow\infty}\frac{1}{T_n}\int_0^{T_n}\mathbb{E}I_{A_1}\Big(\Psi(\int_0^tg(s,\omega,1)ds,y)\Big)dt
\geq 0.419\mathbb{P}(\Omega_{T_0}^{N_0}\cap \Omega^{\epsilon}_{T_0})
&\geq& 0.419\times 0.9998\nonumber\\
 &\geq& 0.41.
\end{eqnarray}

\textbf{Step 2.} Let $S_n=S_{\rm out}^n$. We analyze $\frac{1}{S_n}\int_0^{S_n}I_{A_1}\Big(\Psi(\int_0^tg(s,\omega,1),y)\Big)dt$.

For any $n$ satisfying $n\geq N_0$ and $S_n\geq T_0$, choosing any $\omega\in \Omega_{T_0}^n\cap \Omega^{\epsilon}_{T_0}$, we have
\begin{itemize}
\item[$\bullet$] $(1-\epsilon)S_{\rm out}^n=(1-\epsilon)S_n\leq \int_0^{S_n}g(s,\omega,1)ds\leq(1+\epsilon)S_n =(1+\epsilon)S_{\rm out}^n$,

\item[$\bullet$] $t^{n+1}_2(\omega)\geq t^{n+1}_1(\omega)\geq t^n_2(\omega)\geq t^n_1(\omega)\geq T_0$,

\item[$\bullet$] $(1-\epsilon)t^i_1(\omega)\leq \int_0^{t^i_1(\omega)}g(s,\omega,1)ds=T_{\rm in}^i\leq(1+\epsilon)t^i_1(\omega),\ \ \ i=n,n+1$,

\item[$\bullet$] $(1-\epsilon)t^i_2(\omega)\leq \int_0^{t^i_2(\omega)}g(s,\omega,1)ds=T_{\rm out}^i\leq(1+\epsilon)t^i_2(\omega),\ \ \ i=n,n+1$,

\item[$\bullet$] $T_{\rm in}^{n+1}\simeq S^n_{\rm out}$, $S^n_{\rm out}-S^n_{\rm in}\simeq0.42 S^n_{\rm out}$, $S^n_{\rm in}-T_{\rm out}^n\simeq 0.42 S^n_{\rm in}$.
\end{itemize}

Hence,
\begin{itemize}
\item[$\bullet$] $(1-\epsilon)t^n_2(\omega)\leq T_{\rm out}^n\simeq 0.58 S^n_{\rm in}\simeq 0.58^2 S^n_{\rm out}=0.58^2 S_n\Rightarrow t^n_2(\omega)\leq S_n$,

\item[$\bullet$] \begin{eqnarray*}
       && (1-\epsilon)t^{n+1}_1(\omega)\leq T_{\rm in}^{n+1}\simeq S^n_{\rm out}=S_n\\
       &\leq&
       (1+\epsilon)t^{n+1}_1(\omega)\leq \frac{1+\epsilon}{1-\epsilon} T_{\rm in}^{n+1}\simeq 0.58\frac{1+\epsilon}{1-\epsilon} T_{\rm out}^{n+1}\\
       &\leq&
       0.58\frac{(1+\epsilon)^2}{1-\epsilon}t^{n+1}_2(\omega)
       <t^{n+1}_2(\omega),
       \end{eqnarray*}
       that is,
       $$
       (1-\epsilon)t^{n+1}_1(\omega)\leq S_n
       \leq
       (1+\epsilon)t^{n+1}_1(\omega)<t^{n+1}_2(\omega).
       $$
\end{itemize}

Hence
\begin{eqnarray*}
&&\frac{1}{S_n}\int_0^{S_n}I_{A_1}\Big(\Psi(\int_0^tg(s,\omega,1)ds,y)\Big)dt\\
&=&
\frac{1}{S_n}\sum_{i=1}^\infty \Big(t_2^i(\omega)\bigwedge S_n-t_1^i(\omega)\bigwedge S_n\Big)\\
&=&
\frac{1}{S_n}\Big[\sum_{i=1}^{n} \Big(t_2^i(\omega)-t_1^i(\omega)\Big)
  +
  \Big(S_n-t_1^{n+1}(\omega)\bigwedge S_n\Big)\Big]\\
&\leq& \frac{t_2^n(\omega)+ S_n-t_1^{n+1}(\omega)\bigwedge S_n}{S_n}\\
&\leq& \frac{1}{S_n}\Big(\frac{0.58^2}{1-\epsilon}S_n+S_n-\frac{S_n}{1+\epsilon}\Big)\\
&=& \frac{0.58^2}{1-\epsilon}+\frac{\epsilon}{1+\epsilon}<0.34.
\end{eqnarray*}

Then
\begin{eqnarray}\label{eq 02}
&&\overline{\lim}_{n\rightarrow\infty}\frac{1}{S_n}\int_0^{S_n}\mathbb{E}I_{A_1}\Big(\Psi(\int_0^tg(s,\omega,1)ds,y)\Big)dt\nonumber\\
&\leq& 0.34\mathbb{P}(\Omega_{T_0}^{N_0}\cap \Omega^{\epsilon}_{T_0})+\mathbb{P}\Big[(\Omega_{T_0}^{N_0}\cap \Omega^{\epsilon}_{T_0})^c\Big]\nonumber\\
&\leq& 0.342.
\end{eqnarray}
(\ref{eq 01}) and (\ref{eq 02}) imply that $\frac{1}{T}\int_0^{T}\mathbb{E}I_{A_1}\Big(\Psi(\int_0^tg(s,\omega,1),y)\Big)dt$ does not have unique limit as $T\rightarrow\infty$. Equivalently, $\frac{1}{T}\int_0^{T}\mathbb{E}I_{\Lambda(A_1)}\Big(\Phi(t,\omega,,y)\Big)dt$ does not have unique limit as $T\rightarrow\infty$.
%
 \bibliographystyle{spmpsci}
 \bibliography{refs}
%

\end{document}